\documentclass[11pt,a4paper,twoside]{article}
\usepackage{amssymb}
\usepackage{amsfonts}
\usepackage{geometry}
\usepackage{color}
\usepackage{amsmath,amsthm,amssymb,amsfonts,enumitem}
\usepackage{CJK}
\usepackage{latexsym}
\usepackage{graphicx}
\usepackage{pgfpages}
\usepackage{mathrsfs}
\usepackage{ytableau}
\usepackage[title,titletoc]{appendix}
\usepackage{enumerate}
\usepackage{cite}
\usepackage[hyphens]{url}
\usepackage[bookmarks=false,hidelinks]{hyperref}
\usepackage[affil-it]{authblk}
\allowdisplaybreaks

\newtheorem{definition}{Definition}[section]
\newtheorem{theorem}{Theorem}[section]
\newtheorem{lemma}{Lemma}[section]
\newtheorem{proposition}{Proposition}[section]

\newtheorem{remark}{Remark}[section]
\newtheorem{condition}{Assumption}

\newtheorem{example}{Example}[section]

\numberwithin{equation}{section}

\begin{document}
	\title{Entrance measures and dynamics for time-inhomogeneous McKean-Vlasov stochastic differential equations
}
	\author[a,b]{
		Chunrong Feng}
	\author [b,a] {Baoyou Qu}
	\author[a,b]{Huaizhong Zhao}
	\affil[a]{Department of Mathematical Sciences, Durham University, DH1 3LE, UK}
	\affil[b] {Research Centre for Mathematics and Interdisciplinary Sciences, Shandong University, Qingdao 266237, China}
	
	\affil[ ]{chunrong.feng@durham.ac.uk, qubaoyou@sdu.edu.cn, huaizhong.zhao@durham.ac.uk}
	\date{}
	
\maketitle

\begin{abstract}

In this paper, we study the entrance measures of time-inhomogeneous McKean-Vlasov SDEs. The existence is obtained in great generality, where the system can be expanding globally and/or degenerate for numerous number of time intervals. When the parameters are periodic/quasi-periodic in time, we obtain the existence of periodic/asymptotic quasi-periodic measures. In this case, a double-lift of the random dynamical system first to a dynamical system on cylinder and then on the graph of reparameterized process living on the cylinder is introduced. The double-lifted system gives to a continuous dynamical system over probability measures on the cylinder, and the lifted multi-parameter measure of the asymptotic quasi-periodic measure can then lead to an invariant measure of the lifted semigroup.
	 
	\medskip
	
    \vskip5pt
    
\end{abstract}

\tableofcontents

\section{Introduction}

The propagation of chaos studies the evolution of interacting particles with a certain initial probability distribution $\mu_0$ for each particle independently. Due to interactions between particles, the probability distribution of the particle system evolves in a nonlinear way, but the symmetry of particles and the propagation of chaos suggest that the study of one individual gives the behaviour of the group of particles. A finite $N$-particle system can be regarded as the reduced or approximating system of evolutions of particles, e.g., as described by the Liouville equation and Boltzmannn equation for a dilute gas of hard spheres (Lanford \cite{Lanford1975}, Uchiyama \cite{Uchiyama1988}). Kac \cite{Kac1956} proposed to study the master equations as Markovian evolutions of $N$-particle systems. Sznitman in his work \cite{Sznitman1991} studied the evolution of probability distributions by investigating the interacting particle processes given by the  master equations to the solutions of McKean-Vlasov equations (McKean \cite{McKean1967}). This probabilistic approach makes studies of the convergence of the joint distributions of interacting $N$-particle processes and the propagation of chaos possible in the level of processes.

The law of the nonlinear process given by the solution of McKean-Vlasov equation, denoted by $\mu$, provides a rigorous mathematical description of the propagation of chaos. Let $\mu^N$ be a sequence of symmetric probabilities on $\mathbb{X}^N$, where $\mathbb{X}$ is the state space of individual particles. Then $\mu^N$ is called $\mu$-chaotic on $\mathbb{X}$, if for any $\varphi_1,\cdots,\varphi_k\in C_b(\mathbb{X})$, $k\geq 1$,
\begin{equation*}
	\lim_{N\to \infty}\langle \mu^N, \varphi_1\otimes\cdots\otimes \varphi_k\otimes 1\otimes\cdots\otimes 1 \rangle=\prod_{i=1}^{k}\langle \mu,\varphi_i\rangle.
\end{equation*}
In Sznitman \cite{Sznitman1991}, it was proved that the above convergence is equivalent to the law of large numbers with $\mu$ as its limit:
\begin{equation}
	\frac{1}{N}\sum_{i=1}^{N}\delta_{X_i}\to \mu.
\end{equation}

Thus, it is key to study evolutions of McKean-Vlasov processes and their probability laws due to their fundamental importance in the study of propagation of chaos. The large time limits are particularly interesting, as they provide invariance or equilibrium of the interacting particles in law. In this paper, we consider time-inhomogeneous McKean-Vlasov equation:
\begin{equation}\label{McKean-Vlasov SDE 2}
	\begin{cases}
		\mathrm{d}X_t=b(t,X_t,\mathscr{L}(X_t))\mathrm{d}t+\sigma(t,X_t,\mathscr{L}(X_t))\mathrm{d}W_t, \ t\geq s,\\
	  X_s=\xi, 
	\end{cases}
\end{equation}
where $\xi$ is an $\mathcal{F}_s:=\mathcal{F}_{-\infty}^s$-measurable random variable with $\mathscr{L}(\xi)=\nu$ for a given probability distribution $\nu$. The corresponding interacting $N$-particle system is instead
\begin{equation}\label{eq:N-SDE}
	\mathrm{d}X_t^i=b\biggl(t,X_t^i,\frac{1}{N}\sum_{j=1}^N\delta_{X_t^j}\biggr)\mathrm{d}t+\sigma\biggl(t,X_t^i,\frac{1}{N}\sum_{j=1}^N\delta_{X_t^j}\biggr)\mathrm{d}W_t^i, \ i=1,\cdots,N.
\end{equation}
Denote the solution of McKean-Vlasov equation \eqref{McKean-Vlasov SDE 2} by $X_t^{s,\nu}$. The limit distribution of \eqref{eq:N-SDE} is given by an entrance law $\mu: (-\infty,\infty)\to \mathcal{P}(\mathbb{X})$ satisfying
\begin{equation*}
	\mathscr{L}\bigl(X_t^{s,\mu_s}\bigr)=\mu_t  \ \text{ for all } \  t,s\in \mathbb{R}, t\geq s.
\end{equation*}
An invariant measure is a special case in the case of time independent McKean-Vlasov equation and $\mu_t=\mu_s$ for all $t,s\in \mathbb{R}$.

The invariant measures of the stationary McKean-Vlasov equation attracted many mathematicians' attention. The case with additive noise, where $\sigma$ is a constant matrix, $b(x,\mu)=\nabla V+\nabla (W\ast \mu)$, was considered by Carrillo-McCann-Villani \cite{Carrillo-McCann-Villani2003,Carrillo-McCann-Villani2006}, Liu-Wu-Zhang \cite{Liu-Wu-Zhang2021}, Guillin-Liu-Wu-Zhang \cite{Guillin-Liu-Wu-Zhang2021,Guillin-Liu-Wu-Zhang2022}, Ren-Wang \cite{Ren-Wang2021}, Wang \cite{wang2018}, in the study of granular media type equations and functional McKean-Vlasov equations by Bao-Scheutzow-Yuan \cite{Bao-Scheutzow-Yuan2022}. Other works on invariant measures in this direction include Butkovsky \cite{Butkovsky2014}, Muzychka-Vaninsky \cite{Muzychka-Vaninsky2011} etc. The Curie-Weiss model in the study of ferromagnetism is the case of double well potential that $b(x,\mu)=x-x^3+\nabla_x\int W(x,y)\mu(\mathrm{d}y)$. The phase transition was obtained in Carrillo-Gvalani-Pavliotis-Schlichting \cite{Carrillo-Gvalani-Pavliotis-Schlichting2020} and Delgadino-Gvalani-Pavliotis-Smith \cite{Delgadino-Gvalani-Pavliotis-Smith2023}. In all these works so far, the interacting particle systems or the limiting McKean-Vlasov equations are all time-homogeneous. However, the non-stationary or time-inhomogeneous case remains unsolved.

Non-stationary processes arise, e.g. when coefficients of stochastic differential equations vary in time. Periodic and quasi-periodic stochastic systems are two important special cases. They appear in mathematical modelling of many real world problems, such as the temperature variant, climate changes, sunspot activities, economic cycles, etc. 
The study of entrance measures including periodic measures and quasi-periodic measures and their ergodicity have recently received a lot of attention (Feng-Zhao-Zhong \cite{Feng-Zhao-Zhong2023}, Feng-Zhao \cite{Feng-Zhao2020}, Feng-Qu-Zhao \cite{Feng-Qu-Zhao2021,Feng-Qu-Zhao2023}). 
Non-stationary McKean-Vlasov equations arise in, e.g. studying the Ising model or Curie-Weiss model where interacting particles are essentially subject to time-inhomogeneous interactions or external forcing.

The first task of this paper is to investigate the existence of entrance measures of McKean-Vlasov equations. They serve as a vehicle towards further understanding the dynamics of the McKean-Vlasov processes. In particular, they give the statistical equilibrium of the system and the limiting distribution. In contrast to stationary problems, the statistical equilibrium is not static, but a flow of probability measures $s\mapsto \mu_s$ carried by the weak solution of the McKean-Vlasov equation. Intuitively speaking, $\{\mu_s\}_{s\in \mathbb{R}}$ is the flow of probability laws that one gets when ``letting the probability laws start to evolve at time $-\infty$''. 

For McKean-Vlasov equations, the Markov operator $P(t,s): {\mathcal B}_b({\mathbb R}^d)\to {\mathcal B}_b({\mathbb R}^d)$ defined by $(P(t,s)f)(x)=E[f(X_t^{s,x})]$ is no longer a semigroup (Wang \cite{wang2018}). This is due to the nonlinearity of the evolution of measures, which is still a semigroup and can be naturally driven by weak solutions, in contrast to the case of Markovian processes. These features create challenges for their studies. To tackle them, we first fix a measure-valued continuous function in  the McKean-Vlasov equation  to make the semigroup linear, and then apply the step by step `Harris small set' tools that we developed for SDE in \cite{Feng-Qu-Zhao2023} to study entrance measures of the SDEs depending on the pre-fixed distribution function. This gives a map from a given measure-valued function to a measure-valued function $\mu_{\cdot}\mapsto \Psi(\mu_{\cdot})$. In this article, a fixed point argument is developed to obtain the fixed point of $\Psi(\mu_{\cdot})$. It is noted that the fixed point method had already been used in the homogeneous setting in some previous works (see, e.g. Ahmed-Ding \cite{Ahmed-Ding1993}, Bao-Scheutzow-Yuan \cite{Bao-Scheutzow-Yuan2022}, Dawson \cite{Dawson1983}, Wang \cite{wang2018}, Zhang \cite{Zhang2023}). However, the entrance measure is a measure-valued function rather than a measure, so techniques for the invariant measure argument in the literature are not applicable in this situation. In this paper, a new fixed point argument is built in the space $C(\mathbb{R};\mathcal{P}_2(\mathbb{R}^d))$ to tackle the nonstationary problem.

It should be noted that Liu-Wu-Zhang \cite{Liu-Wu-Zhang2021} studied the existence and uniqueness of the invariant measure in the stationary case when the noise is strong. Our existence of entrance measure result does not require noise to be strong comparing with other parameters, e.g. inverse temperature in the Curie-Weiss model without any constraints on the parameters, as long as it is nondegenerate, we can prove the existence of entrance measures. In fact, the interaction between particles makes it harder for particles to move between different minima of the potentials, so synchronization may not necessarily occur and the uniqueness of entrance measure does not hold in general. As an example, we show that the Curie-Weiss model possesses multiple distinct entrance measures in the case of lower temperature. This is in contrast to SDEs with no mean-field interactions where particles can move between minima of potentials in the presence of nondegenerate noise, no matter how small it is, so that these SDEs have unique entrance measures (\!\!\cite{Feng-Zhao-Zhong2023,Feng-Qu-Zhao2023}). The nonuniqueness of entrance measure causes further difficulties when we consider the periodicity and quasi-periodicity of the entrance measure.

Another novelty of this paper is the result of periodic/asymptotic quasi-periodic measures for McKean-Vlasov equations with periodic/quasi-periodic coefficients. These measures are used to describe random periodicity/quasi-periodicity in the sense of distributions, i.e., repeated pattern of statistical laws. They are entrance measures satisfying periodic or asymptotic quasi-periodic conditions. The reparameterization observed in Feng-Qu-Zhao \cite{Feng-Qu-Zhao2021} made it possible to obtain a multi-parameter entrance measure for the reparameterized system. Note that the reparameterized system can be reduced to the original quasi-periodic system by letting the multi-parameter be zero. So, if both the entrance measures of the original system and the reparameterized system were unique, then it would be easy to claim that the multi-parameter entrance measure should be the desired multi-parameter representation. But the idea above is broken due to the lack of the uniqueness of (multi-parameter) entrance measures. For this reason, we introduce the $\mathcal{Q}_2$ (see \eqref{eq:def-Q_2}) space of pairs of quasi-periodic measures and its multi-parameter representation, and we carry out the Kakutani fixed point argument on  $\overline{\mathcal{Q}_2}$.  With this new idea, we obtain an asymptotic quasi-periodic measure, which is an entrance measure as the limit of a sequence of quasi-periodic measures. However, due to the lack of the uniqueness, whether or not any entrance measure of a quasi-periodic stochastic McKean-Vlasov system is an asymptotic quasi-periodic measure, i.e., the entrance measure as the limit of a sequence of multi-parameter entrance measures, is still unclear as a general result.

In the SDE setting, the multi-parameter measure can be extended to a measure on the cylinder over a torus, and the average of these measures yields an invariant measure for the lifted semigroup. However, the construction of an invariant measure on a lifted space is marred by a number of issues for McKean-Vlasov equations. This is caused mainly by the nonlinearity of the semigroup on measures and the failure of the Fubini theorem as a result of loss of linearity in contrast to the case of stochastic differential equations. Here, instead of lifting onto $\mathbb{T}^n\times \mathbb{R}^d$ as being used for the SDEs case, a further novelty of this paper is the double lift onto the graph of reparameterized system, on which we can naturally define a dynamical system. With the extended probability space $\hat{\Omega}=\mathbb{T}^n\times \Omega$, one can define the law of random variables on the graph and the law of the dynamical system. This defines a semigroup $P_t^*$ that describes the dynamics of the evaluation of the law of random variables on the graph. {This pushes the semigroup $P_t^*$ acting on the measure $\delta_{s_1,\cdots,s_n}\times \mu_{s_1,\cdots,s_n}$ ``linearly'' in the sense of (\ref{0531}),
although it is nonlinear. This enables us to obtain an invariant measure with respect to $\hat{P}^*_t$.}

The rest of this paper is organised as follows. In Section 2, we present the main definitions, assumptions, and theorems (Theorems~\ref{Theorem existence of entrance measure}--\ref{Theorem nonuniqueness}). The existence of entrance measures for the McKean-Vlasov equation (proof of Theorem~\ref{Theorem existence of entrance measure}) is established in Section~3. In Section~4, we study the dynamical systems on the lifted graph of reparameterized systems under the quasi-periodic setting and provide proofs for Theorem~\ref{thm:existence-of-invariant-measure} (existence of invariant measures for the lifted semigroup) and Theorem~\ref{Thm existence quasi-periodic measure} (existence of asymptotic quasi-periodic measures). Section~5 is devoted to the study of the nonuniqueness of entrance, periodic, and asymptotic quasi-periodic measures (proof of Theorem~\ref{Theorem nonuniqueness}). Appendix~A contains several basic theorems used throughout the paper.

\section{Assumptions and main results}

Denote by $\mathcal{P}(\mathbb{R}^d)$ (resp. $\mathcal{P}(E)$) the space of all probability measures on $(\mathbb{R}^d,\mathcal{B}(\mathbb{R}^d))$ (resp. a measurable space $(E,\mathcal{E})$). It is well-known that $\mathcal{P}(\mathbb{R}^d)$ is a separable metric space under the weak convergence topology. For a given $p \geq 1$, we consider the following subspace of $\mathcal{P}(\mathbb{R}^d)$:
 \begin{equation*}
	\mathcal{P}_p(\mathbb{R}^d):=\bigg\{\mu\in \mathcal{P}(\mathbb{R}^d)\bigg| \ \|\mu\|_p:= \biggl(\int_{\mathbb{R}^d}|x|^p\mu(\mathrm{d}x)\biggr)^{\frac{1}{p}}<\infty \bigg\},
 \end{equation*}
 which is a Polish space under the following $p$-Wasserstein distance:
 \begin{equation*}
	 \mathcal{W}_p(\mu_1,\mu_2):=\inf_{\gamma\in \mathscr{C}(\mu_1,\mu_2)}\biggl(\int_{\mathbb{R}^d\times \mathbb{R}^d}|x-y|^p\gamma(\mathrm{d}x,\mathrm{d}y)\biggr)^{\frac{1}{p}}, \ \mu_1,\mu_2\in \mathcal{P}_p(\mathbb{R}^d),
 \end{equation*}
where $\mathscr{C}(\mu_1,\mu_2)$ is the set of all coupling measures on $\mathbb{R}^d\times \mathbb{R}^d$ with marginal distributions $\mu_1$ and $\mu_2$.

For convenience, we introduce some definitions of spaces. For a given measurable space $(E,\mathcal{E})$ with a measure $\mu$:
 \begin{itemize}
	 \item $L^0(E)$ (or $L^0(\mathcal{E})$) is the collection of all measurable $\mathbb{R}^d$-valued function $f: E\to \mathbb{R}^d$;
	 \item For any $p\geq 1$, $L^p(E)$ (or $L^p(\mathcal{E}), L^p(\mu)$) is the collection of all $f\in L^0(E)$ such that
	 \begin{equation*}
		 \|f\|_{L^p}:=\biggl(\int_{E}|f(x)|^p\mu(\mathrm{d}x)\biggr)^{\frac{1}{p}}<\infty.
	 \end{equation*}
 \end{itemize}

Given a metric space $\mathbb{T}$ ($\mathbb{T}=\mathbb{R}, \mathbb{R}^n$ or torus $\mathbb{T}^n$) and subset $I \subset \mathbb{T}$, for any $p \geq 1$:
 \begin{itemize}
	 \item $C(I;\mathcal{P}(\mathbb{R}^d))$ is the collection of all continuous maps $\mu:I\to \mathcal{P}(\mathbb{R}^d)$ under the weak convergence topology;
	 \item $C(I;\mathcal{P}_p(\mathbb{R}^d))$ is the collection of all continuous maps $\mu:I\to \mathcal{P}_p(\mathbb{R}^d)$ under the $\mathcal{W}_p$ distance;
	 \item $\mathcal{M}_p(I):=\big\{\mu: I\to \mathcal{P}_p(\mathbb{R}^d) \text{ such that } \|\mu\|_{\mathcal{M}_p}:=\sup_{t\in I}\|\mu_t\|_p<\infty\big\}$;
	 \item $\mathcal{M}_{p,loc}(I):=\big\{\mu: I\to \mathcal{P}_p(\mathbb{R}^d) \text{ such that } \mu|_{K}\in \mathcal{M}_p(K) \text{ for all compact } K\subset I\big\}$.
 \end{itemize}
 
 For $I\subset \mathbb{R}$ and $p\geq 1$,
 \begin{itemize}
     \item $\mathcal{S}_p(I)$ is the collection of all $\mathbb{R}^d$-valued $\mathbf{P}$-a.s. continuous adapted processes $(z_t)_{t\in I}$ such that $$\mathbf{E}\biggl[\sup_{t\in I}|z_t|^p\biggr]<\infty;$$
	 \item $\mathcal{S}_{p,loc}(I):=\big\{z:I\times \Omega\to \mathbb{R}^d| z \text{ is adapted and } z|_{K}\in \mathcal{S}_p(K) \text{ for all compact } K\subset I\big\}$.
 \end{itemize}

 In this paper, we consider the McKean-Vlasov SDE \eqref{McKean-Vlasov SDE 2} in state space $\mathbb{R}^d$. Here $W_t, t\in \mathbb{R}$ in \eqref{McKean-Vlasov SDE 2} is a standard two-sided $m$-dimensional Brownian motion on a complete probability space with natural filtration $\big(\Omega,\mathcal{F},(\mathcal{F}_t)_{t\in \mathbb{R}},\mathbf{P}\big)$, $b: \mathbb{R}\times \mathbb{R}^d\times \mathcal{P}_2(\mathbb{R}^d)\to \mathbb{R}^d$, $\sigma: \mathbb{R}\times \mathbb{R}^d\times \mathcal{P}_2(\mathbb{R}^d)\to \mathbb{R}^d\to \mathbb{R}^{d\times m}$ are continuous functions.

 \begin{definition}
	A $\mathbf{P}$-a.s. continuous adapted process $\{X_t\}_{t\geq s}$ on $\mathbb{R}^d$ is called a (strong) solution to \eqref{McKean-Vlasov SDE 2} with starting time $s\in \mathbb{R}$ and initial condition $\xi\in L^0(\mathcal{F}_s)$ if $\mathbf{P}$-a.s.,
	\begin{equation*}
	   X_t=\xi+\int_{s}^{t}b(r,X_r,\mathscr{L}(X_r))\mathrm{d}r+\int_{s}^{t}\sigma(r,X_r,\mathscr{L}(X_r))\mathrm{d}W_r, \ t\geq s.
	\end{equation*}
	We say that \eqref{McKean-Vlasov SDE 2} has a unique (strong or pathwise) solution if for any two solutions $(X^i_t)_{t\geq s}, i=1,2$ with the same initial condition $(s,\xi)$, we have $X^1_t=X^2_t$, for all $t\geq s$, $\mathbf{P}$-a.s.

	A pair $\big(\tilde{X}_t, \tilde{W}_t\big)_{t\geq s}$ (or $\tilde{X}$ without ambiguity) is called a weak solution of \eqref{McKean-Vlasov SDE 2} with initial distribution $\nu$ if $\mathscr{L}(\tilde{X}_s)=\nu$, $\{\tilde{W}_t\}_{t\geq s}$ is a standard $m$-dimensional Brownian motion on a complete probability space with filtration $\big(\tilde{\Omega},\tilde{\mathcal{F}},(\tilde{\mathcal{F}}_t)_{t\geq s},\tilde{\mathbf{P}}\big)$, and $\{\tilde{X}_t\}_{t\geq s}$ solves the following McKean-Vlasov SDE
	\begin{equation}\label{MVSDE of weak solution}
	   \tilde{X}_t=\tilde{X}_s+\int_{s}^{t}b(r,\tilde{X}_r,\mathscr{L}(\tilde{X}_r))\mathrm{d}r+\int_{s}^{t}\sigma(r,\tilde{X}_r,\mathscr{L}(\tilde{X}_r))\mathrm{d}\tilde{W}_r, \ t\geq s.
	\end{equation}
	We say that \eqref{McKean-Vlasov SDE 2} has a unique weak solution if for any two weak solutions $\tilde{X}^1$ and $\tilde{X}^2$ with the same initial law $\nu$ at $s$, we have $\mathscr{L}(\tilde{X}^1_t)=\mathscr{L}(\tilde{X}^2_t)$ for all $t\geq s$.
\end{definition}

\begin{remark}
	We often denote by $\{X_t^{s,\xi}\}_{t\geq s}$ (or $X_t^{s,\xi}$) a strong solution of \eqref{McKean-Vlasov SDE 2} with starting time $s\in \mathbb{R}$ and initial condition $\xi$. Similarly, $\{X_t^{s,\nu}\}_{t\geq s}$ (or $X_t^{s,\nu}$) is often written as a weak solution of \eqref{McKean-Vlasov SDE 2} with starting time $s\in \mathbb{R}$ and initial distribution $\nu$.
\end{remark}

Now, we provide the following assumption, which guarantees the well-posedness of the equation \eqref{McKean-Vlasov SDE 2}.
  \begin{condition}\label{A3}
	\begin{description}
		\item [(i)] There exist constants $c_{\sigma}>c'_{\sigma}>0$ such that for all $t\in \mathbb{R}, x, y\in \mathbb{R}^d, \mu,\nu\in \mathcal{P}_2(\mathbb{R}^d)$
			\begin{equation}
				\label{Ineq of non-degenerate diffusion}
				c'_{\sigma}|y|^2\leq \langle \sigma\sigma^{\top}(t,x,\mu)y, y \rangle\leq c_{\sigma}|y|^2,
			\end{equation}
			and there exists constant $L>0$ such that
			\begin{equation}
				\label{my1127}
				\|\sigma(t,x,\mu)-\sigma(t,y,\nu)\|\leq L (|x-y|+\mathcal{W}_2(\mu,\nu)),
			\end{equation}
			where $\|\sigma\|^2:=Tr (\sigma\sigma^{\top})$.
		\item [(ii)] There exists $L>0$ such that for all $t\in \mathbb{R}, x,y\in \mathbb{R}^d, \mu,\nu\in \mathcal{P}_2(\mathbb{R}^d)$,
	  \begin{equation}\label{eq:one-side-Lipschitz}
		  \langle x-y, b(t,x,\mu)-b(t,y,\nu)\rangle\leq L\bigl(|x-y|^2+|x-y|\mathcal{W}_2(\mu,\nu)\bigr),
	  \end{equation}
	   and there exists $\kappa\geq 1$ such that for all $t\in \mathbb{R}, x,y\in \mathbb{R}^d, \mu,\nu\in \mathcal{P}_2(\mathbb{R}^d)$,
		\begin{equation}
			\label{Ineq polynomial growth}
			|b(t,x,\mu)|\leq L(1+|x|^{\kappa}+\|\mu\|_2^{\kappa}).
		\end{equation}
		\end{description}
	\end{condition}

	\begin{remark}
		By \eqref{Ineq of non-degenerate diffusion}, we know that
		\begin{equation*}
			\|\sigma(t,x,\mu)\|^2\leq dc_{\sigma} \ \text{ for all $t\in \mathbb{R}, \ x\in \mathbb{R}^d, \ \mu\in \mathcal{P}_2(\mathbb{R}^d)$.}
		\end{equation*}
	\end{remark}

	Under Assumption \ref{A3}, equation \eqref{McKean-Vlasov SDE 2} has a unique strong (resp. weak) solution $X_t^{s,\xi}$ (resp. $X_t^{s,\mu}$) in $\mathcal{S}_{2,loc}([s,\infty))$ with initial condition $\xi\in L^2(\mathcal{F}_s)$ (resp. $\mu\in \mathcal{P}_2(\mathbb{R}^d)$) for any starting time $s\in \mathbb{R}$ (see c.f. Wang \cite{wang2018}). In this case, for any $t\geq s$, we define $P^*(t,s): \mathcal{P}_2(\mathbb{R}^d)\to \mathcal{P}_2(\mathbb{R}^d)$ by
	\begin{equation}\label{eq:def-P^*}
		P^*(t,s)\mu=\mathscr{L}(X_t^{s,\mu}), \ \text{ for any } \mu\in \mathcal{P}_2(\mathbb{R}^d).
	\end{equation}
	Then $P^*(t,s)$ is well-defined on $\mathcal{P}_2(\mathbb{R}^d)$ since equation \eqref{McKean-Vlasov SDE 2} has a unique weak solution in $\mathcal{S}_{2,loc}([s,\infty))$ for any starting time $s$ and initial distribution $\mu\in \mathcal{P}_2(\mathbb{R}^d)$.

We now define entrance measures, quasi-periodic measures, and asymptotic quasi-periodic measures according to $P^*$.

	\begin{definition}
	  A measure-valued function $\mu_{\cdot}:\mathbb{R}\to \mathcal{P}_{2}(\mathbb{R}^d)$ is called an entrance measure of \eqref{McKean-Vlasov SDE 2} (or $P^*$) if $P^*(t,s)\mu_s=\mu_t$ for any $t,s$ with $t\geq s$.

	  A measure-valued function $\mu_{\cdot}:\mathbb{R}\to \mathcal{P}_{2}(\mathbb{R}^d)$ is called quasi-periodic with periods $\tau_1, \tau_2, \cdots, \tau_n$ if there exists a weakly continuous measure-valued function $\tilde{\mu}: \mathbb{R}^n\to \mathcal{P}(\mathbb{R}^d)$ such that $\tilde{\mu}_{t,\cdots,t}=\mu_t$ for all $t\in \mathbb{R}$ and 
	  \begin{equation}\label{equ quasi-periodic mu}
		\tilde{\mu}_{t_1,\cdots, t_i+\tau_i,\cdots,t_n}=\tilde{\mu}_{t_1,\cdots, t_i,\cdots,t_n}, \text{ for all } i\in \{1,2,\cdots,n\}, \ (t_1,\cdots,t_n)\in \mathbb{R}^n,
	  \end{equation}
	  where $\tilde{\mu}$ is called the quasi-periodic representation of $\mu$.

       A measure-valued function $\mu_{\cdot}:\mathbb{R}\to \mathcal{P}_{2}(\mathbb{R}^d)$ is called asymptotic quasi-periodic with periods $\tau_1, \tau_2, \cdots, \tau_n$ if there exist a sequence of quasi-periodic measure-valued functions $\{\mu^m_{\cdot}\}_{m\geq 1}$ such that for any $T>0$,
      \begin{equation*}
          \lim_{m\to\infty}\sup_{t\in [-T,T]}\mathcal{W}_2(\mu^m_t,\mu_t)=0.
      \end{equation*}

      We call $\mu_{\cdot}$ a quasi-periodic/asymptotic quasi-periodic measure with periods $\tau_1, \tau_2, \cdots, \tau_n$ of \eqref{McKean-Vlasov SDE 2} (or $P^*$) if it is an entrance measure and quasi-periodic/asymptotic quasi-periodic with periods $\tau_1, \tau_2, \cdots, \tau_n$.
			 In particular, when $n=1$, the quasi-periodic measure is a periodic measure.
  \end{definition}

Now we introduce additional conditions on the drift term $b$ to guarantee the existence of an entrance measure.
	\begin{condition}\label{A4}
		Assume that
		\begin{description}
			\item [(i)] there exist bounded functions $\alpha,\beta,\gamma$ such that for all $t\in \mathbb{R}, x\in \mathbb{R}^d, \mu\in \mathcal{P}_2(\mathbb{R}^d)$,
			\begin{equation}
				\label{Ineq weakly coercivity}
				\langle x, b(t,x,\mu)\rangle \leq \alpha_t |x|^2+\beta_t\|\mu\|_2^2+\gamma_t.
			\end{equation}
			\item [(ii)] $\alpha,\beta,\gamma$ satisfy
			\begin{equation}\label{eq:average weakly dissipation}
				\limsup_{T\to \infty}\frac{1}{T}\int_{-T}^{0}(\alpha_r+\beta_r)\mathrm{d}r<0,
			\end{equation}
			and
			\begin{equation*}
				\limsup_{t\to -\infty}\int_{-\infty}^{t}e^{2\int_{u}^{t}(\alpha_r+\beta_r)\mathrm{d}r}(2\gamma_u+dc_{\sigma})\mathrm{d}u<\infty.
			\end{equation*}
		\end{description}
	\end{condition}
Denote 
\[
\vartheta_t=\int_{-\infty}^{t}e^{2\int_{u}^{t}(\alpha_r+\beta_r)\mathrm{d}r}(2\gamma_u+dc_{\sigma})\mathrm{d}u.
\]
\begin{remark}\label{remark 0522}
		\begin{itemize}
			\item [(i)] By \eqref{Ineq weakly coercivity}, we conclude that $\beta, \gamma$ are non-negative functions. In fact, letting $x=0, \mu=\delta_y$ for some $y\in \mathbb{R}^d$ in \eqref{Ineq weakly coercivity}, we have $\beta_t|y|^2+\gamma_t\geq 0$. So $\beta_t\geq 0$ and $\gamma_t\geq 0$ since $y\in \mathbb{R}^d$ can be arbitrary.
			\item [(ii)] If there exist $l>0, a>0$ such that $\int_{t}^{t+l}(\alpha_r+\beta_r)\mathrm{d}r\leq -a$ for all $t\in \mathbb{R}$, then Assumption \ref{A4} {\bf (ii)} holds. In particular, Assumption \ref{A4} {\bf (ii)} holds when $\alpha_t\equiv \alpha, \beta_t\equiv \beta, \gamma_t\equiv \gamma$ with $\alpha+\beta<0$.

			In fact, for any $s\leq t\leq t_0$ with $il\leq |t-s|\leq (i+1)l$ for some integer $i\geq 0$, we know that
			\begin{equation*}
				\int_{s}^{t}(\alpha_r+\beta_r)\mathrm{d}r=\int_{t-il}^{t}(\alpha_r+\beta_r)\mathrm{d}r+\int_{s}^{t-il}(\alpha_r+\beta_r)\mathrm{d}r\leq -ia+2lM,
			\end{equation*}
			where $M=|\alpha|_{\infty}\vee|\beta|_{\infty}\vee|\gamma|_{\infty}$. Then it follows that
			\begin{equation*}
				\begin{split}
					\vartheta_t&=\sum_{j=0}^{\infty}\int_{t-(j+1)\Delta}^{t-j\Delta}e^{2\int_{u}^{t}(\alpha_r+\beta_r)\mathrm{d}r}(2\gamma_u+dc_{\sigma})\mathrm{d}u\\
					&\leq \sum_{j=0}^{\infty} e^{-2ja+4lM}(2M+dc_{\sigma})l
					=\frac{e^{4lM}(2M+dc_{\sigma})l}{1-e^{-2a}}.
				\end{split}
			\end{equation*}	
		\end{itemize}
	\end{remark}

Set
	\begin{equation*}
		\mathcal{M}_{2,\vartheta}:=\bigg\{\mu:\mathbb{R}\to \mathcal{P}_2(\mathbb{R}^d) \bigg| \int_{\mathbb{R}^d}|x|^2\mu_t(\mathrm{d}x)\leq \vartheta_t\bigg\}.
	\end{equation*}

	The following is our first main theorem.

	\begin{theorem}\label{Theorem existence of entrance measure}
		Assume Assumptions \ref{A3}, \ref{A4} hold. Then equation \eqref{McKean-Vlasov SDE 2} has an entrance measure in $C(\mathbb{R};\mathcal{P}_2(\mathbb{R}^d))\cap \mathcal{M}_{2,\vartheta}$.
	\end{theorem}

Next, we assume further that the coefficients $b, \sigma$ of \eqref{McKean-Vlasov SDE 2} satisfy the quasi-periodic condition.

	\begin{condition}\label{Quasi-periodic condition}
		The coefficients $b,\sigma$ are quasi-periodic in $t$ with periods $\tau_1,\cdots,\tau_n$ where the reciprocals of $\tau_1,\cdots,\tau_n$ are rationally independent, i.e., there exist continuous $\tilde{b}: \mathbb{R}^n\times \mathbb{R}^d\times \mathcal{P}_2(\mathbb{R}^d)\to \mathbb{R}^d$ and $\tilde{\sigma}: \mathbb{R}^n\times \mathbb{R}^d\times \mathcal{P}_2(\mathbb{R}^d)\to \mathbb{R}^{d\times m}$ such that for any $t\in \mathbb{R}, x\in \mathbb{R}^d, \mu\in \mathcal{P}_2(\mathbb{R}^d)$, $i\in \{1,2,\cdots,n\}$ and for $f=b,\sigma$,
		\begin{equation*}
			\tilde{f}(t,\cdots,t,x,\mu)=f(t,x,\mu), \ \ \tilde{f}(t_1,\cdots,t_{i}+\tau_i,\cdots,t_n,x,\mu)=\tilde{f}(t_1,\cdots,t_{i},\cdots,t_n,x,\mu).
		\end{equation*}

		Moreover, when Assumptions \ref{A3} and \ref{A4} {\bf (i)} hold, $\alpha,\beta,\gamma$ given in \eqref{Ineq weakly coercivity} are also quasi-periodic with periods $\tau_1,\cdots,\tau_n$, i.e., there exist continuous $\tilde{f}: \mathbb{R}^n\to \mathbb{R}, f=\alpha, \beta,\gamma$ such that $\tilde{f}_{t,\cdots,t}=f_t, f=\alpha,\beta,\gamma$ and $\tilde{\alpha}_{t_1,\cdots,t_n},\tilde{\beta}_{t_1,\cdots,t_n},\tilde{\gamma}_{t_1,\cdots,t_n}$ are periodic in $(t_1,\cdots,t_n)$ with periods $\tau_1,\cdots,\tau_n$ respectively.
	\end{condition}

\begin{remark}\label{lemma 0427}
	If Assumptions \ref{A3}, \ref{A4} {\bf (i)} and \ref{Quasi-periodic condition} hold, then $\tilde{b}(t_1,\cdots,t_n,x), \tilde{\sigma}(t_1,\cdots,t_n,x)$ satisfy \eqref{Ineq of non-degenerate diffusion}-\eqref{Ineq polynomial growth} with the same $c'_{\sigma}, c_{\sigma},L, \kappa$ and satisfy \eqref{Ineq weakly coercivity} with $\tilde{\alpha}_{t_1,\cdots,t_n}, \tilde{\beta}_{t_1,\cdots,t_n}, \tilde{\gamma}_{t_1,\cdots,t_n}$.
\end{remark}

Let $\mathbb{T}^n:=[0,\tau_1)\times \cdots \times [0,\tau_n)$. Then $\mathbb{T}^n$ is an $n$-torus with the following metric $d_0$: 
\begin{equation*}
		d_0(\vec{\mathbf{t}},\vec{\mathbf{s}}):=\sum_{i=1}^n\min\{|t_i-s_i|, \ \ \tau_i-|t_i-s_i|\}, \ \vec{\mathbf{t}}=(t_1,\cdots,t_n), \ \text{ for all } \ \vec{\mathbf{s}}=(s_1,\cdots,s_n)\in \mathbb{T}^n.
\end{equation*}
For any $t\in\mathbb{R}$, set the minimal rotation $T_t:\mathbb{T}^n\to\mathbb{T}^n$ by
\begin{equation}\label{0529-1}
		T_t(\vec{\mathbf{s}}):=(t+s_1 \mod \tau_1,\cdots,t+s_n\mod \tau_n), \ \text{ for all } \ \vec{\mathbf{s}}=(s_1,\cdots,s_n)\in \mathbb{T}^n
	\end{equation}

We assume $\tilde{b}, \tilde{\sigma}$ in Assumption \ref{Quasi-periodic condition} satisfy the following uniformly continuous condition.

\begin{condition}\label{One-sided Lip of tilde b}
		There exists a continuous function $h: \mathbb{T}^n\times\mathbb{T}^n\to \mathbb{R}^+$ with $h(\vec{\mathbf{t}}, \vec{\mathbf{t}})=0$ for all $\vec{\mathbf{t}}\in \mathbb{T}^n$ such that for any $\vec{\mathbf{t}}, \vec{\mathbf{s}}\in \mathbb{T}^n$, $x\in \mathbb{R}^d$ and $\mu\in \mathcal{P}_2(\mathbb{R}^d)$,
		\begin{equation*}
			|\tilde{b}(\vec{\mathbf{t}};x,\mu)-\tilde{b}(\vec{\mathbf{s}};x,\mu)|^2 +\|\tilde{\sigma}(\vec{\mathbf{t}};x,\mu)-\tilde{\sigma}(\vec{\mathbf{s}};x,\mu)\|^2
			\leq h(\vec{\mathbf{t}}, \vec{\mathbf{s}}).
		\end{equation*}
	\end{condition}

 Under Assumption \ref{Quasi-periodic condition}, we consider the reparameterized McKean-Vlasov SDE corresponding to \eqref{McKean-Vlasov SDE 2}:
	\begin{equation}\label{New Equation K_r_1,r_2}
		\begin{split}
			K^{\vec{\mathbf{s}}}(t,s,\xi)&=\xi+\int_{s}^{t}\tilde{b}^{\vec{\mathbf{s}}}\big(u,K^{\vec{\mathbf{s}}}(u,s,\xi), \mathscr{L}\big(K^{\vec{\mathbf{s}}}(u,s,\xi)\big)\big)\mathrm{d}u\\
			&\ \ \ \ +\int_{s}^{t}\tilde{\sigma}^{\vec{\mathbf{s}}}\big(u,K^{\vec{\mathbf{s}}}(u,s,\xi), \mathscr{L}\big(K^{\vec{\mathbf{s}}}(u,s,\xi)\big)\big)\mathrm{d}W_u,
		\end{split}
	\end{equation}
	where $\vec{\mathbf{s}}:=(s_1,\cdots,s_n)$, and $\tilde{f}^{\vec{\mathbf{s}}}(t,x,\nu):=\tilde{f}(t+s_1,\cdots,t+s_n,x,\nu), \ f=b,\sigma$. If further Assumption \ref{A3} holds, then equation \eqref{New Equation K_r_1,r_2} has a unique strong/weak solution $K^{\vec{\mathbf{s}}}(t,s,\xi)$ for all $t\geq s, \xi\in L^2(\mathcal{F}_s)$. Similar to $P^*(t,s)$, define $P^{\vec{\mathbf{s}},*}(t,s): 
	\mathcal{P}_2(\mathbb{R}^d)\to \mathcal{P}_2(\mathbb{R}^d)$ by 
	\begin{equation*}
			P^{\vec{\mathbf{s}},*}(t,s)\mu=\mathscr{L}\big(K^{\vec{\mathbf{s}}}(t,s,\xi)\big) \text{ with } \mathscr{L}(\xi)=\mu.
	\end{equation*}
It is easy to see from its definition that 
this multi-parameter semigroup is nonlinear and satisfies $P^{\vec{\mathbf{s}},*}(t+r,s+r)=P^{T_r(\vec{\mathbf{s}}),*}(t,s)$. For the case of SDEs and SPDEs, this semigroup, which is apparently linear, was also used in \cite{Feng-Qu-Zhao2021,Feng-Qu-Zhao2023,Liu-Lu2025}. The nonlinear semigroup is very difficult to deal with. But we found a way of lifting it to a time-homogeneous semigroup and constructing an invariant measure.

Recall the following definition of a continuous dynamical system on a metric space $(\mathbb{X},d)$.
		\begin{definition}\label{Dynamical system}
			A map $\Phi: 
			\mathbb{T}\times \mathbb{X}\rightarrow \mathbb{X}$ ($\mathbb{T}=\mathbb{R}$ or $\mathbb{R}^+$) is said to be a continuous dynamical system, if 
			\begin{itemize}
				\item  $\Phi: \mathbb{T}\times \mathbb{X}\rightarrow \mathbb{X}$ is continuous;
				\item $\Phi_0(x)=x$;
				\item $\Phi_t(\Phi_s(x))=\Phi_{t+s}(x)$ for all $t,s\in \mathbb{T}$ and $x\in \mathbb{X}$.
			\end{itemize}
		\end{definition}

Now for any $\vec{\mathbf{s}}\in \mathbb{T}^n$ and $\xi \in L^{2}(\mathcal{F}_0)$, we define $\Phi_t:\mathbb{T}^n\times L^2(\mathcal{F}_0)\to \mathbb{T}^n\times L^2(\mathcal{F}_0)$ for all $t\geq 0$ by
	\begin{equation}\label{0508-4}
		\Phi_t(\vec{\mathbf{s}};\xi)=\big(T_t(\vec{\mathbf{s}}); K^{\vec{\mathbf{s}}}(t,0,\xi)\circ \theta_{-t}\big),
	\end{equation}
where $\theta_r: \Omega\to \Omega$ is the Wiener shift, i.e., $\theta_rW_t(\omega)=W_{t+r}(\omega)-W_r(\omega), \ \mathbf{P}-a.s.$ for all $r\in \mathbb{R}$.
We will show that $(\mathbb{R}^+, \mathbb{T}^n\times L^{2}(\mathcal{F}_0), \Phi)$ is a continuous dynamical system (see Lemma \ref{lemma 0530}). This is called the first lift with pull-back to a dynamical system on the space of square integrable multi-parameter pull-back flows defined on cylinder. This first lift, although it appears similar, is distinct from the lift to cocycles on the cylinder for SDEs or SPDEs (see \cite[Lemma 3.12]{Feng-Qu-Zhao2021}, \cite[(2.10)]{Liu-Lu2025}). It will be needed in the second lift explained below.

Let $\hat{\Omega}:=\mathbb{T}^n\times \Omega$. Denote by $C(\mathbb{T}^n; L^{2}(\mathcal{F}_0))$ the collection of all continuous maps $\tilde{\xi}: \mathbb{T}^n\to L^{2}(\mathcal{F}_0)$. From a different point of view, for any $\tilde{\xi}\in C(\mathbb{T}^n; L^{2}(\mathcal{F}_0))$, it is easy to check that $\tilde{\xi}: \hat{\Omega}\to \mathbb{R}^d$ is measurable with respect to the $\sigma$-algebra $\hat{\mathcal{F}}:= \mathcal{B}(\mathbb{T}^n)\otimes \mathcal{F}_0$ on $\hat{\Omega}$. Now let $\mathcal{G}_{2}$ be the graph of $C(\mathbb{T}^n; L^{2}(\mathcal{F}_0))$, i.e.,
	\begin{equation*}
		\mathcal{G}_{2}:=\big\{{\rm graph}(\tilde{\xi})=\{(\vec{\mathbf{s}};\tilde{\xi}_{\vec{\mathbf{s}}})\}_{\vec{\mathbf{s}}\in \mathbb{T}^n}: \tilde{\xi} \in C(\mathbb{T}^n; L^{2}(\mathcal{F}_0))\big\}.
	\end{equation*}
	It can be proved that $\mathcal{G}_{2}$ is a complete metric space under the following norm:
	\begin{equation*}
		\|\hat{\xi}-\hat{\eta}\|_{\mathcal{G}_{2}}:=\sup_{\vec{\mathbf{s}}\in \mathbb{T}^n}\bigl(\mathbf{E}\bigl[|\tilde{\xi}_{\vec{\mathbf{s}}}-\tilde{\eta}_{\vec{\mathbf{s}}}|^2\bigr]\bigr)^{\frac{1}{2}}, \ \text{ for all } \hat{\xi}={\rm graph}(\tilde{\xi}), \ \hat{\eta}={\rm graph}(\tilde{\eta})\in \mathcal{G}_{2}.
	\end{equation*}
	Moreover, every element $\hat{\xi}={\rm graph}(\tilde{\xi})\in \mathcal{G}_{2}$ can be evaluated at $\vec{\mathbf{s}}\in \mathbb{T}^n$ and $\omega\in \Omega$ by
	\begin{equation*}
		\hat{\xi}(\vec{\mathbf{s}},\omega):=\big(\vec{\mathbf{s}};\tilde{\xi}_{\vec{\mathbf{s}}}(\omega)\big).
	\end{equation*}
  Let $\hat{\mathbb{X}}=\mathbb{T}^n\times \mathbb{R}^d$.
	Then it is easy to check that $\hat{\xi}: (\hat{\Omega},\hat{\mathcal{F}})\to \big(\hat{\mathbb{X}},\mathcal{B}(\hat{\mathbb{X}})\big)$ is a measurable map. 
	
For any $t\geq 0$, we define $\hat{\Phi}_t$ on $\mathcal{G}_{2}$ by the following relation: for any $\hat{\xi}={\rm graph}(\tilde{\xi})\in \mathcal{G}_{2}$,
	\begin{equation}\label{0529-2}
			(\hat{\Phi}_t\hat{\xi})(\vec{\mathbf{s}},\omega)=\big(T_t(\vec{\mathbf{s}}), K^{\vec{\mathbf{s}}}(t,0,\tilde{\xi}_{\vec{\mathbf{s}}}) (\theta_{-t}\omega)\big), \ \mathbf{P}-a.s.
	\end{equation}
	Then we will prove that $(\mathbb{R}^+,\mathcal{G}_{2},\hat{\Phi})$ is a continuous dynamical system (see Lemma \ref{coro 0530}). By \eqref{0508-4} and \eqref{0529-2},
	it is easy to see that 
	\begin{equation*}
		(\hat{\Phi}_t\hat{\xi})(\vec{\mathbf{s}},\omega)=\Phi_t(\vec{\mathbf{s}},\tilde{\xi}_{\vec{\mathbf{s}}})(\omega), \ \mathbf{P}-a.s.,
	\end{equation*}
	i.e., the evaluation of the dynamical system $(\mathbb{R}^+,\mathcal{G}_{2},\hat{\Phi})$ on the graph $\mathcal{G}_{2}$ is a dynamical system of the first lift multi-parameter pull-back flows on a cylinder. We call $(\mathbb{R}^+,\mathcal{G}_{2},\hat{\Phi})$ the second lift to the flow on the graph of the square integrable multi-parameter pull-back flows. This second lift is key to the analysis for the dynamics of quasi-periodic McKean-Vlasov stochastic differential equations.

According to \eqref{0508-4}, it is natural to define $\hat{P}_t^*$ on $\delta_{\vec{\mathbf{s}}}\times \mu$ where $\vec{\mathbf{s}}\in \mathbb{T}^n$ and $\mu\in \mathcal{P}_{2}(\mathbb{R}^d)$ by
	\begin{equation}\label{0530-1}
		\hat{P}_t^* \bigl(\delta_{\vec{\mathbf{s}}}\times \mu\bigr)=\delta_{T_t(\vec{\mathbf{s}})}\times P^{\vec{\mathbf{s}},*}(t,0)\mu.
	\end{equation}
	But how should we define $\hat{P}_t^*$ more generally on $\mathcal{P}(\hat{\mathbb{X}})$? If we look back \eqref{0529-2}, it seems natural to define $\hat{P}_t^*$ by
	\begin{equation}\label{0530-2}
		\hat{P}_t^*\bigl(``\mathscr{L}"(\hat{\xi})\bigr)=``\mathscr{L}"(\hat{\Phi}_t\hat{\xi}).
	\end{equation}
Note here \eqref{0530-2} is conceptual as $``\mathscr{L}"(\hat{\xi})$ is not a real law since we have not defined a probability measure on $(\hat{\Omega},\hat{\mathcal{F}})$. In the following, we will formalize \eqref{0530-2}.

For any $\nu\in \mathcal{P}(\mathbb{T}^n)$, $\hat{\Omega}_{\nu}:=(\hat{\Omega},\hat{\mathcal{F}},\nu\times \mathbf{P})$ is a probability space and $\mathscr{L}_{\nu}(\hat{\xi})$ is the law of a measurable map $\hat{\xi}: \hat{\Omega}_{\nu}\to (\hat{\mathbb{X}},\mathcal{B}(\hat{\mathbb{X}}))$. Thus, in mind of \eqref{0530-2}, we define:
	\begin{equation}\label{0530-3}
		\hat{P}_t^*\bigl(\mathscr{L}_{\nu}(\hat{\xi})\bigr)=\mathscr{L}_{\nu}(\hat{\Phi}_t\hat{\xi}), \ \text{ for all } \hat{\xi}\in \mathcal{G}_{2}, \ \nu\in \mathcal{P}(\mathbb{T}^n).
	\end{equation}
This lifted semigroup is nonlinear in $\mathscr{L}_{\nu}(\hat{\xi})$. In the case of SDEs or SPDEs, it is a linear semigroup, see Section 3 in \cite{Feng-Qu-Zhao2021}, Section 6 in \cite{Feng-Qu-Zhao2023}, this semigroup can be done from the lift of $P^{\vec{\mathbf{s}},*}(t,s)$ linearly. In particular, for the SDE case, the solution $K^{\vec{\mathbf{s}}}(t,s,x)$ can be lifted to a random dynamical system $\hat{\Phi}$ on the cylinder $\hat{\mathbb{X}}=\mathbb{T}^n\times \mathbb{R}^d$:
\begin{equation}\label{eq:explain-1}
	\hat{\Phi}(t,\omega)(\vec{\mathbf{s}},x)=(T_t(\vec{\mathbf{s}}), K^{\vec{\mathbf{s}}}(t,0,x)(\omega)).
\end{equation}
Then consider the corresponding Markovian transition function $\hat{P}:\mathbb{R}^+\times \hat{\mathbb{X}}\times \mathcal{B}(\hat{\mathbb{X}})$ by
\begin{equation*}
	\hat{P}(t, (\vec{\mathbf{s}},x),\hat{\Gamma})=\mathbf{P}\{\omega\in\Omega: \hat{\Phi}(t,\omega)(\vec{\mathbf{s}},x)\in \hat{\Gamma}\}.
\end{equation*}
The dual of $\hat{P}$ is 
\begin{equation*}
	\hat{P}^*_t\hat{\mu}(\hat{\Gamma})=\int_{\hat{\mathbb{X}}}\hat{P}(t, (\vec{\mathbf{s}},x),\hat{\Gamma})\hat{\mu}({\rm d}\vec{\mathbf{s}}, {\rm d}x)=\int_{\hat{\mathbb{X}}}\big(\hat{P}^*_t\delta_{(\vec{\mathbf{s}},x)}\big)(\Gamma)\hat{\mu}({\rm d}\vec{\mathbf{s}}, {\rm d}x).
\end{equation*}
However, this dual relation breaks down in the case of McKean-Vlasov distribution dependent SDEs due to nonlinearity (see \cite{wang2018}).
Instead, we can define $\hat{P}^*_t$ equivalently via \eqref{eq:explain-1} by
\begin{equation}\label{eq:explain-2}
	\hat{P}^*_t\hat{\mu}=\mathscr{L}(\hat{\Phi}(t,\cdot)\hat{\xi}), \ \text{ with } \ \mathscr{L}(\hat{\xi})=\hat{\mu}.
\end{equation}
According to the disintegration theorem of measure (see Theorem 10.2.1 and 10.2.2 in \cite{Dudley2002}), any $\hat{\mu}\in \mathcal{P}(\hat{\mathbb{X}})$ has the form:
    \begin{equation}\label{eq:disintegration-measure}
        \hat{\mu}=\int_{\mathbb{T}^n}\big(\delta_{\vec{\mathbf{s}}}\times \tilde{\mu}_{\vec{\mathbf{s}}}\big) \nu(\mathrm{d}\vec{\mathbf{s}}),
    \end{equation}
    where $\nu(A):=\hat{\mu}(A\times \mathbb{R}^d)$ for all $A\in \mathcal{B}(\mathbb{T}^n)$ is the marginal measure of $\hat{\mu}$ on $\mathbb{T}^n$ and $\{\tilde{\mu}_{\vec{\mathbf{s}}}\}_{\vec{\mathbf{s}}\in \mathbb{T}^n}\subset \mathcal{P}(\mathbb{R}^d)$ is the $\nu$-a.s. unique disintegration of $\hat{\mu}$.
Let $\vec{\mathbf{s}}:\Omega\to \mathbb{T}^n$ be a random vector with law $\nu$ and independent of $\mathcal{F}$, and let $\tilde{\xi}_{\vec{\mathbf{s}}}:\Omega\to \mathbb{R}^d$ be a $\mathcal{F}$-measurable random vector with law $\tilde{\mu}_{\vec{\mathbf{s}}}$ for $\vec{\mathbf{s}}\in \mathbb{T}^n$. Then we know that 
\[
\hat{\xi}(\omega)=(\vec{\mathbf{s}}(\omega), \tilde{\xi}_{\vec{\mathbf{s}}(\omega)}(\omega))
\]
has the law $\hat{\mu}$
and 
\begin{equation*}
	\hat{P}^*_t\hat{\mu}=\mathscr{L}(\hat{\Phi}(t,\cdot)\hat{\xi})=\mathscr{L}\Big(\big(T_t(\vec{\mathbf{s}}), K^{\vec{\mathbf{s}}}(t,0,\tilde{\xi}_{\vec{\mathbf{s}}})\big)\Big).
\end{equation*}
Since $\theta_t$ preserves $\mathbf{P}$ and is $\mathcal{F}$-measurable, we have
\begin{equation}\label{eq:explain-3}
\hat{P}^*_t\hat{\mu}=\mathscr{L}\Big(\big(T_t(\vec{\mathbf{s}}), K^{\vec{\mathbf{s}}}(t,0,\tilde{\xi}_{\vec{\mathbf{s}}})\circ\theta_{-t}\big)\Big).
\end{equation}
Actually, \eqref{eq:explain-3} and the above argument also work for the distribution dependent SDE \eqref{McKean-Vlasov SDE 2}, and correspond to the second lift in the McKean-Vlasov case.
     In particular, according to \eqref{eq:disintegration-measure} and \eqref{eq:explain-3}, it is easy to see that for all $t\geq 0$,
    \begin{equation}\label{eq:new1202-1}
			\hat{P}^*_t\int_{\mathbb{T}^n}\big(\delta_{\vec{\mathbf{s}}}\times \tilde{\mu}_{\vec{\mathbf{s}}}\big) \nu(\mathrm{d}\vec{\mathbf{s}})=\int_{\mathbb{T}^n}\big(\delta_{T_t(\vec{\mathbf{s}})}\times P^{\vec{\mathbf{s}},*}(t,0)\tilde{\mu}_{\vec{\mathbf{s}}}\big) \nu(\mathrm{d}\vec{\mathbf{s}}),
		\end{equation}
    whenever the right hand side of \eqref{eq:new1202-1} is well-defined.

    Recall that $\hat{\mathbb{X}}$ is a Polish space with the following metric $\hat{d}$:
\begin{equation*}
	\hat{d}(\hat{x},\hat{y}):=\sqrt{d_0(\vec{\mathbf{t}},\vec{\mathbf{s}})^2+|x-y|^2}, \ \text{ for all } \ \hat{x}=(\vec{\mathbf{t}},x),\hat{y}=(\vec{\mathbf{s}},y)\in \mathbb{T}^n\times \mathbb{R}^d.
\end{equation*}
For all $p\geq 1$, consider the $p$-Wasserstein space on $(\hat{\mathbb{X}}, \hat{d})$:
    \begin{equation*}
        \mathcal{P}_p(\hat{\mathbb{X}}):=\Big\{\hat{\mu}\in \mathcal{P}(\hat{\mathbb{X}}): \int_{\hat{\mathbb{X}}}\hat{d}^p(\hat{x},\hat{0})\hat{\mu}(\mathrm{d}\hat{x})<\infty \Big\}
    \end{equation*}
    with the following $p$-Wasserstein distance:
 \begin{equation*}
	 \mathcal{D}_p(\mu_1,\mu_2):=\inf_{\hat{\pi}\in \mathscr{C}(\hat{\mu}_1,\hat{\mu}_2)}\biggl(\int_{\hat{\mathbb{X}}\times \hat{\mathbb{X}}}\hat{d}^p(\hat{x},\hat{y})\hat{\pi}(\mathrm{d}\hat{x}, \mathrm{d}\hat{y})\biggr)^{\frac{1}{p}}, \ \hat{\mu}_1,\hat{\mu}_2\in \mathcal{P}_p(\hat{\mathbb{X}}),
 \end{equation*}
 where $\hat{0}=(\vec{0},0^d)\in \hat{\mathbb{X}}$, $0^d\in \mathbb{R}^d$ is the zero vector, and $\mathscr{C}(\hat{\mu}_1,\hat{\mu}_2)$ is the set of all coupling measures on $\hat{\mathbb{X}}\times \hat{\mathbb{X}}$ with marginal distributions $\hat{\mu}_1$ and $\hat{\mu}_2$. We will prove for all $t\geq 0$, $\hat{P}^*_t$ maps $\mathcal{P}_2(\hat{\mathbb{X}})$ into itself. Furthermore, we show that $(\mathbb{R}^+, \mathcal{P}_2(\hat{\mathbb{X}}), \hat{P}^*)$ is a continuous dynamical system (see Proposition \ref{prop:hat-P*-CDS}).

For a continuous dynamical system $\hat{P}^*$ on $\mathcal{P}_2(\hat{\mathbb{X}})$, a crucial task is to explore the existence of invariant measures. If $\hat{\mu}\in \mathcal{P}_2(\hat{\mathbb{X}})$ is an invariant measure of $\hat{P}^*$, let $\nu$ and $\{\tilde{\mu}_{\vec{\mathbf{s}}}\}_{\vec{\mathbf{s}}\in\mathbb{T}^n}$ be the marginal measure on $\mathbb{T}^n$ and the disintegration of $\hat{\mu}$, respectively. Then for any $t\geq 0$, we have
	\begin{equation*}
		\begin{split}
			\hat{P}^*_t\hat{\mu}
			&=\int_{\mathbb{T}^n}\big(\delta_{\vec{\mathbf{s}}}\times P^{T_{-t}(\vec{\mathbf{s}}),*}(t,0)\tilde{\mu}_{T_{-t}(\vec{\mathbf{s}})}\big)\nu\circ T_{t}^{-1}(\mathrm{d}\vec{\mathbf{s}})\\
			&=\int_{\mathbb{T}^n}\big(\delta_{\vec{\mathbf{s}}}\times \tilde{\mu}_{\vec{\mathbf{s}}}\big)\nu(\mathrm{d}\vec{\mathbf{s}}).
		\end{split}
	\end{equation*}
  For any $A\in\mathcal{B}(\mathbb{T}^n)$, we conclude that
  \begin{equation*}
		\nu(T_t^{-1}A)=\hat{P}^*_t\hat{\mu}(A\times \mathbb{R}^d)=\hat{\mu}(A\times \mathbb{R}^d)=\nu(A).
	\end{equation*}
	Thus, $\nu$ is a $T_t$-invariant measure on $\mathbb{T}^n$ for all $t\geq 0$ and 
	\begin{equation*}
		P^{T_{-t}(\vec{\mathbf{s}}),*}(t,0)\tilde{\mu}_{T_{-t}(\vec{\mathbf{s}})}=\tilde{\mu}_{\vec{\mathbf{s}}}, \ \text{ $\nu$-a.s. for all } t\geq 0.
	\end{equation*}
	Hence,
	\begin{equation*}
		\nu(\mathrm{d}\vec{\mathbf{s}})=\frac{1}{\tau_1\cdots\tau_n}\mathrm{d}\vec{\mathbf{s}},
	\end{equation*}
	and for all $t\geq 0$,
	\begin{equation*}
		P^{\vec{\mathbf{s}},*}(t,0)\tilde{\mu}_{\vec{\mathbf{s}}}=\tilde{\mu}_{T_t(\vec{\mathbf{s}})}, \ \text{ $\mathrm{d}\vec{\mathbf{s}}$-a.s.}
	\end{equation*}
	Then we have the following main theorem.

    \begin{theorem}\label{thm:existence-of-invariant-measure}
		Assume Assumptions \ref{A3}, \ref{A4}, \ref{Quasi-periodic condition} and \ref{One-sided Lip of tilde b} hold. Then $\hat{P}^*$ has an invariant measure $\hat{\mu}$ in $\mathcal{P}_{2}(\hat{\mathbb{X}})$. Moreover, there exists $\{\tilde{\mu}_{\vec{\mathbf{s}}}\}_{\vec{\mathbf{s}}\in \mathbb{T}^n}\subset \mathcal{P}_2(\mathbb{R}^d)$ such that
			\begin{equation*}
				\hat{\mu}=\frac{1}{\tau_1\cdots\tau_n}\int_{\mathbb{T}^n}\big(\delta_{\vec{\mathbf{s}}}\times \tilde{\mu}_{\vec{\mathbf{s}}}\big)\mathrm{d}\vec{\mathbf{s}}
			\end{equation*}
			and
			\begin{equation*}
				P^{\vec{\mathbf{s}},*}(t,0)\tilde{\mu}_{\vec{\mathbf{s}}}=\tilde{\mu}_{T_t(\vec{\mathbf{s}})}, \ \text{ $\mathrm{d}\vec{\mathbf{s}}$-a.s. for all } \ t\geq 0.
			\end{equation*}
	\end{theorem}

\begin{remark}
		Taking into considerations of \eqref{0530-1}, \eqref{eq:new1202-1} can be rewritten as
		\begin{equation}\label{0531}
			\hat{P}_t^* \int_{\mathbb{T}^n}\big(\delta_{\vec{\mathbf{s}}}\times \tilde{\mu}_{\vec{\mathbf{s}}} \big)\nu(\mathrm{d}\vec{\mathbf{s}})=\int_{\mathbb{T}^n}\hat{P}_t^*\bigl(\delta_{\vec{\mathbf{s}}}\times \tilde{\mu}_{\vec{\mathbf{s}}}\bigr) \nu(\mathrm{d}\vec{\mathbf{s}}).
		\end{equation}     
It is easy to see that $\hat{P}_t^* \int_{\mathbb{T}^n}\big(\delta_{\vec{\mathbf{s}}}\times \tilde{\mu}_{\vec{\mathbf{s}}} \big)\nu(\mathrm{d}\vec{\mathbf{s}})$ is the law of the coupled process 
$\big(T_t(\cdot
), K^{\cdot
}
(t,0,\tilde{\xi}_{\cdot
})\circ\theta_{-t}\big)
$
with initial distribution $\mathscr{L}\big((\vec{\mathbf{s}},\tilde{\xi}_{\vec{\mathbf{s}}})_{\vec{\mathbf{s}}\in\mathbb{T}^n}\big)=\hat{\mu}=\int_{\mathbb{T}^n}\big(\delta_{\vec{\mathbf{s}}}\times \tilde{\mu}_{\vec{\mathbf{s}}} \big)\nu(\mathrm{d}\vec{\mathbf{s}})$.
On the other hand, $\hat{P}_t^*(\delta_{\vec{\mathbf{s}}}\times \tilde{\mu}_{\vec{\mathbf{s}}})=\delta_{T_t(\vec{\mathbf{s}})}\times \tilde{\mu}_{T_t(\vec{\mathbf{s}})}$ is the law of the coupled process 
$\big(T_t(\vec{\mathbf{s}}), K^{\vec{\mathbf{s}}}(t,0,\tilde{\xi}_{\vec{\mathbf{s}}})\circ\theta_{-t}\big)$
with initial distribution $\mathscr{L}\big((\vec{\mathbf{s}},\tilde{\xi}_{\vec{\mathbf{s}}})\big)=\delta_{\vec{\mathbf{s}}}\times \tilde{\mu}_{\vec{\mathbf{s}}}$. It follows 
that $\int_{\mathbb{T}^n}\hat{P}_t^*\bigl(\delta_{\vec{\mathbf{s}}}\times \tilde{\mu}_{\vec{\mathbf{s}}}\bigr) \nu(\mathrm{d}\vec{\mathbf{s}})$ is the law of the graph $\big(T_t(\vec{\mathbf{s}}), K^{\vec{\mathbf{s}}}(t,0,\tilde{\xi}_{\vec{\mathbf{s}}})\circ\theta_{-t}\big)_{\vec{\mathbf{s}}\in\mathbb{T}^n}$ for any fixed $t$. 

These two laws, generally speaking, are different if acting on a general measure in $\mathcal{P}_{2}(\hat{\mathbb{X}})$. But the special skew-product structure of the dynamical system $\hat{\Phi}_t$ on $\mathcal{G}_2$ that acts fibrewise without mixing different base points in $\mathbb{T}^n$ makes $\hat{P}_t^*$ preserve the measure of the form $\delta_{T_{\cdot}(\vec{\mathbf{s}})}\times \tilde{\mu}_{T_{\cdot}(\vec{\mathbf{s}})}$. In other words, $\hat{P}_t^*$ is compatible with the disintegration of measures over a torus and acts independently along each fibre.

To see that \eqref{0531} does not hold for general $\hat{\mu}\in \mathcal{P}_2(\hat{\mathbb{X}})$, \eqref{0531} turns out to be in this case as
\begin{equation}\label{eq:1214}
	\hat{P}_t^* \int_{\mathbb{T}^n}\hat{\mu}_{\vec{\mathbf{s}}}\nu(\mathrm{d}\vec{\mathbf{s}})=\int_{\mathbb{T}^n}(\hat{P}_t^*\hat{\mu}_{\vec{\mathbf{s}}}) \nu(\mathrm{d}\vec{\mathbf{s}}), \ \ \hat{\mu}_{\vec{\mathbf{s}}}\in \mathcal{P}_2(\hat{\mathbb{X}}) \ \text{ for all $\vec{\mathbf{s}}\in \mathbb{T}^n$ and } \ \nu\in\mathcal{P}(\mathbb{T}^n).
\end{equation}
In fact, let $\nu(\mathrm{d}\vec{\mathbf{s}})=\frac{1}{\tau_1\cdots\tau_n}\mathrm{d}\vec{\mathbf{s}}$ be the normalized Lebesgue measure on $\mathbb{T}^n$. Consider $\hat{\mu}_1, \hat{\mu}_2\in \mathcal{P}_2(\hat{\mathbb{X}})$ with $\hat{\mu}_i=\delta_{\vec{\mathbf{s}}_0}\times \tilde{\mu}^i_{\vec{\mathbf{s}}_0}$, $i=1,2$ for some $\vec{\mathbf{s}}_0\in \mathbb{T}^n$ and choose
\begin{equation*}
	\hat{\mu}_{\vec{\mathbf{s}}}=
   \begin{cases}
		\hat{\mu}_1, & \vec{\mathbf{s}}\in A\\
    \hat{\mu}_2, & \vec{\mathbf{s}}\notin A,
	 \end{cases}
   \ \ \text{ for some $A\in\mathcal{B}(\mathbb{T}^n)$ with $\nu(A)=1/2$.}
\end{equation*}
In this case, \eqref{eq:1214} turns out to be
\begin{equation}\label{eq:nonlinear}
            \hat{P}_t^*\Big(\frac{1}{2}(\hat{\mu}_1+\hat{\mu}_2)\Big)=\frac{1}{2}\hat{P}_t^*\hat{\mu}_1+\frac{1}{2}\hat{P}_t^*\hat{\mu}_2.
\end{equation}
However, \eqref{eq:nonlinear} does not hold in general since
\begin{equation*}
	\hat{P}_t^*\Big(\frac{1}{2}(\delta_{\vec{\mathbf{s}}_0}\times \tilde{\mu}^1_{\vec{\mathbf{s}}_0}+\delta_{\vec{\mathbf{s}}_0}\times\tilde{\mu}^2_{\vec{\mathbf{s}}_0})\Big)=\hat{P}_t^*\Bigl(\delta_{\vec{\mathbf{s}}_0}\times \Big(\frac{1}{2}\tilde{\mu}^1_{\vec{\mathbf{s}}_0}+\frac{1}{2}\tilde{\mu}^2_{\vec{\mathbf{s}}_0}\Big)\Bigr) =\delta_{T_t(\vec{\mathbf{s}}_0)}\times P^{\vec{\mathbf{s}}_0,*}(t,0)\Big(\frac{1}{2}\tilde{\mu}^1_{\vec{\mathbf{s}}_0}+\frac{1}{2}\tilde{\mu}^2_{\vec{\mathbf{s}}_0}\Big),
\end{equation*}
\begin{equation*}
	\frac{1}{2}\big(\hat{P}_t^*(\delta_{\vec{\mathbf{s}}_0}\times \tilde{\mu}^1_{\vec{\mathbf{s}}_0})+\hat{P}_t^*(\delta_{\vec{\mathbf{s}}_0}\times\tilde{\mu}^2_{\vec{\mathbf{s}}_0})\big) =\delta_{T_t(\vec{\mathbf{s}}_0)}\times \Big(\frac{1}{2}P^{\vec{\mathbf{s}}_0,*}(t,0)\tilde{\mu}^1_{\vec{\mathbf{s}}_0}+\frac{1}{2}P^{\vec{\mathbf{s}}_0,*}(t,0)\tilde{\mu}^2_{\vec{\mathbf{s}}_0}\Big),
\end{equation*}
and $P^{\vec{\mathbf{s}}_0,*}(t,0)$ is  nonlinear in general.
	\end{remark}

In fact, under the quasi-periodic condition, in addition to the existence of an entrance measure, the following theorem establishes the existence of an asymptotic quasi-periodic measure.

\begin{theorem}\label{Thm existence quasi-periodic measure}
		Under Assumptions \ref{A3}, \ref{A4}, \ref{Quasi-periodic condition} and \ref{One-sided Lip of tilde b}, equation \eqref{McKean-Vlasov SDE 2} has an asymptotic quasi-periodic measure $\mu$. Moreover, $\mu\in\mathcal{M}_p(\mathbb{R})$ for all $p\geq 2$.
	\end{theorem}
    
In Theorem \ref{Theorem existence of entrance measure} and \ref{Thm existence quasi-periodic measure}, we showed the existence of an entrance/asymptotic measure. A natural question to ask is whether they are unique. The following theorem suggests that there may be multiple entrance and asymptotic measures.

Set $f^a(t,x,\nu):=f(t,x+a,\nu^{-a}), \ f=b,\sigma$ and $\nu^{-a}(A):=\nu(A-a)$ for all $A\in\mathcal{B}(\mathbb{R}^d)$.
 
\begin{theorem}\label{Theorem nonuniqueness}
		Suppose that there exists $a\in \mathbb{R}^d$ such that $b^a,\sigma^a$ satisfy Assumptions \ref{A3}, \ref{A4} with parameters $\alpha_t^a, \beta_t^a,\gamma_t^a,c^a_{\sigma}$ and $\vartheta_t^a$ defined by \eqref{eq m}. Assume that there exists a measurable function $g_a: [0,\infty)^2\to \mathbb{R}$ such that for any $\nu\in \mathcal{P}_1(\mathbb{R}^d)$,
		\begin{equation}\label{0815-2}
			2\langle x,b^a(t,x,\nu)\rangle+\|\sigma^a(t,x,\nu)\|_2^2\leq -g_a(|x|,\|\nu\|_1), \ \text{ for all } t\in \mathbb{R}, x\in \mathbb{R}^d.
		\end{equation}
		If there exists $\theta_a>0$ such that $\inf_{0\leq w\leq \theta_a}g_a(z,w)=g_a(z,\theta_a)$ for all $z\geq 0$ and $g_a(\cdot,\theta_a)$ is continuous, convex, and
		\begin{equation}\label{0814-2}
			g_a(z,\theta_a)>0, \ \text{ for all } z>\theta_a,
		\end{equation}
		then \eqref{McKean-Vlasov SDE 2} has an entrance measure $\mu_{\cdot}$ in $\mathcal{M}^a_{2,\vartheta^a}\cap\mathcal{M}^a_{1,\theta_a}$, where
		\begin{equation*}
			\mathcal{M}^a_{2,\vartheta^a}:=\big\{\mu:\mathbb{R}\to \mathcal{P}_2(\mathbb{R}^d)\big| \|\mu_t^a\|_2^2\leq \vartheta^a_t, \text{ for all } t\in \mathbb{R}\big\}.
		\end{equation*}
		Consequently, if there exist $a_1,a_2\in \mathbb{R}^d$ such that $\theta_{a_1}+\theta_{a_2}<|a_1-a_2|$ and the above assumptions hold for both $a_1,a_2$ respectively, then equation \eqref{McKean-Vlasov SDE 2} has at least two distinct entrance measures $\mu^i\in \mathcal{M}^{a_i}_{2,\vartheta^{a_i}}\cap\mathcal{M}^{a_i}_{1,\theta_{a_i}}$, $i=1,2$. Moreover, if $b,\sigma$ are periodic/quasi-periodic, then equation \eqref{McKean-Vlasov SDE 2} has at least two distinct periodic/asymptotic quasi-periodic measures.
	\end{theorem}
	\begin{remark}
	It is now obvious that, under the assumptions of Theorem \ref{Theorem nonuniqueness}, equation \eqref{McKean-Vlasov SDE 2} has at least two distinct invariant measures in the sense of the double lift system.
	\end{remark}

\section{Existence of entrance measures: Proof of Theorem \ref{Theorem existence of entrance measure}}\label{section2}

In this section, we will give a proof of Theorem \ref{Theorem existence of entrance measure}. To show the existence of entrance measures of the McKean-Vlasov equation \eqref{McKean-Vlasov SDE 2}, the key idea is to consider the following SDE with a fixing $\mu\in \mathcal{M}_{2,loc}(\mathbb{R})$:
\begin{equation}\label{Classical SDE}
		\begin{cases}
			\mathrm{d}X_t=b(t,X_t,\mu_t)\mathrm{d}t+\sigma(t,X_t,\mu_t)\mathrm{d}W_t, \ t\geq s,\\
		  X_s=\xi.
		\end{cases}
	\end{equation}
Under Assumption \ref{A3}, it is well-known that for any $\xi\in L^2(\mathcal{F}_s)$, \eqref{Classical SDE} has a unique strong solution $X_t^{\mu,s,\xi}$ in $\mathcal{S}_{2,loc}([s,\infty))$.  
	
Recall that 
\begin{equation}\label{eq m}
		\vartheta_t=\int_{-\infty}^{t}e^{2\int_{u}^{t}(\alpha_r+\beta_r)\mathrm{d}r}(2\gamma_u+dc_{\sigma})\mathrm{d}u.
	\end{equation}
Under Assumption \ref{A4}, it is easy to check that $\mathcal{M}_{2,\vartheta}\subset \mathcal{M}_{2,loc}(\mathbb{R})$. Moreover, the function $\vartheta_{\cdot}$ satisfies the following identity.

	\begin{lemma}\label{Lemma 0320}
		Under Assumption \ref{A4}, we have
		\begin{equation*}
			\int_{-\infty}^{t}e^{2\int_{u}^{t}\alpha_r\mathrm{d}r}(2\beta_u\vartheta_u+2\gamma_u+dc_{\sigma})\mathrm{d}u=\vartheta_t, \ \text{ for all } \ t\in \mathbb{R}.
		\end{equation*}
	\end{lemma}
	\begin{proof}
		By the definition of $\vartheta$ (see \eqref{eq m}), we know that for all $t\in \mathbb{R}$,
		\begin{equation*}
			\begin{split}
				&\int_{-\infty}^{t}e^{2\int_{u}^{t}\alpha_r\mathrm{d}r}(2\beta_u\vartheta_u+2\gamma_u+dc_{\sigma})\mathrm{d}u\\
				=&\int_{-\infty}^{t}e^{2\int_{u}^{t}\alpha_r\mathrm{d}r}\bigg[2\beta_u\int_{-\infty}^{u}e^{2\int_{v}^{u}(\alpha_r+\beta_r)\mathrm{d}r}(2\gamma_v+dc_{\sigma})\mathrm{d}v+2\gamma_u+dc_{\sigma}\bigg]\mathrm{d}u\\
				=&\int_{-\infty}^{t}\int_{-\infty}^{u}2\beta_ue^{2\int_{v}^{t}\alpha_r\mathrm{d}r}e^{2\int_{v}^{u}\beta_r\mathrm{d}r}(2\gamma_v+dc_{\sigma})\mathrm{d}v\mathrm{d}u+\int_{-\infty}^{t}e^{2\int_{u}^{t}\alpha_r\mathrm{d}r}(2\gamma_u+dc_{\sigma})\mathrm{d}u\\
				=&\int_{-\infty}^{t}e^{2\int_{v}^{t}\alpha_r\mathrm{d}r}(2\gamma_v+dc_{\sigma})\biggl(\int_{v}^{t}2\beta_ue^{2\int_{v}^{u}\beta_r\mathrm{d}r}\mathrm{d}u\biggr)\mathrm{d}v+\int_{-\infty}^{t}e^{2\int_{u}^{t}\alpha_r\mathrm{d}r}(2\gamma_u+dc_{\sigma})\mathrm{d}u\\
				=&\int_{-\infty}^{t}e^{2\int_{v}^{t}\alpha_r\mathrm{d}r}(2\gamma_v+dc_{\sigma})\bigl(e^{2\int_{v}^{t}\beta_r\mathrm{d}r}-1\bigr)\mathrm{d}v+\int_{-\infty}^{t}e^{2\int_{u}^{t}\alpha_r\mathrm{d}r}(2\gamma_u+dc_{\sigma})\mathrm{d}u\\
				=&\int_{-\infty}^{t}e^{2\int_{v}^{t}(\alpha_r+\beta_r)\mathrm{d}r}(2\gamma_v+dc_{\sigma})\mathrm{d}v\\
				=&\vartheta_t.
			\end{split}
		\end{equation*}
	\end{proof}

    Now, fix an arbitrary $\mu\in \mathcal{M}_{2,\vartheta}$ in \eqref{Classical SDE}. Applying It{\^o}'s formula to the solution $|X_t^{\mu,s,0}|^2$ of \eqref{Classical SDE}, we have
    \begin{equation}\label{0320-1}
			\begin{split}
				\mathrm{d}|X_t^{\mu,s,0}|^2&=\bigl(2\langle X_t^{\mu,s,0}, b(t,X_t^{\mu,s,0},\mu_t)\rangle+\|\sigma(t,X_t^{\mu,s,0},\mu_t)\|^2\bigr)dt+2\langle X_t^{\mu,s,0}, \sigma(t,X_t^{\mu,s,0},\mu_t)\mathrm{d}W_t\rangle\\
				&\leq \bigl(2\alpha_t|X_t^{\mu,s,0}|^2+2\beta_t\|\mu_t\|_2^2+2\gamma_t+dc_{\sigma}\bigr)\mathrm{d}t+2\langle X_t^{\mu,s,0}, \sigma(t,X_t^{\mu,s,0},\mu_t)\mathrm{d}W_t\rangle\\
                &\leq \bigl(2\alpha_t|X_t^{\mu,s,0}|^2+2\beta_t\vartheta_t+2\gamma_t+dc_{\sigma}\bigr)\mathrm{d}t+2\langle X_t^{\mu,s,0}, \sigma(t,X_t^{\mu,s,0},\mu_t)\mathrm{d}W_t\rangle.
			\end{split}
	\end{equation}
    Adding factor $e^{-2\int_s^t\alpha_r{\rm d}r}$, we conclude that
    \begin{equation*}
				\mathrm{d}\big(e^{-2\int_s^t\alpha_r{\rm d}r}|X_t^{\mu,s,0}|^2\big)
				\leq e^{-2\int_s^t\alpha_r{\rm d}r}\bigl(2\beta_t\vartheta_t+2\gamma_t+dc_{\sigma}\bigr)\mathrm{d}t+2e^{-2\int_s^t\alpha_r{\rm d}r}\langle X_t^{\mu,s,0}, \sigma(t,X_t^{\mu,s,0},\mu_t)\mathrm{d}W_t\rangle.
	\end{equation*}
    Taking expectations and recalling Lemma \ref{Lemma 0320}, we know that
		\begin{equation*}
			\mathbf{E}\bigl[|X_t^{\mu,s,0}|^2\bigr]
			  \leq \int_{s}^{t}e^{2\int_{u}^{t}\alpha_r\mathrm{d}r}\bigl(2\beta_u\vartheta_u+2\gamma_u+dc_{\sigma}\bigr)\mathrm{d}u\leq \int_{-\infty}^{t}e^{2\int_{u}^{t}\alpha_r\mathrm{d}r}\bigl(2\beta_u\vartheta_u+2\gamma_u+dc_{\sigma}\bigr)\mathrm{d}u= \vartheta_t.
		\end{equation*}
        By \cite[(4.8) in Theorems 4.2]{Feng-Qu-Zhao2023}, we know that \eqref{Classical SDE} has a unique entrance measure $\rho$ such that
        \begin{equation*}
            \|\rho_t\|_2^2=\liminf_{s\to -\infty}\mathbf{E}\bigl[|X_t^{\mu,s,0}|^2\bigr]\leq \vartheta_t.
        \end{equation*}
        Hence, the unique entrance measure $\rho\in \mathcal{M}_{2,\vartheta}$.

      Now we consider the Markov transition function $P^{\mu}(t,s,x,\cdot)$ of equation \eqref{Classical SDE} on $(\mathbb{R}^d,\mathcal{B}(\mathbb{R}^d))$, i.e., $P^{\mu}(t,s,x,\cdot)=\mathscr{L}(X_t^{\mu,s,x})$, where $X_t^{\mu,s,x}$ is the unique solution of \eqref{Classical SDE} with initial condition $(s,x)$. Then the corresponding semigroups $P^{\mu}(t,s): \mathcal{B}_b(\mathbb{R}^d)\to \mathcal{B}_b(\mathbb{R}^d)$, $P^{\mu,*}(t,s): \mathcal{P}(\mathbb{R}^d)\to \mathcal{P}(\mathbb{R}^d)$ are given by
	\begin{equation*}
		P^{\mu}(t,s)f(x):=\int_{\mathbb{R}^d}f(y)P^{\mu}(t,s,x,\mathrm{d}y), \ t\geq s, f\in \mathcal{B}_b(\mathbb{R}^d),
	\end{equation*}
	\begin{equation*}
		P^{\mu,*}(t,s)\nu(A):=\int_{\mathbb{R}^d}P^{\mu}(t,s,x,A)\nu(\mathrm{d}x), \ t\geq s, \nu\in \mathcal{P}(\mathbb{R}^d), A\in \mathcal{B}(\mathbb{R}^d).
	\end{equation*}

   Summarizing all the above with \cite[Theorems 4.2]{Feng-Qu-Zhao2023}, we obtain the following theorem.
	\begin{theorem}\label{Theorem geometric converge to periodic measure}
		Assume Assumptions \ref{A3}, \ref{A4} hold. Then for any $\mu\in \mathcal{M}_{2,\vartheta}$, \eqref{Classical SDE} has a unique entrance measure $\rho$ in $\mathcal{M}_{2,\vartheta}$ and there exists a decreasing sequence $\{t_n, n\geq 0\}$ with $\lim_{n\to \infty}t_n=-\infty$ such that for all $x\in \mathbb{R}^d, t\in \mathbb{R}$
		\begin{equation*}
			\lim_{n\to \infty}\|P^{\mu}(t,t_n,x,\cdot)-\rho_t\|_{TV}= 0,
		\end{equation*}
		where $\|\cdot\|_{TV}$ denotes the total variation of a signed measure.
	\end{theorem}

	We also have the following lemma.
	\begin{lemma}\label{Lemma entrance measure in M_p}
		Assume Assumptions \ref{A3}, \ref{A4} hold. For any $\mu\in \mathcal{M}_{2,\vartheta}$, the unique entrance measure $\rho\in \mathcal{M}_{2,\vartheta}$ of \eqref{Classical SDE} satisfies: for any $2\leq p\leq 4$, we have $\rho_t\in \mathcal{P}_p(\mathbb{R}^d)$ and
		\begin{equation}\label{M_p estimate of entrance measure}
			\|\rho_t\|_p^p\leq \int_{-\infty}^{t}pe^{p\int_{u}^{t}\alpha_r\mathrm{d}r}\Big(\beta_u\vartheta_u+\gamma_u+\frac{p-1}{2}dc_{\sigma}\Big)\vartheta_u^{\frac{p}{2}-1}\mathrm{d}u=: a_p(t)<\infty.
		\end{equation}
	\end{lemma}

	\begin{proof}
		For any $p\geq 2$, applying It{\^o}'s formula to $|X_t^{\mu,s,0}|^p$ gives that
		\begin{equation}\label{0320-2}
			\begin{split}
			\mathrm{d}|X_t^{\mu,s,0}|^p
			  &\leq p\alpha_t |X_t^{\mu,s,0}|^p \mathrm{d}t+p\Big(\beta_t\vartheta_t+\gamma_t+\frac{p-1}{2}dc_{\sigma}\Big)|X_t^{\mu,s,0}|^{p-2}\mathrm{d}t\\
			  &\ \ \ +p|X_t^{\mu,s,0}|^{p-2}\langle X_t^{\mu,s,0}, \sigma(t,X_t^{\mu,s,0},\mu_t)\mathrm{d}W_t\rangle.
			\end{split}
		\end{equation}
		Adding factor $e^{-p\int_{s}^{t}\alpha_r\mathrm{d}r}$ in \eqref{0320-2} and taking expectations, we conclude that for all $2\leq p\leq 4$,
		\begin{equation*}
			\begin{split}
			  \mathbf{E}\bigl[|X_t^{\mu,s,0}|^p\bigr]
			  &\leq \int_{s}^{t}pe^{p\int_{u}^{t}\alpha_r\mathrm{d}r}\Big(\beta_u\vartheta_u+\gamma_u+\frac{p-1}{2}dc_{\sigma}\Big)\mathbf{E}\bigl[|X_u^{\mu,s,0}|^{p-2}\bigr]\mathrm{d}u\\
			  &\leq \int_{s}^{t}pe^{p\int_{u}^{t}\alpha_r\mathrm{d}r}\Big(\beta_u\vartheta_u+\gamma_u+\frac{p-1}{2}dc_{\sigma}\Big)\bigl(\mathbf{E}\bigl[|X_u^{\mu,s,0}|^2\bigr]\bigr)^{\frac{p}{2}-1}\mathrm{d}u\\
			  &\leq \int_{s}^{t}pe^{p\int_{u}^{t}\alpha_r\mathrm{d}r}\Big(\beta_u\vartheta_u+\gamma_u+\frac{p-1}{2}dc_{\sigma}\Big)\vartheta_u^{\frac{p}{2}-1}\mathrm{d}u\leq a_p(t).
			\end{split}
		\end{equation*}
		In particular, we know that $a_2(t)=\vartheta_t$.
		By Theorem \ref{Theorem geometric converge to periodic measure}, we have
		\begin{equation*}
			\begin{split}
				\|\rho_t\|_p^p&=\int_{\mathbb{R}^d}|x|^p\rho_t(\mathrm{d}x)=\lim_{N\to \infty}\int_{\mathbb{R}^d}(|x|^p\wedge N)\rho_t(\mathrm{d}x)\\
				&=\lim_{N\to \infty}\lim_{n\to \infty}\int_{\mathbb{R}^d}(|x|^p\wedge N)P^{\mu}(t,t_n,0,\mathrm{d}x)\\
				&\leq \limsup_{n\to \infty}\mathbf{E}\bigl[|X_t^{\mu,t_n,0}|^p\bigr]
				\leq a_p(t).
			\end{split}
		\end{equation*}
		Finally, we show that $a_p(t)<\infty$ for all $t\in \mathbb{R}$. Since $\limsup_{t\to-\infty}\vartheta_t<\infty$ and $\beta,\gamma$ are bounded, there exists $C_{p}(t)>0$ such that 
    \begin{equation*}
				a_p(t)=\int_{-\infty}^{t}pe^{p\int_{u}^{t}\alpha_r\mathrm{d}r}\Big(\beta_u\vartheta_u+\gamma_u+\frac{p-1}{2}dc_{\sigma}\Big)\vartheta_u^{\frac{p}{2}-1}\mathrm{d}u\leq C_{p}(t)e^{p\int_{u}^{t}\alpha_r\mathrm{d}r}\mathrm{d}u.
		\end{equation*}
    Note that the function $\beta_{\cdot}$ is nonnegative, we have
    \begin{equation*}
		\limsup_{T\to \infty}\frac{1}{T+t}\int_{-T}^{t}\alpha_r\mathrm{d}r\leq \limsup_{T\to \infty}\frac{1}{T+t}\int_{-T}^{t}(\alpha_r+\beta_r)\mathrm{d}r=\limsup_{T\to \infty}\frac{1}{T}\int_{-T}^{0}(\alpha_r+\beta_r)\mathrm{d}r<0,
	\end{equation*}
	so there exist $a>0, T_t>0$ such that $\int_s^t\alpha_r\mathrm{d}r\leq -a(t-s)$ for all $s\leq -T_t$. Hence, for any $t\in\mathbb{R}$,
\begin{equation*}
		\begin{split}
			a_p(t)&\leq C_p(t)\int_{-T_t}^{t}e^{p\int_{u}^{t}\alpha_r\mathrm{d}r}\mathrm{d}u+C_p(t)\int_{-\infty}^{-T_t}e^{-pa(t-u)}\mathrm{d}u\\
			&\leq C_p(t)\int_{-T_t}^{t}e^{p\int_{u}^{t}\alpha_r\mathrm{d}r}\mathrm{d}u+\frac{C_p(t)}{pa}e^{-pa(T_t+t)}<\infty.
		\end{split}
	\end{equation*}
The proof is complete.
	\end{proof}

	Now for any $p\geq 1$, let us consider the space $C(\mathbb{R};\mathcal{P}_p(\mathbb{R}^d))$ equipped with the topology generated by the following distance:
	\begin{equation}\label{eq:distance-d_2}
		d_p(\mu^1,\mu^2):=\sum_{n=1}^{\infty}\frac{1}{2^n}\frac{d_{p,n}(\mu^1,\mu^2)}{1+d_{p,n}(\mu^1,\mu^2)}, \ \  d_{p,n}(\mu^1,\mu^2):=\sup_{|t|\leq n}\mathcal{W}_p(\mu^1_t,\mu^2_t).
	\end{equation}
	It is easy to check that for any $p\geq 1$, $(C(\mathbb{R};\mathcal{P}_p(\mathbb{R}^d)), d_p)$ is a complete metric space. Moreover, a sequence $\{\mu^m\}_{m\geq 1}\subset C(\mathbb{R};\mathcal{P}_p(\mathbb{R}^d))$ converges to some $\mu \in C(\mathbb{R};\mathcal{P}_p(\mathbb{R}^d))$ in $d_p$ if and only if it converges to $\mu$ in $d_{p,n}$ for all $n\geq 1$.

	By Theorem \ref{Theorem geometric converge to periodic measure}, for any $\mu\in \mathcal{M}_{2,\vartheta}$, there is a unique entrance measure $\rho^{\mu}\in \mathcal{M}_{2,\vartheta}$ of SDE \eqref{Classical SDE}. Now we define a map
	$\Psi: \mathcal{M}_{2,\vartheta}\to \mathcal{M}_{2,\vartheta}$ by $\Psi (\mu):=\rho^{\mu}$. Set
	\begin{equation*}
		\widetilde{\mathcal{M}}:=\{\Psi(\mu): \mu\in \mathcal{M}_{2,\vartheta}\}.
	\end{equation*}
	We have the following lemma.
	\begin{lemma}\label{Lemma of invariant set}
		Under the same assumptions as those of Theorem \ref{Theorem existence of entrance measure}, $\widetilde{\mathcal{M}}\subset C(\mathbb{R};\mathcal{P}_2(\mathbb{R}^d))\cap \mathcal{M}_{2,\vartheta}$ and $\Psi(\widetilde{\mathcal{M}})\subset \widetilde{\mathcal{M}}$.
	\end{lemma}
	\begin{proof}
		First, $\widetilde{\mathcal{M}}\subset \mathcal{M}_{2,\vartheta}$ and $\Psi(\widetilde{\mathcal{M}})\subset \widetilde{\mathcal{M}}$ can be obtained immediately from Theorem \ref{Theorem geometric converge to periodic measure}. It remains to show that $\widetilde{\mathcal{M}}\subset C(\mathbb{R};\mathcal{P}_2(\mathbb{R}^d))$.

		For any $\mu\in \mathcal{M}_{2,\vartheta}$, we know that $\Psi(\mu)\in \mathcal{M}_{2,\vartheta}$. For any fixed $t_0\in \mathbb{R}$ and $s=t_0-1$, let $\xi_{s}\in L^2(\mathcal{F}_{s})$ such that $\mathscr{L}(\xi_s)=\Psi(\mu)_s$. It is easy to see that $\mathscr{L}(X_t^{\mu,s,\xi_s})=\Psi(\mu)_t$ for all $t\geq s$. Similar to \eqref{0320-1}, we know that for any $t\geq s$,
		\begin{equation*}
			\begin{split}
				\sup_{s\leq r\leq t}|X_r^{\mu,s,\xi_s}|^2&\leq |\xi_s|^2+\int_s^t\bigl(2\alpha_u^+|X_u^{\mu,s,\xi_s}|^2+2\beta_u\|\mu_u\|_2^2+2\gamma_u+dc_{\sigma}\bigr)\mathrm{d}u\\
				&\ \ \ \ +2\sup_{s\leq r\leq t}\bigg|\int_s^r\langle X_u^{\mu,s,\xi_s}, \sigma(u,X_u^{\mu,s,\xi_s},\mu_u)\mathrm{d}W_u\rangle\bigg|,
			\end{split}
		\end{equation*}
		then \eqref{Ineq of non-degenerate diffusion} and the Burkholder-Davis-Gundy inequality yield
		\begin{equation*}
			\begin{split}
				\mathbf{E}\Bigl[\sup_{s\leq r\leq t}|X_r^{\mu,s,\xi_s}|^2\Bigr]&\leq \mathbf{E}[|\xi_s|^2]+\int_s^t(2\beta_u\vartheta_u+2\gamma_u+dc_{\sigma})\mathrm{d}u+2C_{BDG}^2c_{\sigma}(t-s)\\
				&\ \ \ \ +\frac{1}{2}\mathbf{E}\Bigl[\sup_{s\leq r\leq t}|X_r^{\mu,s,\xi_s}|^2\Bigr]+\int_s^t2\alpha_u^+\mathbf{E}\Bigl[\sup_{s\leq r\leq u}|X_r^{\mu,s,\xi_s}|^2\Bigr]\mathrm{d}u.
			\end{split}
		\end{equation*}
		Then the Gronwall inequality gives
		\begin{equation}\label{0613-1}
			\mathbf{E}\Bigl[\sup_{s\leq r\leq t}|X_r^{\mu,s,\xi_s}|^2\Bigr]\leq 2\biggl(\mathbf{E}[|\xi_s|^2]+\int_s^t(2\beta_u\vartheta_u+2\gamma_u+dc_{\sigma})\mathrm{d}u+2C_{BDG}^2c_{\sigma}(t-s)\biggr)e^{2\int_s^t\alpha_r^+\mathrm{d}r}.
		\end{equation}
        Note that 
        \begin{equation*}
            \lim_{t\to t_0}X_t^{\mu,s,\xi_s}=X_{t_0}^{\mu,s,\xi_s}, \ \text{$\mathbf{P}$-a.s.}
        \end{equation*}
		It follows from the dominated convergence theorem that
		\begin{equation*}
			\lim_{t\to s}\mathcal{W}_2(\Psi(\mu)_t,\Psi(\mu)_s)\leq \lim_{t\to s}\mathbf{E}\bigl[|X_t^{\mu,s,\xi_s}-\xi_s|^2\bigr]=0.
		\end{equation*}
	\end{proof}

	\begin{lemma}\label{Lemma of compact set}
		The closure of $\widetilde{\mathcal{M}}$ in $C(\mathbb{R};\mathcal{P}_2(\mathbb{R}^d))$ is compact.
	\end{lemma}

	\begin{proof}
		It is enough to show that $\widetilde{\mathcal{M}}$ is pre-compact in $C(\mathbb{R};\mathcal{P}_2(\mathbb{R}^d))$. This is to prove that for any sequence $\{\rho^m\}_{m\geq 1}=\{\Psi(\mu^m)\}_{m\geq 1}\subset \widetilde{\mathcal{M}}$, there exists a subsequence $\{\rho^{m_k}\}_{k\geq 1}$ such that it converges in $C(\mathbb{R};\mathcal{P}_2(\mathbb{R}^d))$.

		By Lemma \ref{Lemma entrance measure in M_p}, we know that for any $2<p\leq 4$,
		\begin{equation*}
			\sup_{m\geq 1}\|\rho^m_t\|_p^p=\sup_{m\geq 1}\|\Psi(\mu^m)_t\|_p^p\leq a_p(t)<\infty.
		\end{equation*}
		Note that for any $t\in\mathbb{R}$, $\mathcal{F}_t:=\sigma\{B_u-B_v: u,v\leq t\}$ is the $\sigma$-algebra generated by the standard two-sided Brownian motion before $t$, so $(\Omega,\mathcal{F}_t,\mathbf{P})$ is a standard probability space. Hence, for an arbitrary fixed $n\geq 1$,  there exists a sequence of $\mathcal{F}_{-n}$-measurable random variables $\{\xi^{m,n}\}$ such that
		\begin{equation}\label{0713-1}
			\mathscr{L}(\xi^{m,n})=\rho^m_{-n}, \ \text{ and } \ \sup_{m\geq 1}\mathbf{E}\bigl[|\xi^{m,n}|^p\bigr]=\sup_{m\geq 1}\|\rho^m_{-n}\|_p^p\leq a_p(-n)<\infty, \ \text{ for all } \ 2<p\leq 4.
		\end{equation}
		Let $\{X_t^{\mu^m,-n,\xi^{m,n}}\}_{t\geq -n}$ be the unique solution to \eqref{Classical SDE} with initial condition $(-n,\xi^{m,n})$, i.e.,
		\begin{equation*}
			X_t^{\mu^m,-n,\xi^{m,n}}=\xi^{m,n}+\int_{-n}^{t}b\big(r,X_r^{\mu^m,-n,\xi^{m,n}},\mu_r^m\big)\mathrm{d}r+\int_{-n}^{t}\sigma\big(r,X_r^{\mu^m,-n,\xi^{m,n}},\mu_r^m\big)\mathrm{d}W_r.
		\end{equation*}
		It can be seen that 
		\begin{equation*}
			\mathscr{L}(X_t^{\mu^m,-n,\xi^{m,n}})=\rho^m_t, \ \text{ for all } \ t\geq -n.
		\end{equation*}
		By \eqref{0320-2} and similar to the proof of \eqref{0613-1}, we know that for any $2<p\leq 4$, there exists a constant $C=C(n,p,d,\alpha,\beta,\gamma,\vartheta,c_{\sigma})>0$ such that
		\begin{equation}\label{0713-3}
			\sup_{m\geq 1}\mathbf{E}\Bigl[\sup_{-n\leq t\leq n}\big|X_t^{\mu^m,-n,\xi^{m,n}}\big|^p\Bigr]\leq C\Big(1+\sup_{m\geq 1}\mathbf{E}\bigl[|\xi^{m,n}|^p\bigr]\Big)\leq C(1+a_p(-n)).
		\end{equation}
		On the other hand, the Burkholder-Davis-Gundy inequality and the boundedness of $\sigma$ imply
		\begin{equation}\label{my1128}
			\begin{split}
				\mathbf{E}\biggl[\sup_{-n\leq t\leq n}\bigg|\int_{-n}^{t}\sigma\big(r,X_r^{\mu^m,-n,\xi^{m,n}},\mu_r^m\big)\mathrm{d}W_r\bigg|^p\biggr]
				&\leq C(p,n,d,c_{\sigma}),
			\end{split}
		\end{equation}
		then we conclude
		\begin{equation}\label{0613-3}
		  \mathbf{E}\biggl[\sup_{-n\leq t\leq n}\bigg|\int_{-n}^{t}b\big(r,X_r^{\mu^m,-n,\xi^{m,n}},\mu_r^m\big)\mathrm{d}r\bigg|^p\biggr]\leq C(n,p,d,\alpha,\beta,\gamma,\vartheta,c_{\sigma})\bigl(1+a_p(-n)\bigr)=:\ell_p,
		\end{equation}
		  from \eqref{my1128} and the fact that
		\begin{equation*}
		  \begin{split}
			  &\ \ \ \sup_{-n\leq t\leq n}\bigg|\int_{-n}^{t}b\big(r,X_r^{\mu^m,-n,\xi^{m,n}},\mu_r^m\big)\mathrm{d}r\bigg|^p\\
			  &\leq C(p)\biggl(|\xi^{m,n}|^p+\sup_{-n\leq t\leq n}|X_t^{\mu^m,-n,\xi^{m,n}}|^p+\sup_{-n\leq t\leq n}\bigg|\int_{-n}^{t}\sigma\big(r,X_r^{\mu^m,-n,\xi^{m,n}},\mu_r^m\big)\mathrm{d}W_r\bigg|^p\biggr).
		  \end{split}
		\end{equation*}
		Let
        \begin{equation*}
			\tau_n^{m,R}:=\inf\{-n\leq t\leq n: |X_t^{\mu^m,-n,\xi^{m,n}}\big|\geq R\}.
		\end{equation*}
		Then \eqref{0713-3} and the Chebyshev's inequality yield that for any $2<p\leq 4$,
		\begin{equation}\label{0613-4}
			\mathbf{P}\{\tau_n^{m,R}<n\}=\mathbf{P}\Big\{\sup_{-n\leq t\leq n}|X_t^{\mu^m,-n,\xi^{m,n}}\big|\geq R\Big\}\leq \frac{\ell_p}{R^p}.
		\end{equation}
		Since $b$ satisfies \eqref{Ineq polynomial growth}, then for any $-n\leq s\leq t\leq n$ and $R>\sqrt{\sup_{t\in[-n,n]}\vartheta_t}=:R_0$, we have
        \begin{equation}\label{0613-5}
			\begin{split}
				\bigg|\int_s^tb\big(r,X_r^{\mu^m,-n,\xi^{m,n}},\mu_r^m\big)\mathrm{d}r\bigg|&=\bigg|\int_s^tb\big(r,X_r^{\mu^m,-n,\xi^{m,n}},\mu_r^m\big)\mathrm{d}r\bigg| 1_{\{\tau_n^{m,R}=n\}}\\
				&\ \ \ \ +\bigg|\int_s^tb\big(r,X_r^{\mu^m,-n,\xi^{m,n}},\mu_r^m\big)\mathrm{d}r\bigg|1_{\{\tau_n^{m,R}<n\}}\\
				&\leq 2L(1+R^{\kappa})|t-s|\\
				&\ \ \ \ +2\sup_{-n\leq t\leq n}\bigg|\int_{-n}^tb\big(r,X_r^{\mu^m,-n,\xi^{m,n}},\mu_r^m\big)\mathrm{d}r\bigg|1_{\{\tau_n^{m,R}<n\}}.
			\end{split}
		\end{equation}
		Then it follows from \eqref{0613-3}, \eqref{0613-4}, \eqref{0613-5} and H\"older inequality that for any $R>R_0$, $2<p\leq 4$,
		\begin{equation*}
			\begin{split}
			  \mathbf{E}\biggl[\sup_{-n\leq s\leq t\leq n, |t-s|\leq h}\bigg|\int_{s}^{t}b\big(r,X_r^{\mu^m,-n,\xi^{m,n}},\mu_r^m\big)\mathrm{d}r\bigg|\biggr]\leq 2L(1+R^{\kappa})h+\frac{2\ell_p}{R^{p-1}}.
			\end{split}
		\end{equation*}
		Hence, for any $\epsilon>0$ and $2<p\leq 4$,
		\begin{equation*}
			\lim_{h\downarrow 0}\sup_{m\geq 1}\mathbf{P}\bigg\{\sup_{-n\leq s\leq t\leq n, |t-s|\leq h}\bigg|\int_{s}^{t}b\big(r,X_r^{\mu^m,-n,\xi^{m,n}},\mu_r^m\big)\mathrm{d}r\bigg|>\epsilon\bigg\}\leq \frac{2\ell_p}{R^{p-1}\epsilon},
		\end{equation*}
		for all $R>R_0$ and therefore equal to $0$ by letting $R\to \infty$. Note also that for any $p>2$,
		\begin{equation*}
			\mathbf{E}\biggl[\bigg|\int_{s}^{t}\sigma\big(r,X_r^{\mu^m,-n,\xi^{m,n}},\mu_r^m\big)\mathrm{d}W_r\bigg|^p\biggr]\leq C(p)(dc_{\sigma})^{\frac{p}{2}}|t-s|^{\frac{p}{2}}.
		\end{equation*}
		Then Theorem \ref{Thm I-4.3} shows that
		\begin{equation}
			\lim_{h\downarrow 0}\sup_{m\geq 1}\mathbf{P}\bigg\{\sup_{-n\leq s\leq t\leq n, |t-s|\leq h}\bigg|\int_{s}^{t}\sigma\big(r,X_r^{\mu^m,-n,\xi^{m,n}},\mu_r^m\big)\mathrm{d}W_r \bigg|>\epsilon\bigg\}=0.
		\end{equation}
		Thus
		\begin{equation}\label{0320-3}
			\lim_{h\downarrow 0}\sup_{m\geq 1}\mathbf{P}\bigg\{\sup_{-n\leq s\leq t\leq n, |t-s|\leq h}\big|X_t^{\mu^m,-n,\xi^{m,n}}-X_s^{\mu^m,-n,\xi^{m,n}}\big|>\epsilon\bigg\}=0.
		\end{equation}
		Then it follows from \eqref{0713-1}, \eqref{0320-3} and Theorem \ref{Thm I-4.2} that there exist a subsequence $\{m_k\}_{k\geq 1}$, a probability space $(\tilde{\Omega},\tilde{\mathcal{F}},\tilde{\mathbf{P}})$ and $d$-dimensional continuous processes $\{\tilde{X}_{m_k}\}_{k\geq 1}$ and $\tilde{X}$ on it such that
		\begin{itemize}
			\item [\emph{(i)}] $\mathscr{L}\big(\tilde{X}_{m_k}(t)\big)=\mathscr{L}\big(X^{\mu^{m_k},-n,\xi^{m_k,n}}_t\big)=\rho^{m_k}_t, \ k\geq 1$, for all $t\in [-n,n]$;
			\item [\emph{(ii)}] the sequence $\{\tilde{X}_{m_k}, k\geq 1\}$ uniformly converges to $\tilde{X}$ on $[-n,n]$, $\tilde{\mathbf{P}}$-almost surely.
		\end{itemize}
		Moreover, \eqref{0713-3} indicates that $\{\tilde{X}_{m_k}(t)\}_{k\geq 1}$ is $L^2(\Omega;C([-n,n];\mathbb{R}^d))$-uniformly integrable, then the Vitali convergence theorem yields that
		\begin{equation*}
			\lim_{k,l\to \infty}d_{2,n}(\rho^{m_k},\rho^{m_l})\leq \lim_{k,l\to \infty}\tilde{\mathbf{E}}\Bigl[\sup_{-n\leq t\leq n}|\tilde{X}_{m_k}(t)-\tilde{X}_{m_l}(t)|^2\Bigr]=0.
		\end{equation*}
		Hence $\{\rho^{m_k}\}$ is a Cauchy sequence in $C(\mathbb{R};\mathcal{P}_2(\mathbb{R}^d))$ with respect to the semi-metric $d_{2,n}$. Now we can use the diagonal method to choose a subsequence, denote by $\{\rho^{m_k}\}_{k\geq 1}$, which is a Cauchy sequence with respect to all semi-metrics $\{d_{2,n}, n\geq 1\}$. Hence it is a Cauchy sequence in $C(\mathbb{R};\mathcal{P}_2(\mathbb{R}^d))$ with respect to the metric $d_2$ by the definition of $d_2$.
	\end{proof}

   Now, we give the proof of Theorem \ref{Theorem existence of entrance measure}.

	\begin{proof}[Proof of Theorem \ref{Theorem existence of entrance measure}]
		If the map $\Psi$ has a fixed point $\mu$ in $\mathcal{M}_{2,\vartheta}$, i.e., $\mu=\Psi(\mu)$, it is easy to check that $\mu$ is an entrance measure of McKean-Vlasov SDE \eqref{McKean-Vlasov SDE 2}. Thus it remains to prove the existence of a fixed point of $\Psi$.

		Let
		\begin{equation*}
			\mathcal{MS}_2:=\bigg\{\mu: \mu \text{ is a finite signed measure on } \mathbb{R}^d \text{ such that } \int_{\mathbb{R}^d}|x|^2|\mu|(\mathrm{d}x)<\infty\bigg\}.
		\end{equation*}
		According to \cite{rachev1984}, $\mathcal{MS}_2$ is a Hausdorff locally convex space under the following Kantorovich-Rubinstein metric
		\begin{equation*}
			\mathbb{W}_2(\mu,\nu):=\biggl(\sup_{(f,g)\in Lip_2}\biggl(\int_{\mathbb{R}^d}f(x)\mu(\mathrm{d}x)+\int_{\mathbb{R}^d}g(x)\nu(\mathrm{d}x)\biggr)\biggr)^{\frac{1}{2}},
		\end{equation*}
		where 
		\begin{equation*}
			Lip_2:=\{(f,g)| f,g: \mathbb{R}^d \to \mathbb{R} \text{ are bounded Lipschitz continuous such that } f(x)+g(y)\leq |x-y|^2\}.
		\end{equation*}
		By Theorem 1 in \cite{rachev1984}, we know that for any $\mu,\nu\in \mathcal{P}_2(\mathbb{R}^d)$,
		\begin{equation*}
			\mathbb{W}_2(\mu,\nu)=\mathcal{W}_2(\mu,\nu).
		\end{equation*}
		It can be seen that $\mathcal{P}_2(\mathbb{R}^d)$ is a closed convex subset of $\mathcal{MS}_2$. Hence $C(\mathbb{R};\mathcal{P}_2(\mathbb{R}^d))$ is a closed convex subset of $C(\mathbb{R};\mathcal{MS}_2)$. It is also easy to check that $C(\mathbb{R};\mathcal{MS}_2)$ is a Hausdorff locally convex space.

		Denote by $\overline{{\rm co}}(\widetilde{\mathcal{M}})$ the closed convex hull of $\widetilde{\mathcal{M}}$ in $C(\mathbb{R};\mathcal{P}_2(\mathbb{R}^d))$, hence $\overline{{\rm co}}(\widetilde{\mathcal{M}})$ is convex in $C(\mathbb{R};\mathcal{MS}_2)$. Note that $C(\mathbb{R};\mathcal{P}_2(\mathbb{R}^d))$ is a complete metric space and $\widetilde{\mathcal{M}}$ is pre-compact in $C(\mathbb{R};\mathcal{P}_2(\mathbb{R}^d))$ (see Lemma \ref{Lemma of compact set}), then by Theorem 5.35 in \cite{Aliprantis-Border2006}, we know that $\overline{{\rm co}}(\widetilde{\mathcal{M}})$ is compact in $C(\mathbb{R};\mathcal{P}_2(\mathbb{R}^d))$ and therefore compact in $C(\mathbb{R};\mathcal{MS}_2)$.

		It is easy to show that $\overline{{\rm co}}(\widetilde{\mathcal{M}})\subset \mathcal{M}_{2,\vartheta}$ is nonempty and $\Psi(\overline{{\rm co}}(\widetilde{\mathcal{M}}))\subset \overline{{\rm co}}(\widetilde{\mathcal{M}})$. Since $\Psi$ is a single-valued map, it is enough to show that the graph $\{(\mu,\rho)\in \overline{{\rm co}}(\widetilde{\mathcal{M}})\times \overline{{\rm co}}(\widetilde{\mathcal{M}}): \rho=\Psi(\mu)\}$ is closed in $\overline{{\rm co}}(\widetilde{\mathcal{M}})\times \overline{{\rm co}}(\widetilde{\mathcal{M}})$, i.e., for any sequence $\{(\mu^m,\rho^m): \rho^m=\Psi(\mu^m)\}_{m\geq 1}\subset \overline{{\rm co}}(\widetilde{\mathcal{M}})\times \overline{{\rm co}}(\widetilde{\mathcal{M}})$, if $d_2(\mu^m,\mu)\to 0$ and $d_2(\rho^m,\rho)\to 0$ as $m\to \infty$, then $\mu,\rho\in \overline{{\rm co}}(\widetilde{\mathcal{M}})$ and $\rho=\Psi(\mu)$. Note that $\overline{{\rm co}}(\widetilde{\mathcal{M}})$ is closed, so the first part of this claim is obvious. It remains to show that $\rho=\Psi(\mu)$.

		Note that for any $s\in \mathbb{R}$, there exists a sequence of $\mathcal{F}_s$-measurable random variables $\{\xi^m_s\}$ such that $\mathscr{L}(\xi^m_s)=\rho^m_s$. Let $\{X_t^{\mu^m,s,\xi^m_s}, t\geq s\}$ be the unique solution to \eqref{Classical SDE} with initial condition $(s,\xi^m_s)$. Then it follows that $\mathscr{L}(X_t^{\mu^m,s,\xi^m_s})=\rho^m_t$ for any $t\geq s$.

		For any $T>0$, set
		\begin{equation*}
			\eta^m_t=\biggl(X_t^{\mu^m,s,\xi^m_s}, \int_{s}^{t}b\Bigl(r,X_r^{\mu^m,s,\xi^m_s},\mu^m_r\Bigr)\mathrm{d}r, \int_{s}^{t}\sigma\Bigl(r,X_r^{\mu^m,s,\xi^m_s},\mu^m_r\Bigr)\mathrm{d}W_r, W_t\biggr), \ s\leq t\leq s+T.
		\end{equation*}
		Similar to \eqref{0713-1}-\eqref{0320-3} in the proof of Lemma \ref{Lemma of compact set}, we know that for any $2<p\leq 4$,
		\begin{equation}\label{0713-4}
			\sup_{m\geq 1}\mathbf{E}\Bigl[\sup_{s\leq t\leq s+T}|\eta^m_t|^p\Bigr]\leq \ell_p<\infty,
		\end{equation}
		and for any $\epsilon>0$,
		\begin{equation*}
			\lim_{h\downarrow 0}\sup_{m\geq 1}\mathbf{P}\bigg\{\sup_{s\leq t_1,t_2\leq s+T, |t_1-t_2|\leq h}|\eta^m_{t_1}-\eta^m_{t_2}|>\epsilon\bigg\}=0.
		\end{equation*}
		Then Theorem \ref{Thm I-4.2} yields that there exist a subsequence $\{m_k\}_{k\geq 1}$, a probability space $(\tilde{\Omega},\tilde{\mathcal{F}},\tilde{\mathbf{P}})$ and continuous processes 
		\begin{equation*}
			\tilde{\eta}^{m_k}_t=\Bigl(\tilde{X}^{m_k}_t, \tilde{B}^{m_k}_t, \tilde{\Sigma}^{m_k}_t, \tilde{W}^{m_k}_t\Bigr), k=1,2,\cdots, \ \tilde{\eta}_t=\bigl(\tilde{X}_t, \tilde{B}_t, \tilde{\Sigma}_t, \tilde{W}_t\bigr), \ s\leq t\leq s+T,
		\end{equation*}
		defined on $(\tilde{\Omega},\tilde{\mathcal{F}},\tilde{\mathbf{P}})$ such that
	    \begin{itemize}
		\item [(i)] $\mathscr{L}(\tilde{\eta}^{m_k})=\mathscr{L}(\eta^{m_k}), \ k\geq 1$;
		\item [(ii)] the sequence $\{\tilde{\eta}^{m_k}, k\geq 1\}$ converges to $\tilde{\eta}$ uniformly on $[s,s+T]$, $\tilde{\mathbf{P}}$-almost surely.
	    \end{itemize}
		Let $\tilde{\mathcal{F}}^{m_k}_t$, $\tilde{\mathcal{F}}_t$ be the filtration generated by $\{\tilde{\eta}^{m_k}_r, s\leq r\leq t\}$, $\{\tilde{\eta}_r, s\leq r\leq t\}$, respectively, completed by all null sets. It is easy to see that $\{\tilde{W}^{m_k}_t-\tilde{W}^{m_k}_s\}_{s\leq t\leq s+T}$ and $\{\tilde{W}_t-\tilde{W}_s\}_{s\leq t\leq s+T}$ are standard $\tilde{\mathcal{F}}^{m_k}_t$- and $\tilde{\mathcal{F}}_t$-Brownian motion. Note that $\eta^{m_k}=(X^{m_k},B^{m_k},\Sigma^{m_k},W)$ where for all $s\leq t\leq s+T$,
	\begin{equation}\label{0614-2}
		X^{m_k}_t=X^{\mu^{m_k},s,\xi_s^{m_k}}_t, \ \ B^{m_k}_t=\int_{s}^{t}b\bigl(r,X_r^{m_k},\mu^{m_k}_r\bigr)\mathrm{d}r, \ \ \Sigma^{m_k}_t=\int_{s}^{t}\sigma\bigl(r,X_r^{m_k},\mu^{m_k}_r\bigr)\mathrm{d}W_r.
	\end{equation}
	Then $\mathscr{L}(\tilde{\eta}^{m_k})=\mathscr{L}(\eta^{m_k})$ implies that 
	\begin{equation*}
		\begin{split}
			\mathscr{L}\biggl(\tilde{B}^{m_k}_{\cdot}-\int_{s}^{\cdot}b\bigl(r,\tilde{X}_r^{m_k},\mu^{m_k}_r\bigr)\mathrm{d}r\biggr)
			&=\mathscr{L}\biggl(B^{m_k}_{\cdot}-\int_{s}^{\cdot}b\bigl(r,X_r^{m_k},\mu^{m_k}_r\bigr)\mathrm{d}r\biggr)=\delta_0.
		\end{split}
	\end{equation*}
	Hence
	\begin{equation*}
		\mathbf{E}\biggl[\sup_{t\in [s,s+T]}\bigg|\tilde{B}^{m_k}_{t}-\int_{s}^{t}b\bigl(r,\tilde{X}_r^{m_k},\mu^{m_k}_r\bigr)\mathrm{d}r\bigg|\biggr]=0,
	\end{equation*}
	which yields that $\tilde{\mathbf{P}}$-a.s.
	\begin{equation*}
		\tilde{B}^{m_k}_{t}=\int_{s}^{t}b\bigl(r,\tilde{X}_r^{m_k},\mu^{m_k}_r\bigr)\mathrm{d}r, \ \text{ for all } s\leq t\leq s+T.
	\end{equation*}
	Similarly, we also have $\tilde{\mathbf{P}}$-a.s.
	\begin{equation*}
		\tilde{\Sigma}^{m_k}_t=\int_{s}^{t}\sigma\bigl(r,\tilde{X}_r^{m_k},\mu^{m_k}_r\bigr)\mathrm{d}\tilde{W}^{m_k}_r, \ \text{ for all } s\leq t\leq s+T.
	\end{equation*}
	Then again $\mathscr{L}(\tilde{\eta}^{m_k})=\mathscr{L}(\eta^{m_k})$ shows that $\tilde{\mathbf{P}}$-a.s. 
	\begin{equation}\label{0614-1}
		\tilde{X}_t^{m_k}=\tilde{X}_s^{m_k}+\int_{s}^{t}b\bigl(r,\tilde{X}_r^{m_k},\mu^{m_k}_r\bigr)\mathrm{d}r+\int_{s}^{t}\sigma\bigl(r,\tilde{X}_r^{m_k},\mu^{m_k}_r\bigr)\mathrm{d}\tilde{W}^{m_k}_r,  \ s\leq t\leq s+T.
	\end{equation}

	In the following, we will show that every term in \eqref{0614-1} has a $L^2$-limit. Note that $\tilde{X}^{m_k}$ converge to $\tilde{X}$ uniformly on $[s,s+T]$, $\tilde{\mathbf{P}}$-a.s., so \eqref{0713-4} and the Vitali convergence theorem imply
	\begin{equation}
		\lim_{k\to \infty}\tilde{\mathbf{E}}\bigl[|\tilde{X}^{m_k}_t-\tilde{X}_t|^2\bigr]=0.
	\end{equation}
	Moreover, since $\lim_{m\to \infty}\sup_{s\leq t\leq s+T}\mathcal{W}_2(\mu^m_t,\mu_t)=0$, the continuity of $b,\sigma$ yields $\tilde{\mathbf{P}}$-a.s.
	\begin{equation*}
		f\bigl(r,\tilde{X}_r^{m_k},\mu^{m_k}_r\bigr)\xrightarrow{k\to \infty} f\bigl(r,\tilde{X}_r,\mu_r\bigr),  \text{ uniformly on }  [s,s+T], \ f=b,\sigma.
	\end{equation*}
	Hence $\tilde{\mathbf{P}}$-a.s., for all $s\leq t\leq s+T$
	\begin{equation*}
		\int_{s}^{t}b\bigl(r,\tilde{X}_r^{m_k},\mu^{m_k}_r\bigr)\mathrm{d}r \xrightarrow{k\to \infty} \int_{s}^{t}b\bigl(r,\tilde{X}_r,\mu_r\bigr)\mathrm{d}r,
	\end{equation*}
	and again \eqref{0713-4} together with Vitali convergence theorem gives
	\begin{equation}
\int_{s}^{t}b\bigl(r,\tilde{X}_r^{m_k},\mu^{m_k}_r\bigr)\mathrm{d}r \xrightarrow{k\to \infty} \int_{s}^{t}b\bigl(r,\tilde{X}_r,\mu_r\bigr)\mathrm{d}r \ \text{ in } \ L^2(\tilde{\Omega}).
	\end{equation}
	For the integral of $\sigma$ term on the right-hand side of \eqref{0614-1}, \eqref{0713-4} and Vitali convergence theorem also lead that
	\begin{equation*}
		\lim_{k\to \infty}\tilde{\mathbf{E}}\Bigl[\big\|\sigma\bigl(r,\tilde{X}_r^{m_k},\mu^{m_k}_r\bigr)-\sigma\bigl(r,\tilde{X}_r,\mu_r\bigr)\big\|^2\Bigr]=0, \ \text{ for all } \ s\leq r\leq s+T,
	\end{equation*}
	and the dominated convergence theorem yields that for all $s\leq r\leq s+T$,
	\begin{equation*}
		\lim_{k\to \infty}\int_s^t\tilde{\mathbf{E}}\Bigl[\big\|\sigma\bigl(r,\tilde{X}_r^{m_k},\mu^{m_k}_r\bigr)-\sigma\bigl(r,\tilde{X}_r,\mu_r\bigr)\big\|^2\Bigr]\mathrm{d}r=0.
	\end{equation*}
	Then \eqref{0713-4} and Theorem \ref{Theorem of Skorohod} show that
	\begin{equation}
		\int_{s}^{t}\sigma\bigl(r,\tilde{X}_r^{m_k},\mu^{m_k}_r\bigr)\mathrm{d}\tilde{W}^{m_k}_r \xrightarrow{k\to \infty} \int_{s}^{t}\sigma\bigl(r,\tilde{X}_r,\mu_r\bigr)\mathrm{d}\tilde{W}_r \ \text{ in } \ L^2(\tilde{\Omega}).
	\end{equation}
	Now taking the $L^2$-limits to both sides of \eqref{0614-1} for all $t\in \mathbb{Q}\cap [s,s+T]$, where $\mathbb{Q}$ stands for the set of rational numbers in $\mathbb{R}$, we have $\tilde{\mathbf{P}}-a.s.,$
		\begin{equation}\label{0713-5}
			\tilde{X}_t=\tilde{X}_s+\int_{s}^{t}b\big(r,\tilde{X}_r,\mu_r\big)\mathrm{d}r+\int_{s}^{t}\sigma\big(r,\tilde{X}_r,\mu_r\big)\mathrm{d}\tilde{W}_r, \ t\in \mathbb{Q}\cap [s,s+T].
		\end{equation}
		Since $\tilde{\eta}$ is a continuous process, we have \eqref{0713-5} holds for all $t\in [s,s+T]$, $\tilde{\mathbf{P}}-a.s.$.
		Note that $\mathscr{L}(\tilde{X}^{m_k}_t)=\mathscr{L}(X_t^{\mu^{m_k},s,\xi^{m_k}_s})=\rho^{m_k}_t$ for all $s\leq t$, so \eqref{0713-4} and the Vitali convergence theorem imply,
		\begin{equation*}
			\lim_{k\to \infty}\mathcal{W}_2\big(\rho^{m_k}_t,\mathscr{L}(\tilde{X}_t)\big)\leq \lim_{k\to \infty}\tilde{\mathbf{E}}\bigl[|\tilde{X}^{m_k}_t-\tilde{X}_t|^2\bigr]=0, \ s\leq t\leq s+T.
		\end{equation*}
		Thus $\mathscr{L}(\tilde{X}_t)=\rho_t$ for all $s\leq t\leq s+T$ since $d_2(\rho^m,\rho)\to 0$ as $m\to \infty$. Therefore, \eqref{0713-5} shows that $P^{\mu,*}(t,s)\rho_s=\rho_t$ for all $s\leq t\leq s+T$. Hence $\rho=\Psi(\mu)$ follows from the arbitrary choices of $s\in \mathbb{R}, T>0$.

		Finally, it follows from Theorem \ref{Lemma Kakutani's fixed point theorem} and the fact that $\Psi$ is single-valued that there exists $\mu\in\overline{{\rm co}}(\widetilde{\mathcal{M}})$ such that $\mu=\Psi(\mu)$.
	\end{proof}

\begin{remark}\label{rem:bounded-of-entrance-measure}
	If Assumption \ref{A4} (ii) is replaced by the  condition:
\begin{equation}\label{eq:new-replace-A2(ii)}
	\text{there exist $l>0, a>0$ such that $\frac{1}{l}\int_{t}^{t+l}(\alpha_r+\beta_r)\mathrm{d}r\leq -a$ for all $t\in \mathbb{R}$,}
\end{equation}
then, as noted in Remark \ref{remark 0522} (ii), Assumption \ref{A4} (ii) is fulfilled, and $\vartheta$ is bounded.

Let $\mu = \Psi(\mu)\in \mathcal{M}_{2,\vartheta}$ be an entrance measure of equation \eqref{McKean-Vlasov SDE 2}. Since $\vartheta$ is bounded, $\mu$ is in $\mathcal{M}_2(\mathbb{R})$. Furthermore, according to Lemma \ref{Lemma entrance measure in M_p} and the process used to demonstrate the boundedness of $\vartheta$ in Remark \ref{remark 0522} (ii), we conclude that for any $0 \leq p \leq 4$, there exists $C_p > 0$ such that $\|\mu_t\|_p^p\leq a_p(t)\leq C_p$ for all $t\in\mathbb{R}$.

 For any $k\geq 2$ and $2k< p\leq 2k+2$, if we define $a_p(t)$ recursively by
\begin{equation*}
	a_p(t)=\int_{-\infty}^{t}pe^{p\int_{u}^{t}\alpha_r\mathrm{d}r}\Big(\beta_u\vartheta_u+\gamma_u+\frac{p-1}{2}dc_{\sigma}\Big)a_{p-2}(u)\mathrm{d}u,
\end{equation*}
we can demonstrate that $a_p(t)$ is bounded by some $C_p > 0$ for all $p \geq 2$. Then, following the proof of Lemma \ref{Lemma entrance measure in M_p}, we conclude that for all $p\geq 2$ and all $t\in\mathbb{R}$, $\|\mu_t\|_p^p\leq a_p(t)\leq C_p$ for some $C_p>0$.
Thus, we show that any entrance measure $\mu\in \mathcal{M}_{2,\vartheta}$ is in $\mathcal{M}_p(\mathbb{R})$ for all $p\geq 2$.
\end{remark}

\begin{remark}
    If Assumption \ref{A4} is replaced by the following
    \[
    2\big\langle x-y, b(t,x,\mu)-b(t,y,\mu)\big\rangle+\|\sigma(t,x,\mu)-\sigma(t,y,\nu)\|_2^2\leq C_1(t)|x-y|^2+C_2(t)\mathcal{W}_2(\mu,\nu),
    \]
    where $C_1(t), C_2(t)$ are bounded functions such that
    \[
    \sup_{t\in\mathbb{R}}\frac{1}{l}\int_t^{t+l}(C_1(r)+C_2(r)){\rm d}r\leq -2\lambda, \ \text{ for some } \ l,\lambda>0.
    \]
    According to \cite[Theorem 3.1]{wang2018}, there exists $C>0$ such that
    \begin{equation}\label{eq:exponential-contraction}
        \mathcal{W}_2(P^*(t,s)\mu,P^*(t,s)\nu)\leq Ce^{-\lambda |t-s|}\mathcal{W}_2(\mu,\nu), \ \text{ for all } \ t\geq s, \ \mu,\nu\in\mathcal{P}_2(\mathbb{R}^d).
    \end{equation}
    In this case, it is easy to check that for any fixed $t\in\mathbb{R}, \ \mu\in \mathcal{P}_2(\mathbb{R}^d)$, $\{P^*(t,s)\mu\}_{s\leq t}$ is a Cauchy process as $s\to -\infty$. Let $\mu_t:=\lim_{s\to\infty}P^*(t,s)\delta_0$. We can conclude that $\{\mu_t\}_{t\in\mathbb{R}}$ is the unique entrance measure, and it holds that 
    \begin{equation*}
        \mathcal{W}_2(P^*(t,s)\mu,\mu_t)\leq Ce^{-\lambda |t-s|}\mathcal{W}_2(\mu,\mu_s), \ \text{ for all } \ t\geq s, \ \mu\in\mathcal{P}_2(\mathbb{R}^d).
    \end{equation*}
\end{remark}

In the rest of this section, we explore the joint continuity of the map $P^*(t,s): \mathcal{P}_2(\mathbb{R}^d)\to \mathcal{P}_2(\mathbb{R}^d)$ as defined by \eqref{eq:def-P^*}. To be precise, let $\Delta:=\{(t,s)\in \mathbb{R}^2: t\geq s\}\subset \mathbb{R}^2$. We obtain the following result.
	\begin{proposition}\label{Joint cts of P^*}
		Assume Assumption \ref{A3} holds. Then $P^*:\Delta\times\mathcal{P}_2(\mathbb{R}^d)\to \mathcal{P}_2(\mathbb{R}^d)$ is jointly continuous.
	\end{proposition}
	\begin{proof}
		We only need to show that if $(t_n,s_n)\to (t,s)$ in $\Delta$ and $\mu_n\xrightarrow{\mathcal{W}_2}\mu$, then
		\begin{equation}\label{eq: 0430}
			\lim_{n\to \infty}\mathcal{W}_2(P^*(t_n,s_n)\mu_n,P^*(t,s)\mu)=0.
		\end{equation}

		We begin with the following three claims, which will be proven later.

		\noindent\textbf{Claim 1.} For any $\mu_1,\mu_2\in \mathcal{P}_2(\mathbb{R}^d)$ and $0\leq t-s\leq T$, there exists $C_T>0$ such that $\mathcal{W}_2(P^*(t,s)\mu_1,P^*(t,s)\mu_2)\leq C_T \mathcal{W}_2(\mu_1,\mu_2)$.

        \noindent\textbf{Claim 2.} For any $\mu\in \mathcal{P}_2(\mathbb{R}^d)$ and $s\in \mathbb{R}$, $P^*(t,s)\mu$ is continuous in $t$ on $[s,\infty)$, i.e., for any $t_n\to t$, we have $\mathcal{W}_2(P^*(t_n,s)\mu,P^*(t,s)\mu)\to 0$.

		\noindent\textbf{Claim 3.} For any $\mu\in \mathcal{P}_2(\mathbb{R}^d)$ and any sequence $(t_n,s_n)\in \Delta$ with $\lim_{n\to \infty}|t_n-s_n|=0$, we have $\mathcal{W}_2(P^*(t_n,s_n)\mu,\mu)\to 0$.
        
        \noindent By Claim 1, we get
		\begin{equation}\label{eq: 0430-3}
			\begin{split}
				\mathcal{W}_2(P^*(t_n,s_n)\mu_n,P^*(t,s)\mu)&\leq \mathcal{W}_2(P^*(t_n,s_n)\mu_n,P^*(t_n,s_n)\mu)+\mathcal{W}_2(P^*(t_n,s_n)\mu,P^*(t_n,s)\mu)\\
				&\ \ \ \ +\mathcal{W}_2(P^*(t_n,s)\mu,P^*(t,s)\mu)\\
				&\leq C_T\mathcal{W}_2(\mu_n,\mu)+C_T\mathcal{W}_2(P^*(s_n\vee s, s_n\wedge s)\mu,\mu)\\
				&\ \ \ \ +\mathcal{W}_2(P^*(t_n,s)\mu,P^*(t,s)\mu),
			\end{split}
		\end{equation}
		where the second inequality holds by the following equality
		\begin{equation*}
			\begin{split}
				\mathcal{W}_2(P^*(t_n,s_n)\mu,P^*(t_n,s)\mu)=\mathcal{W}_2(P^*(t_n,s_n\vee s)P^*(s_n\vee s,s_n)\mu,P^*(t_n,s_n\vee s)P^*(s_n\vee s,s)\mu).
			\end{split}
		\end{equation*}
		Since $t_n\to t$ and $s_n\to s$ as $n\to \infty$, \eqref{eq: 0430-3} together with Claim 2 and Claim 3, implies \eqref{eq: 0430}.  

        Now, we present the proofs of Claim 1-Claim 3.

		\noindent\textbf{Proof of Claim 1:}  According to \cite[Theorem 4.1]{Villani2009}, we can choose $\xi_1,\xi_2\in L^2(\mathcal{F}_s)$ such that $\mathscr{L}(\xi_1)=\mu_1, \mathscr{L}(\xi_2)=\mu_2$ and $\mathcal{W}_2^2(\mu_1,\mu_2)=\mathbf{E}[|\xi_1-\xi_2|^2]$. Let $X_t^{s,\xi_1}, X_t^{s,\xi_2}$ be the solutions of \eqref{McKean-Vlasov SDE 2} with initial conditions $\xi_1, \xi_2$ respectively.
		Set $\hat{X}_t:=X_t^{s,\xi_1}-X_t^{s,\xi_2}$ and
        \begin{equation*}
        \hat{f}_t:=f\big(t,X_t^{s,\xi_1}, \mathscr{L}(X_t^{s,\xi_1})\big)-f\big(t,X_t^{s,\xi_2}, \mathscr{L}(X_t^{s,\xi_2})\big), \ f=b,\sigma.
        \end{equation*}
        Applying It\^o's formula to $|\hat{X}_t|^2$ yields
		\begin{equation*}
		\begin{split}
				\mathbf{E}[|\hat{X}_t|^2]&=\mathbf{E}[|\hat{X}_s|^2]+\mathbf{E}\biggl[\int_s^t\bigl(2\langle \hat{X}_r, \hat{b}_r\rangle+\|\hat{\sigma}_r\|^2\bigr)\mathrm{d}r\biggr]\\
				&\leq \mathbf{E}[|\xi_1-\xi_2|^2]+\mathbf{E}\biggl[\int_s^t\Bigl(2L\bigl(|\hat{X}_r|^2+|\hat{X}_r|(\mathbf{E}[|\hat{X}_r|^2])^{\frac{1}{2}}\bigr)+L^2\bigl(|\hat{X}_r|+(\mathbf{E}[|\hat{X}_r|^2])^{\frac{1}{2}}\bigr)^2\Bigr)\mathrm{d}r\biggr]\\
				&\leq \mathcal{W}_2^2(\mu_1,\mu_2)+(4L^2+4L)\int_s^t\mathbf{E}[|\hat{X}_r|^2]\mathrm{d}r.
›		\end{split}
		\end{equation*}
		It follows from the Gronwall inequality that for any $0\leq t-s\leq T$
		\begin{equation*}
			\mathcal{W}_2^2(P^*(t,s)\mu_1,P^*(t,s)\mu_2)\leq \mathbf{E}[|\hat{X}_t|^2]\leq e^{(4L^2+4L)T}\mathcal{W}_2^2(\mu_1,\mu_2).
		\end{equation*}
		Claim 1 is proved by letting $C_T=e^{(2L^2+2L)T}$.

		\noindent\textbf{Proof of Claim 2:} For any $\xi\in L^2(\mathcal{F}_s)$ with $\mathscr{L}(\xi)=\mu$. It\^o's formula, together with the Burkholder-Davis-Gundy inequality and Gronwall's inequality gives that, for any $T>0$, there exists $C_T>0$ such that for any $s\in \mathbb{R}$,
		\begin{equation}\label{eq:mvsde-estimate}
			\mathbf{E}\Big[\sup_{s\leq t\leq s+T}|X^{s,\xi}_t|^2\Big]\leq C_T\left(1+\mathbf{E}[|\xi|^2]\right)=C_T\left(1+\|\mu\|_2^2\right).
		\end{equation}
		Note that the solution $X^{s,\xi}_t$ is continuous a.s. in $t$, then $\lim_{n\to \infty}|X^{s,\xi}_{t_n}-X^{s,\xi}_t|^2=0$, a.s. With \eqref{eq:mvsde-estimate}, the dominated convergence theorem gives
		\begin{equation*}
			\lim_{n\to \infty}\mathcal{W}_2(P^*(t_n,s)\mu,P^*(t,s)\mu)\leq \lim_{n\to \infty}\bigl(\mathbf{E}\bigl[|X^{s,\xi}_{t_n}-X^{s,\xi}_t|^2\bigr]\bigr)^{\frac{1}{2}}=0.
		\end{equation*}

		\noindent\textbf{Proof of Claim 3:} We first show that
		\begin{equation}\label{eq: W_p converge}
			\lim_{n\to \infty}\mathcal{W}_p(P^*(t_n,s_n)\mu,\mu)=0, \ \text{ for any } \ 1\leq p<2.
		\end{equation} 
		Let $\xi\in L^2(\mathcal{F}_{\underline{s}})$ with $\mathscr{L}(\xi)=\mu$ where $\underline{s}=\min\{s_n: n\geq 1\}$. Assume $\{X_t^{s_n,\xi}\}_{t\geq s_n}$ is the solution of the following equation
		\begin{equation*}
			X_t^{s_n,\xi}=\xi+\int_{s_n}^tb\big(r,X_r^{s_n,\xi}, \mathscr{L}(X_r^{s_n,\xi})\big)\mathrm{d} r+\int_{s_n}^t\sigma\big(r,X_r^{s_n,\xi}, \mathscr{L}(X_r^{s_n,\xi})\big)\mathrm{d} W_r.
		\end{equation*}
		Let $T=\max\{|t_n-s_n|: n\geq 1\}$. Then by \eqref{eq:mvsde-estimate}, we know that
		\begin{equation}\label{eq: 0417-1}
			\mathbf{E}\Big[\sup_{s_n\leq r\leq t_n}|X^{s_n,\xi}_r|^2\Big]\leq C_T\left(1+\mathbf{E}[|\xi|^2]\right)=C_T\left(1+\|\mu\|_2^2\right).
		\end{equation}
		Set 
		\begin{equation*}
			\tau^n_N:=\inf\{t\geq s_n: |X_t^{s_n,\xi}|\geq N\},
		\end{equation*}
		where $\inf\emptyset:=\infty$. Then \eqref{eq: 0417-1} and Chebyshev's inequality give that
		\begin{equation*}
			\mathbf{P}\{\tau^n_N\leq t_n\}=\mathbf{P}\Big\{\sup_{s_n\leq r\leq t_n}|X^{s_n,\xi}_r|\geq N\Big\}\leq \frac{C_T(1+\|\mu\|_2^2)}{N^2}, \ \text{ for all } \ n,N\geq 1.
		\end{equation*}
		Note that
		\begin{equation*}
			\begin{split}
				\mathbf{E}\biggl[\bigg|\int_{s_n}^{t_n}b\big(r,X_r^{s_n,\xi}, \mathscr{L}(X_r^{s_n,\xi})\big)\mathrm{d} r\bigg|^2\biggr]&\leq 2\mathbf{E}\bigl[|X_{t_n}^{s_n,\xi}-\xi|^2\bigr]+2\mathbf{E}\biggl[\bigg|\int_{s_n}^{t_n}\sigma\big(r,X_r^{s_n,\xi}, \mathscr{L}(X_r^{s_n,\xi})\big)\mathrm{d} W_r\bigg|^2\biggr]\\
				&=2\mathbf{E}\bigl[|X_{t_n}^{s_n,\xi}-\xi|^2\bigr]+2\int_{s_n}^{t_n}\mathbf{E}\bigl[\big\|\sigma\big(r,X_r^{s_n,\xi}, \mathscr{L}(X_r^{s_n,\xi})\big)\big\|^2\bigr]\mathrm{d} r\\
				&\leq C_T\left(1+\|\mu\|_2^2\right).
			\end{split}
		\end{equation*}
		Then for any $1\leq p<2$,
		\begin{equation*}
			\begin{split}
				&\ \ \ \ \mathbf{E}\biggl[\bigg|\int_{s_n}^{t_n}b\big(r,X_r^{s_n,\xi}, \mathscr{L}(X_r^{s_n,\xi})\big)\mathrm{d} r\bigg|^p\biggr]\\
				&\leq 2^{p-1}\mathbf{E}\biggl[\bigg|\int_{s_n}^{t_n}b\big(r,X_r^{s_n,\xi}, \mathscr{L}(X_r^{s_n,\xi})\big)\mathrm{d} r1_{\{\tau^n_N\leq t_n\}}\bigg|^p\biggr]\\
				&\ \ \ \ +2^{p-1}\mathbf{E}\biggl[\bigg|\int_{s_n}^{t_n}b\big(r,X_r^{s_n,\xi}, \mathscr{L}(X_r^{s_n,\xi})\big)\mathrm{d} r1_{\{\tau^n_N> t_n\}}\bigg|^p\biggr]\\
				&\leq 2^{p-1}\biggl(\mathbf{E}\biggl[\bigg|\int_{s_n}^{t_n}b\big(r,X_r^{s_n,\xi}, \mathscr{L}(X_r^{s_n,\xi})\big)\mathrm{d} r\bigg|^2\biggr]\biggr)^{\frac{p}{2}}\bigl(\mathbf{P}\{\tau^n_N\leq t_n\}\bigr)^{1-\frac{p}{2}}\\
				&\ \ \ \ +2^{p-1}C_T(N^{\kappa}+\|\mu\|_2^{\kappa})(t_n-s_n)\\
				&\leq \frac{2^{p-1}C_T\left(1+\|\mu\|_2^2\right)}{N^{2-p}}+2^{p-1}C_T\big(N^{\kappa}+\|\mu\|_2^{\kappa}\big)(t_n-s_n).
			\end{split}
		\end{equation*}
		By the Burkholder-Davis-Gundy inequality, we know that there exists $C_p>0$ such that 
		\begin{equation*}
			\begin{split}
				\mathbf{E}\biggl[\bigg|\int_{s_n}^{t_n}\sigma\big(r,X_r^{s_n,\xi}, \mathscr{L}(X_r^{s_n,\xi})\big)\mathrm{d} W_r\bigg|^p\biggr]&\leq C_p\mathbf{E}\biggl[\bigg(\int_{s_n}^{t_n}\big\|\sigma\big(r,X_r^{s_n,\xi}, \mathscr{L}(X_r^{s_n,\xi})\big)\big\|^2\mathrm{d} r\bigg)^{\frac{p}{2}}\biggr]\\
				&\leq C_p\bigl(c_{\sigma}d(t_n-s_n)\bigr)^{\frac{p}{2}}.
			\end{split}
		\end{equation*}
		Hence, for all $N\geq 1$,
		\begin{equation*}
			\begin{split}
				\mathbf{E}\bigl[|X_{t_n}^{s_n,\xi}-\xi|^p\bigr]&\leq 2^{p-1}\mathbf{E}\biggl[\bigg|\int_{s_n}^{t_n}b\big(r,X_r^{s_n,\xi}, \mathscr{L}(X_r^{s_n,\xi})\big)\mathrm{d} r\bigg|^p\biggr]\\
				&\ \ \ \ +2^{p-1}\mathbf{E}\biggl[\bigg|\int_{s_n}^{t_n}\sigma\big(r,X_r^{s_n,\xi}, \mathscr{L}(X_r^{s_n,\xi})\big)\mathrm{d} W_r\bigg|^p\biggr]\\
				&\leq \frac{2^{p-1}C_T\left(1+\|\mu\|_2^2\right)}{N^{2-p}}+2^{p-1}C_T\big(N^{\kappa}+\|\mu\|_2^{\kappa}\big)(t_n-s_n)+C_p\bigl(c_{\sigma}d(t_n-s_n)\bigr)^{\frac{p}{2}}.
			\end{split}
		\end{equation*}
		Taking the limit of $n\to \infty$ and then the limit of $N\to \infty$, we have
		\begin{equation*}
				\lim_{n\to \infty}\mathbf{E}\bigl[|X_{t_n}^{s_n,\xi}-\xi|^p\bigr]=0.
		\end{equation*}
		We proved \eqref{eq: W_p converge} since $\mathcal{W}_p^p(P^*(t_n,s_n)\mu,\mu)\leq \mathbf{E}\bigl[|X_{t_n}^{s_n,\xi}-\xi|^p\bigr]$.

		Now we are in the position to show that $\mathcal{W}_2(P^*(t_n,s_n)\mu,\mu)\to 0$ as $n\to \infty$. According to Theorem 6.9 in Villani \cite{Villani2009}, this is equivalent to show that $P^*(t_n,s_n)\mu$ weakly converges to $\mu$ and 
		\begin{equation}\label{eq: 0430-1}
			\lim_{n\to \infty}\int_{\mathbb{R}^d}|x|^2\mathrm{d} P^*(t_n,s_n)\mu(x)=\int_{\mathbb{R}^d}|x|^2\mathrm{d} \mu(x).
		\end{equation}
		Note $\mathcal{W}_p(P^*(t_n,s_n)\mu,\mu)\to 0$ as $n\to \infty$ for all $1\leq p<2$ implies the weak convergence and to prove \eqref{eq: 0430-1} is equivalent to prove
		\begin{equation}\label{eq: 0430-2}
			\lim_{n\to \infty}\mathbf{E}[|X_{t_n}^{s_n,\xi}|^2]=\mathbf{E}[|\xi|^2].
		\end{equation}
		For this, first applying It\^{o}'s formula to $|X_t^{s_n,\xi}|^2$ on $[s_n,t_n]$, we have
		\begin{equation*}
			\begin{split}
				\mathbf{E}[|X_{t_n}^{s_n,\xi}|^2]&=\mathbf{E}[|\xi|^2]+\mathbf{E}\biggl[\int_{s_n}^{t_n}\big(2\big\langle X_{r}^{s_n,\xi}, b\big(r,X_r^{s_n,\xi}, \mathscr{L}(X_r^{s_n,\xi})\big)\big\rangle+\big\|\sigma\big(r,X_r^{s_n,\xi}, \mathscr{L}(X_r^{s_n,\xi})\big)\big\|^2\big) \mathrm{d} r\biggr]\\
				&\leq \mathbf{E}[|\xi|^2]+\int_{s_n}^{t_n}\big(5L\mathbf{E}[|X_{r}^{s_n,\xi}|^2]+L+c_{\sigma}d\big)\mathrm{d} r\\
				&\leq \mathbf{E}[|\xi|^2]+C_T(1+\|\mu\|_2^2)(t_n-s_n),
			\end{split}
		\end{equation*}
        where we have used the following inequality by applying \eqref{Ineq of non-degenerate diffusion}, \eqref{eq:one-side-Lipschitz} and \eqref{Ineq polynomial growth} in Assumption \ref{A3},
        \begin{equation*}
            \begin{split}
                2\langle x, b(t,x,\mu)\rangle+\|\sigma(t,x,\mu)\|^2&=2\langle x, b(t,x,\mu)-b(t,0,\delta_0)\rangle+2\langle x, b(t,0,\delta_0)\rangle+\|\sigma(t,x,\mu)\|^2\\
                &\leq 2L\big(|x|^2+|x|\mathcal{W}_2(\mu,\delta_0)\big)+2L|x|+c_{\sigma}d\\
                &\leq 4L|x|^2+L\|\mu\|_2^2+L+c_{\sigma}d.
            \end{split}
        \end{equation*}
		Hence,
		\begin{equation*}
			\limsup_{n\to \infty}\mathbf{E}[|X_{t_n}^{s_n,\xi}|^2]\leq \mathbf{E}[|\xi|^2].
		\end{equation*}
		On the other hand, Theorem 6.9 in \cite{Villani2009} and the fact that $\lim_{n\to \infty}\mathcal{W}_p(P^*(t_n,s_n)\mu,\mu)=0$ for all $1\leq p<2$ imply
		\begin{equation*}
			\lim_{n\to \infty}\mathbf{E}[|X_{t_n}^{s_n,\xi}|^p]=\mathbf{E}[|\xi|^p], \ \text{ for all } \ 1\leq p<2.
		\end{equation*}
		Then it follows that
		\begin{equation*}
			\liminf_{n\to \infty}\mathbf{E}[|X_{t_n}^{s_n,\xi}|^2]\geq \liminf_{n\to \infty}\big(\mathbf{E}[|X_{t_n}^{s_n,\xi}|^p]\big)^{\frac{2}{p}}=\big(\mathbf{E}[|\xi|^p]\big)^{\frac{2}{p}}, \ \text{ for all } \ 1\leq p<2.
		\end{equation*}
		Note that $\mathbf{E}[|\xi|^p]\to \mathbf{E}[|\xi|^2]$ as $p\to 2$, we have
		\begin{equation*}
			\liminf_{n\to \infty}\mathbf{E}[|X_{t_n}^{s_n,\xi}|^2]\geq \mathbf{E}[|\xi|^2].
		\end{equation*}
		Therefore, we obtain \eqref{eq: 0430-2} and Claim 3 holds.	
	\end{proof}

	\section{The existence of invariant measures on the lifted space $\hat{\mathbb{X}}$: Proof of Theorem \ref{thm:existence-of-invariant-measure}}
 In this section, we aim to show the existence of invariant measures of $\hat{P}^*_t$ given as in \eqref{eq:new1202-1} (i.e., Theorem \ref{thm:existence-of-invariant-measure}). The key idea is to show that any weak limit of $\frac{1}{2T}\int_{-T}^T\big(\delta_{t|\vec{\mathbf{\tau}}}\times \mu_t\big){\rm d}t$ as $T\to\infty$, where $t|\vec{\mathbf{\tau}}:=(t\mod \tau_1,\cdots,t\mod \tau_n)\in\mathbb{T}^n$ and $\mu$ is an entrance measure of equation \eqref{McKean-Vlasov SDE 2}, is an invariant measure of $\hat{P}^*$. 

	 In the distribution dependent case, under Assumptions \ref{A3} and \ref{Quasi-periodic condition}, the McKean-Vlasov equations \eqref{McKean-Vlasov SDE 2} and \eqref{New Equation K_r_1,r_2} have unique strong solutions in $L^{2}(\mathcal{F}_t)$ starting from $s\in \mathbb{R}$ with initial condition $\xi\in L^2(\mathcal{F}_s)$, which are still denoted by $X_t^{s,\xi}, K^{\vec{\mathbf{s}}}(t,s,\xi)$.
 We have for all $s\leq u\leq t$, $r\in \mathbb{R}, \vec{\mathbf{s}}\in \mathbb{T}^n$ and $\xi\in L^{2}(\mathcal{F}_s)$,
	\begin{equation}\label{0508-2}
		\begin{cases}
			X_t^{s,\xi}(\omega)=\bigl[X_t^{u,X_u^{s,\xi}}\bigr](\omega), &\mathbf{P}-a.s. \ \omega\in \Omega,\\
			K^{\vec{\mathbf{s}}}\bigl(t,u,K^{\vec{\mathbf{s}}}(u,s,\xi)\bigr)(\omega)=K^{\vec{\mathbf{s}}}(t,s,\xi)(\omega), &\mathbf{P}-a.s. \ \omega\in \Omega,\\
			K^{T_r(\vec{\mathbf{s}})}(t,s,\xi)(\theta_r\omega)=K^{\vec{\mathbf{s}}}(t+r,s+r,\xi\circ \theta_r)(\omega), &\mathbf{P}-a.s. \ \omega\in \Omega.
		\end{cases}
	\end{equation}
 The first two identities are obvious, and we only need to prove the third equality in \eqref{0508-2}. In fact,
    \begin{equation*}
        \begin{split}
            K^{T_r(\vec{\mathbf{s}})}(t,s,\xi)\circ \theta_r
            &=\xi\circ \theta_r+\bigg(\int_s^t\tilde{b}^{T_r(\vec{\mathbf{s}})}\Big(u,K^{T_r(\vec{\mathbf{s}})}(u,s,\xi),\mathscr{L}\big(K^{T_r(\vec{\mathbf{s}})}(u,s,\xi)\big)\Big){\rm d}u\bigg)\circ \theta_r\\
            &\ \ \ \ +\bigg(\int_s^t\tilde{\sigma}^{T_r(\vec{\mathbf{s}})}\Big(u,K^{T_r(\vec{\mathbf{s}})}(u,s,\xi),\mathscr{L}\big(K^{T_r(\vec{\mathbf{s}})}(u,s,\xi)\big)\Big){\rm d}W_u\bigg)\circ \theta_r\\
            &=\xi\circ \theta_r+\bigg(\int_s^t\tilde{b}^{\vec{\mathbf{s}}}\Big(u+r,K^{T_r(\vec{\mathbf{s}})}(u,s,\xi),\mathscr{L}\big(K^{T_r(\vec{\mathbf{s}})}(u,s,\xi)\big)\Big){\rm d}u\bigg)\circ \theta_r\\
            &\ \ \ \ +\bigg(\int_s^t\tilde{\sigma}^{\vec{\mathbf{s}}}\Big(u+r,K^{T_r(\vec{\mathbf{s}})}(u,s,\xi),\mathscr{L}\big(K^{T_r(\vec{\mathbf{s}})}(u,s,\xi)\big)\Big){\rm d}W_u\bigg)\circ \theta_r\\
            &=\xi\circ \theta_r+\int_{s+r}^{t+r}\tilde{b}^{\vec{\mathbf{s}}}\Big(v,K^{T_r(\vec{\mathbf{s}})}(v-r,s,\xi)\circ\theta_r,\mathscr{L}\big(K^{T_r(\vec{\mathbf{s}})}(v-r,s,\xi)\circ\theta_r\big)\Big){\rm d}v\\
            &\ \ \ \ +\int_{s+r}^{t+r}\tilde{\sigma}^{\vec{\mathbf{s}}}\Big(v,K^{T_r(\vec{\mathbf{s}})}(v-r,s,\xi)\circ\theta_r,\mathscr{L}\big(K^{T_r(\vec{\mathbf{s}})}(v-r,s,\xi)\circ\theta_r\big)\Big){\rm d}W_v,
        \end{split}
    \end{equation*}
    where we have used the fact that $\mathscr{L}\big(K^{T_r(\vec{\mathbf{s}})}(v-r,s,\xi)\big)=\mathscr{L}\big(K^{T_r(\vec{\mathbf{s}})}(v-r,s,\xi)\circ\theta_r\big)$ and
    \begin{equation*}
        \Big(\int_s^tf(u){\rm d}W_u\Big)\circ\theta_r=\int_{s+r}^{t+r}f(v-r)\circ\theta_r {\rm d}W_v, \ \text{ for all integrable adapted process $f$.}
    \end{equation*}
    Then, the third equality in \eqref{0508-2} follows from the well-posedness of equation \eqref{New Equation K_r_1,r_2}.

Recall the first lift $\Phi_t$ defined on $\mathbb{T}^n\times L^{2}(\mathcal{F}_0)$ by
\begin{equation*}
	\Phi_t(\vec{\mathbf{s}};\xi)=\big(T_t(\vec{\mathbf{s}}); K^{\vec{\mathbf{s}}}(t,0,\xi)\circ \theta_{-t}\big), \ \ \vec{\mathbf{s}}\in \mathbb{T}^n, \ \xi\in L^{2}(\mathcal{F}_0),
\end{equation*}
where $(K^{\vec{\mathbf{s}}}(t,0,\xi)\circ \theta_{-t})(\omega):=K^{\vec{\mathbf{s}}}(t,0,\xi)(\theta_{-t}\omega)$ is the $\theta_{-t}$ Wiener shift of the solution $K^{\vec{\mathbf{s}}}(t,0,\xi)$. 
	Since $\theta_{-t}\mathcal{F}_t=\mathcal{F}_0$ for all $t\geq 0$, then for any $\xi \in L^{2}(\mathcal{F}_0)$, we know that
	\begin{equation}\label{neweq:add}
		K^{\vec{\mathbf{s}}}(t,0,\xi)\in L^{2}(\mathcal{F}_t) \ \text{ and } \ K^{\vec{\mathbf{s}}}(t,0,\xi)\circ \theta_{-t} \in L^{2}(\mathcal{F}_0).
	\end{equation}

Then we obtain the following result.

	\begin{lemma}\label{lemma 0530}
		Assume Assumptions \ref{A3}, \ref{Quasi-periodic condition} and \ref{One-sided Lip of tilde b} hold. Then $(\mathbb{R}^+, \mathbb{T}^n\times L^{2}(\mathcal{F}_0), \Phi)$ is a continuous dynamical system.
	\end{lemma}
	\begin{proof}
		As we have discussed above, for any $t\geq 0$, $\Phi_t$ maps $\mathbb{T}^n\times L^{2}(\mathcal{F}_0)$ into itself. Now for any $(\vec{\mathbf{s}};\xi)\in \mathbb{T}^n\times L^{2}(\mathcal{F}_0)$ and $t,s\in \mathbb{R}^+$, it follows from \eqref{0508-2} that 
		\begin{equation*}
			\begin{split}
				\Phi_t\circ \Phi_s (\vec{\mathbf{s}};\xi)&=\Phi_t\big(T_s(\vec{\mathbf{s}}); K^{\vec{\mathbf{s}}}(s,0,\xi)\circ \theta_{-s}\big)\\
				&=\Big(T_{t+s}(\vec{\mathbf{s}}); K^{T_s(\vec{\mathbf{s}})}\big(t,0,K^{\vec{\mathbf{s}}}(s,0,\xi)\circ \theta_{-s}\big)\circ \theta_s\circ\theta_{-(t+s)}\Big)\\
				&=\Big(T_{t+s}(\vec{\mathbf{s}}); K^{\vec{\mathbf{s}}}\big(t+s,s,K^{\vec{\mathbf{s}}}(s,0,\xi)\big)\circ \theta_{-(t+s)}\Big)\\
				&=\big(T_{t+s}(\vec{\mathbf{s}}); K^{\vec{\mathbf{s}}}(t+s,0,\xi)\circ \theta_{-(t+s)}\big)\\
				&=\Phi_{t+s}(\vec{\mathbf{s}};\xi).
			\end{split}
		\end{equation*}
		Obviously, $\Phi_0(\vec{\mathbf{s}};\xi)=(\vec{\mathbf{s}};\xi)$, for all $(\vec{\mathbf{s}};\xi)\in \mathbb{T}^n\times L^{2}(\mathcal{F}_0)$.
		
		It remains to show that $\Phi: \mathbb{R}^+\times\mathbb{T}^n\times L^{2}(\mathcal{F}_0)\to \mathbb{T}^n\times L^{2}(\mathcal{F}_0)$ is continuous, i.e., for any sequences $\{t_m, t\}_{m\geq 1}\subset \mathbb{R}^+$ and $\{(\vec{\mathbf{s}}_m;\xi_m), (\vec{\mathbf{s}};\xi)\}_{m\geq 1}\subset \mathbb{T}^n\times L^{2}(\mathcal{F}_0)$ such that
		\begin{equation*}
			\lim_{m\to \infty}\Bigl(|t_m-t|+d_0(\vec{\mathbf{s}}_m, \vec{\mathbf{s}})+\mathbf{E}\big[|\xi_m-\xi|^2\big]\Bigr)=0,
		\end{equation*}
		we have 
		\begin{equation*}
			\lim_{m\to \infty}\Bigl(d_0\bigl(T_{t_m}(\vec{\mathbf{s}}_m),T_{t}(\vec{\mathbf{s}})\bigr)+\mathbf{E}\big[\big|K^{\vec{\mathbf{s}}_m}(t_m,0,\xi_m)\circ\theta_{-t_m}-K^{\vec{\mathbf{s}}}(t,0,\xi)\circ\theta_{-t}\big|^2\big]\Bigr)=0.
		\end{equation*}
		Note that
		\begin{equation*}
			d_0\bigl(T_{t_m}(\vec{\mathbf{s}}_m),T_{t}(\vec{\mathbf{s}})\bigr)\leq d_0\bigl(T_{t_m}(\vec{\mathbf{s}}_m),T_{t_m}(\vec{\mathbf{s}})\bigr)+d_0\bigl(T_{t_m}(\vec{\mathbf{s}}),T_{t}(\vec{\mathbf{s}})\bigr)\leq n|t_m-t|+d_0(\vec{\mathbf{s}}_m,\vec{\mathbf{s}}),
		\end{equation*}
		it is adequate to show that 
		\begin{equation}\label{n0529-1}
			\lim_{m\to \infty}\mathbf{E}\bigl[\big|K^{\vec{\mathbf{s}}_m}(t_m,0,\xi_m)\circ \theta_{-t_m}-K^{\vec{\mathbf{s}}}(t,0,\xi)\circ \theta_{-t}\big|^2\bigr]=0.
		\end{equation}
        Since $\theta_t$ preserves $\mathbf{P}$, we have
        \begin{equation*}
            \begin{split}
                \mathbf{E}\bigl[\big|K^{\vec{\mathbf{s}}_m}(t_m,0,\xi_m)\circ \theta_{-t_m}-K^{\vec{\mathbf{s}}}(t,0,\xi)\circ \theta_{-t}\big|^2\bigr]&\leq 2\mathbf{E}\bigl[\big|K^{\vec{\mathbf{s}}_m}(t_m,0,\xi_m)-K^{\vec{\mathbf{s}}}(t,0,\xi)\big|^2\bigr]\\
                &\ \ \ \ +2\mathbf{E}\bigl[\big|K^{\vec{\mathbf{s}}}(t,0,\xi)\circ \theta_{t-t_m}-K^{\vec{\mathbf{s}}}(t,0,\xi)\big|^2\bigr].
            \end{split}
        \end{equation*}
        By \cite[Proposition 2.2.3]{Da-Prato-Zabczyk1996}, we know that
        \begin{equation*}
            \lim_{m\to\infty}\mathbf{E}\bigl[\big|K^{\vec{\mathbf{s}}}(t,0,\xi)\circ \theta_{t-t_m}-K^{\vec{\mathbf{s}}}(t,0,\xi)\big|^2\bigr]=0.
        \end{equation*}
		Then it suffices to prove that
		\begin{equation}\label{0527}
			\lim_{m\to \infty}\mathbf{E}\bigl[\big|K^{\vec{\mathbf{s}}_m}(t_m,0,\xi_m)-K^{\vec{\mathbf{s}}}(t,0,\xi)\big|^2\bigr]=0.
		\end{equation}
		Now for any fixed $t\geq s$, $\vec{\mathbf{t}},\vec{\mathbf{s}}\in \mathbb{T}^n$ and $\xi,\eta\in L^2(\mathcal{F}_s)$, applying It{\^ o}'s formula to $|K^{\vec{\mathbf{t}}}(t,s,\xi)-K^{\vec{\mathbf{s}}}(t,s,\eta)|^2$ together with Assumptions \ref{A3}, \ref{Quasi-periodic condition} and \ref{One-sided Lip of tilde b} gives
		\begin{equation*}
			\begin{split}
				\mathbf{E}\bigl[\big|K^{\vec{\mathbf{t}}}(t,s,\xi)-K^{\vec{\mathbf{s}}}(t,s,\eta)\big|^2\bigr]&\leq \mathbf{E}\bigl[|\xi-\eta|^2\bigr]+2\int_s^th\big(T_r(\vec{\mathbf{t}}),T_r(\vec{\mathbf{s}})\big)\mathrm{d}r\\
				&\ \ \ \ +(8L^2+4L+1)\int_s^t\mathbf{E}\bigl[\big|K^{\vec{\mathbf{t}}}(r,s,\xi)-K^{\vec{\mathbf{s}}}(r,s,\eta)\big|^2\bigr]\mathrm{d}r.
			\end{split}
		\end{equation*}
		Then the Gronwall inequality implies
		\begin{equation}\label{n0527-1}
			\mathbf{E}\bigl[\big|K^{\vec{\mathbf{t}}}(t,s,\xi)-K^{\vec{\mathbf{s}}}(t,s,\eta)\big|^2\bigr]\leq \Big(\mathbf{E}\bigl[|\xi-\eta|^2\bigr]+2\int_s^th\big(T_r(\vec{\mathbf{t}}),T_r(\vec{\mathbf{s}})\big)\mathrm{d}r\Big)e^{(8L^2+4L+1)(t-s)}.
		\end{equation}
		Since $h: \mathbb{T}^n\times \mathbb{T}^n\to \mathbb{R}^+$ is continuous and $\mathbb{T}^n\times \mathbb{T}^n$ is compact, then $h$ is uniformly continuous, i.e., there exists an increasing function $w: [0,\infty)\to [0,\infty)$ satisfying $w(0)=0$ and continuous at $0$ such that for all $\vec{\mathbf{t}}_1,\vec{\mathbf{t}}_2,\vec{\mathbf{s}}_1,\vec{\mathbf{s}}_2\in \mathbb{T}^n$,
		  \begin{equation}\label{eq:new1204}
			|h(\vec{\mathbf{t}}_1,\vec{\mathbf{t}}_2)-h(\vec{\mathbf{s}}_1,\vec{\mathbf{s}}_2)|\leq w\bigl(d_0(\vec{\mathbf{t}}_1,\vec{\mathbf{s}}_1)+d_0(\vec{\mathbf{t}}_2,\vec{\mathbf{s}}_2)\bigr).
		  \end{equation}
		  Hence,
		  \begin{equation}\label{0527-3}
				\Big|\int_s^th\big(T_r(\vec{\mathbf{t}}),T_r(\vec{\mathbf{s}})\big)\mathrm{d}r\Big|\leq \int_s^t|h\big(T_r(\vec{\mathbf{t}}),T_r(\vec{\mathbf{s}})\big)-h\big(T_r(\vec{\mathbf{s}}),T_r(\vec{\mathbf{s}})\big)|\mathrm{d}r\leq (t-s)w\bigl(d_0(\vec{\mathbf{t}},\vec{\mathbf{s}})\bigr).
		  \end{equation}
		  Then it follows from \eqref{n0527-1} and \eqref{0527-3} that for any $t\geq s$, $\vec{\mathbf{t}},\vec{\mathbf{s}}\in \mathbb{T}^n$ and $\xi,\eta\in L^2(\mathcal{F}_s)$, 
		  \begin{equation}\label{0528-1}
			\mathbf{E}\bigl[\big|K^{\vec{\mathbf{t}}}(t,s,\xi)-K^{\vec{\mathbf{s}}}(t,s,\eta)\big|^2\bigr]\leq \big(\mathbf{E}\bigl[|\xi-\eta|^2\bigr]+2(t-s)w\bigl(d_0(\vec{\mathbf{t}},\vec{\mathbf{s}})\bigr)\big)e^{(8L^2+4L+1)(t-s)}.
		  \end{equation}
		  Let $T=\sup_{m\geq 1}t_m$. Then \eqref{0528-1} directly gives
		  \begin{equation}\label{n0527-2}
			\mathbf{E}\bigl[\big|K^{\vec{\mathbf{s}}_m}(t_m,0,\xi_m)-K^{\vec{\mathbf{s}}}(t_m,0,\xi)\big|^2\bigr]\leq \big(\mathbf{E}\bigl[|\xi_m-\xi|^2\bigr]+2Tw\bigl(d_0(\vec{\mathbf{s}}_m,\vec{\mathbf{s}})\bigr)\big)e^{(8L^2+4L+1)T}.
		  \end{equation}
		  Similar to \eqref{eq:mvsde-estimate}, there exists $C_T>0$ such that for all $\vec{\mathbf{s}}\in \mathbb{T}^n$ and $\xi\in L^2(\mathcal{F}_0)$,
		  \begin{equation}\label{n0530-1}
			  \mathbf{E}\Bigl[\sup_{0\leq t\leq T}\big|K^{\vec{\mathbf{s}}}(t,0,\xi)\big|^2\Bigr]\leq C_T\bigl(1+\mathbf{E}[|\xi|^2]\bigr).
		  \end{equation}
		  Note also that the solution $K^{\vec{\mathbf{s}}}(t,0,\xi)$ is $\mathbf{P}$-a.s. continuous in $t$, the dominated convergence theorem leads to that
		  \begin{equation}\label{0528-2}
			\lim_{m\to \infty}\mathbf{E}\bigl[\big|K^{\vec{\mathbf{s}}}(t_m,0,\xi)-K^{\vec{\mathbf{s}}}(t,0,\xi)\big|^2\bigr]=0.
		  \end{equation}
		  Then \eqref{0527} follows from \eqref{n0527-2} and \eqref{0528-2}.
	\end{proof}

    \begin{remark}
        For any $\vec{\mathbf{s}}\in \mathbb{T}^n$, we define a non-autonomous dynamical system on the fibre, $\{\vec{\mathbf{s}}\}\times L^2(\mathcal{F}_0)$, by
        \begin{equation*}
            \phi_{s,t; \vec{\mathbf{s}}}(\xi):=K^{\vec{\mathbf{s}}}(t,s,\xi\circ \theta_s)\circ \theta_{-t}, \ \text{ for all } \ t\geq s.
        \end{equation*}
        Then the first lifted dynamical system $\Phi$ is actually the skew product flow of the minimal rotation on the $n$-torus $(\mathbb{T}^n, \{T_t\}_{t\geq 0})$ with the dynamical system $\phi$:
        \[
        \Phi_t(\vec{\mathbf{s}},\xi)=(T_t(\vec{\mathbf{s}}),\phi_{0,t;\vec{\mathbf{s}}}(\xi)), \ \text{ for all } \ \vec{\mathbf{s}}\in\mathbb{T}^n, \ \xi\in L^2(\mathcal{F}_0).
        \]
    \end{remark}

    Recall that $\Delta={(t,s)\in\mathbb{R}^2: t\geq s}$ and 
    \begin{equation*}
        P^{\vec{\mathbf{s}},*}(t,s)\mu=\mathscr{L}\big(K^{\vec{\mathbf{s}}}(t,s,\xi)\big) \text{ with } \mathscr{L}(\xi)=\mu.
    \end{equation*}
    Then we obtain the following lemma.

	\begin{lemma}\label{Lemma 0517-1}
		Assume Assumptions \ref{A3}, \ref{Quasi-periodic condition} and \ref{One-sided Lip of tilde b} hold. Then the map $(\vec{\mathbf{s}},(t,s),\mu)\mapsto P^{\vec{\mathbf{s}},*}(t,s)\mu$ from $\mathbb{T}^n\times\Delta\times\mathcal{P}_2(\mathbb{R}^d)$ to $\mathcal{P}_2(\mathbb{R}^d)$ is continuous.
	\end{lemma}
	
	\begin{proof}
		For any sequences $\{(t_m,s_m), (t,s)\}_{m\geq 1}\subset 
		\Delta$, $\{\vec{\mathbf{s}}_m,\vec{\mathbf{s}}\}_{m\geq 1}\subset 
		\mathbb{T}^n$ and $\{\mu_m,\mu\}\subset \mathcal{P}_2(\mathbb{R}^d)$ such that
		\begin{equation*}
			\lim_{m\to\infty}\bigl(|t_m-t|+|s_m-s|+d_0(\vec{\mathbf{s}}_m,\vec{\mathbf{s}})+\mathcal{W}_2(\mu_m,\mu)\bigr)=0,
		\end{equation*}
		we need to show that
		\begin{equation}\label{n0528-1}
			\lim_{m\to\infty}\mathcal{W}_2\big(P^{\vec{\mathbf{s}}_m,*}(t_m,s_m)\mu_m,P^{\vec{\mathbf{s}},*}(t,s)\mu\big)=0.
		\end{equation}
        Note that
        \begin{equation*}
            \begin{split}
                \mathcal{W}_2\big(P^{\vec{\mathbf{s}}_m,*}(t_m,s_m)\mu_m,P^{\vec{\mathbf{s}},*}(t,s)\mu\big)&\leq \mathcal{W}_2\big(P^{\vec{\mathbf{s}}_m,*}(t_m,s_m)\mu_m,P^{\vec{\mathbf{s}},*}(t_m,s_m)\mu_m\big)\\
                &\ \ \ \ 
                +\mathcal{W}_2\big(P^{\vec{\mathbf{s}},*}(t_m,s_m)\mu_m,P^{\vec{\mathbf{s}},*}(t,s)\mu\big)
            \end{split}
        \end{equation*}
        By Proposition \ref{Joint cts of P^*}, for any fixed $\vec{\mathbf{s}}\in\mathbb{T}^n$, we know that
        \begin{equation}\label{eq:1128-1}
            \lim_{m\to\infty}\mathcal{W}_2\big(P^{\vec{\mathbf{s}},*}(t_m,s_m)\mu_m,P^{\vec{\mathbf{s}},*}(t,s)\mu\big)=0.
        \end{equation}
        On the other hand, for any $m\geq 1$, since $(\Omega,\mathcal{F}_{s_m},\mathbf{P})$ is a standard probability space, there exists $\xi_m\in L^2(\mathcal{F}_{s_m})$ such that $\mathscr{L}(\xi_m)=\mu_m$. Then by \eqref{0528-1}, there exists $C_{L,T}>0$ depending only on $(L,T)$ such that
        \begin{equation}\label{eq:1128-2}
            \begin{split}
                \mathcal{W}_2^2\big(P^{\vec{\mathbf{s}}_m,*}(t_m,s_m)\mu_m,P^{\vec{\mathbf{s}},*}(t_m,s_m)\mu_m\big)&\leq \mathbf{E}\big[\big|K^{\vec{\mathbf{s}}_m}(t_m,s_m,\xi_m)-K^{\vec{\mathbf{s}}}(t_m,s_m,\xi_m)\big|^2\big]\\
                &\leq C_{L,T}w\big(d_0(\vec{\mathbf{s}}_m,\vec{\mathbf{s}})\big),
            \end{split}
        \end{equation}
        where $T=\sup_{m\geq 1}|t_m-s_m|<\infty$, and $w: [0,\infty)\to [0,\infty)$ is an increasing function and is continuous at $0$ with $w(0)=0$. Then \eqref{n0528-1} follows from \eqref{eq:1128-1} and \eqref{eq:1128-2}.
	\end{proof}

Let us recall the second lift $\hat{\Phi}$ defined on $\mathcal{G}_{2}$ by
\begin{equation*}
		(\hat{\Phi}_t\hat{\xi})(\vec{\mathbf{s}},\omega)=\Phi_t(\vec{\mathbf{s}},\tilde{\xi}_{\vec{\mathbf{s}}})(\omega), \ \mathbf{P}-a.s., \ \ t\geq 0, \ \hat{\xi}={\rm graph}(\tilde{\xi})\in \mathcal{G}_{2}.
\end{equation*}
The following lemma shows that $\{\hat{\Phi}_t\}_{t\geq 0}$ forms a continuous dynamical system on $\mathcal{G}_{2}$.

	\begin{lemma}\label{coro 0530}
		Assume Assumptions \ref{A3}, \ref{Quasi-periodic condition} and \ref{One-sided Lip of tilde b} hold. Then $\hat{\Phi}_t$ maps $\mathcal{G}_{2}$ into itself for all $t\geq 0$ and $(\mathbb{R}^+,\mathcal{G}_{2},\hat{\Phi})$ is a continuous dynamical system.
	\end{lemma}

	\begin{proof}
		Notice that $T_t$ can be defined as in \eqref{0529-1} for all $t\in \mathbb{R}$ and $T_t$ is invertible with $T_t^{-1}=T_{-t}$. Then for any $\hat{\xi}={\rm graph}(\tilde{\xi})\in \mathcal{G}_{2}$, it follows from \eqref{0529-2} that 
		\begin{equation}\label{eq:graph-of-hat-Phi}
			\hat{\Phi}_t\hat{\xi}={\rm graph}\bigl(K^{T_{-t}(\cdot)}\big(t,0,\tilde{\xi}_{T_{-t}(\cdot)}\big)\circ \theta_{-t}\bigr).
		\end{equation}
		According to the proof of \eqref{n0529-1}, we know that 
		\begin{equation*}
			K^{T_{-t}(\cdot)}\big(t,0,\tilde{\xi}_{T_{-t}(\cdot)}\big)\circ \theta_{-t}\in C(\mathbb{T}^n; L^2(\mathcal{F}_0)),
		\end{equation*}
		and therefore, $\hat{\Phi}_t\hat{\xi}\in \mathcal{G}_{2}$. 
        Moreover, according to \eqref{0508-2} and \eqref{eq:graph-of-hat-Phi}, for any $\vec{\mathbf{s}}\in\mathbb{T}^n$,
        \begin{equation*}
            \begin{split}
                (\hat{\Phi}_t\circ\hat{\Phi}_s\hat{\xi})(\vec{\mathbf{s}})&=K^{T_{-t}(\vec{\mathbf{s}})}\big(t,0,(\hat{\Phi}_s\hat{\xi})(T_{-t}\vec{\mathbf{s}})\big)\circ \theta_{-t}\\
                &=K^{T_s(T_{-(t+s)}\vec{\mathbf{s}})}\big(t,0,K^{T_{-s}(T_{-t}\vec{\mathbf{s}})}(s,0,\tilde{\xi}_{T_{-s}(T_{-t}\vec{\mathbf{s}})})\circ\theta_{-s}\big)\circ \theta_s\circ\theta_{-(t+s)}\\
                &=K^{T_{-(t+s)}(\vec{\mathbf{s}})}\big(t+s,s,K^{T_{-(t+s)}(\vec{\mathbf{s}})}(s,0,\tilde{\xi}_{T_{-(t+s)}(\vec{\mathbf{s}})})\big)\circ\theta_{-(t+s)}\\
                &=K^{T_{-(t+s)}(\vec{\mathbf{s}})}\big(t+s,0,\tilde{\xi}_{T_{-(t+s)}(\vec{\mathbf{s}})}\big)\circ\theta_{-(t+s)}\\
                &=\hat{\Phi}_{t+s}\hat{\xi}(\vec{\mathbf{s}}).
            \end{split}
        \end{equation*}
        Hence, we have
		\begin{equation*}
			\hat{\Phi}_0\hat{\xi}=\hat{\xi}, \ \ \hat{\Phi}_{t+s}\hat{\xi}=\hat{\Phi}_t\circ\hat{\Phi}_s\hat{\xi}, \ \text{ for all } \ t,s\geq 0, \ \hat{\xi}\in \mathcal{G}_{2}.
		\end{equation*}
		It remains to show that for any sequences 
		\begin{equation*}
			\{t_m,t\}_{m\geq 1}\subset \mathbb{R}^+, \ \text{ and } \ \{\hat{\xi}^m={\rm graph}(\tilde{\xi}^m), \hat{\xi}={\rm graph}(\tilde{\xi})\}_{m\geq 1}\in \mathcal{G}_{2}
		\end{equation*}
		with
		\begin{equation*}
			\lim_{m\to \infty}\Big(|t_m-t|+\sup_{\vec{\mathbf{s}}\in \mathbb{T}^n}\mathbf{E}\bigl[|\tilde{\xi}^m_{\vec{\mathbf{s}}}-\tilde{\xi}_{\vec{\mathbf{s}}}|^2\bigr]\Big)=0,
		\end{equation*}
		we have
		\begin{equation}\label{m0530-1}
			\lim_{m\to \infty}\sup_{\vec{\mathbf{s}}\in \mathbb{T}^n}\mathbf{E}\bigl[\big|K^{T_{-t_m}(\vec{\mathbf{s}})}\big(t_m,0,\tilde{\xi}^m_{T_{-t_m}(\vec{\mathbf{s}})}\big)\circ \theta_{-t_m}-K^{T_{-t}(\vec{\mathbf{s}})}\big(t,0,\tilde{\xi}_{T_{-t}(\vec{\mathbf{s}})}\big)\circ \theta_{-t}\big|^2\bigr]=0.
		\end{equation}
		Note that
		\begin{equation}\label{m0530-2}
			\begin{split}
				&\ \ \ \ \mathbf{E}\bigl[\big|K^{T_{-t_m}(\vec{\mathbf{s}})}\big(t_m,0,\tilde{\xi}^m_{T_{-t_m}(\vec{\mathbf{s}})}\big)\circ \theta_{-t_m}-K^{T_{-t}(\vec{\mathbf{s}})}\big(t,0,\tilde{\xi}_{T_{-t}(\vec{\mathbf{s}})}\big)\circ \theta_{-t}\big|^2\bigr]\\
				&\leq 2\mathbf{E}\bigl[\big|K^{T_{-t_m}(\vec{\mathbf{s}})}\big(t_m,0,\tilde{\xi}^m_{T_{-t_m}(\vec{\mathbf{s}})}\big)\circ \theta_{-t_m}-K^{T_{-t}(\vec{\mathbf{s}})}\big(t_m,0,\tilde{\xi}_{T_{-t}(\vec{\mathbf{s}})}\big)\circ \theta_{-t_m}\big|^2\bigr]\\
				&\ \ \ \ +2\mathbf{E}\bigl[\big|K^{T_{-t}(\vec{\mathbf{s}})}\big(t_m,0,\tilde{\xi}_{T_{-t}(\vec{\mathbf{s}})}\big)\circ \theta_{-t_m}-K^{T_{-t}(\vec{\mathbf{s}})}\big(t,0,\tilde{\xi}_{T_{-t}(\vec{\mathbf{s}})}\big)\circ \theta_{-t}\big|^2\bigr].
			\end{split}
		\end{equation}
		Let $T=\sup_{m\geq 1}t_m$. According to \eqref{0528-1}, we know that
		\begin{equation}\label{m0530-3}
			\begin{split}
			&\ \ \ \ \mathbf{E}\bigl[\big|K^{T_{-t_m}(\vec{\mathbf{s}})}\big(t_m,0,\tilde{\xi}^m_{T_{-t_m}(\vec{\mathbf{s}})}\big)\circ \theta_{-t_m}-K^{T_{-t}(\vec{\mathbf{s}})}\big(t_m,0,\tilde{\xi}_{T_{-t}(\vec{\mathbf{s}})}\big)\circ \theta_{-t_m}\big|^2\bigr]\\
			&\leq e^{(8L^2+4L+1)T}\Bigl(\mathbf{E}\bigl[|\tilde{\xi}^m_{T_{-t_m}(\vec{\mathbf{s}})}-\tilde{\xi}_{T_{-t}(\vec{\mathbf{s}})}|^2\bigr]+2Tw\big(d_0(T_{-t_m}(\vec{\mathbf{s}}),T_{-t_m}(\vec{\mathbf{s}}))\big)\Bigr)\\
			&\leq e^{(8L^2+4L+1)T}\bigl(2\|\hat{\xi}^m-\hat{\xi}\|_{\mathcal{G}_2}^2+2\mathbf{E}\bigl[|\tilde{\xi}_{T_{-t_m}(\vec{\mathbf{s}})}-\tilde{\xi}_{T_{-t}(\vec{\mathbf{s}})}|^2\bigr]+2Tw(n|t_m-t|)\bigr).
			\end{split}
		\end{equation}
		Set
		\begin{equation*}
			f(\vec{\mathbf{t}},\vec{\mathbf{s}}):=\mathbf{E}\bigl[|\tilde{\xi}_{\vec{\mathbf{t}}}-\tilde{\xi}_{\vec{\mathbf{s}}}|^2\bigr], \ \text{ for all } \ \vec{\mathbf{t}},\vec{\mathbf{s}}\in \mathbb{T}^n.
		\end{equation*}
		Since $\tilde{\xi}\in C(\mathbb{T}^n;L^2(\mathcal{F}_0))$, it is easy to check that $f: \mathbb{T}^n\times \mathbb{T}^n\to \mathbb{R}^+$ is continuous and $f(\vec{\mathbf{t}},\vec{\mathbf{t}})=0$ for all $\vec{\mathbf{t}}\in \mathbb{T}^n$. Similar to that of function $h$, there exists function $w_{\tilde{\xi}}: [0,\infty)\to [0,\infty)$ satisfying $w_{\tilde{\xi}}(0)=0$ and continuous at $0$ such that for all $\vec{\mathbf{t}}_1,\vec{\mathbf{t}}_2,\vec{\mathbf{s}}_1,\vec{\mathbf{s}}_2\in \mathbb{T}^n$,
		\begin{equation*}
		  |f(\vec{\mathbf{t}}_1,\vec{\mathbf{t}}_2)-f(\vec{\mathbf{s}}_1,\vec{\mathbf{s}}_2)|\leq w_{\tilde{\xi}}\bigl(d_0(\vec{\mathbf{t}}_1,\vec{\mathbf{s}}_1)+d_0(\vec{\mathbf{t}}_2,\vec{\mathbf{s}}_2)\bigr).
		\end{equation*}
		Hence, $f(\vec{\mathbf{t}},\vec{\mathbf{s}})=|f(\vec{\mathbf{t}},\vec{\mathbf{s}})-f(\vec{\mathbf{s}},\vec{\mathbf{s}})|\leq w_{\tilde{\xi}}(d_0(\vec{\mathbf{t}},\vec{\mathbf{s}}))$ for all $\vec{\mathbf{t}},\vec{\mathbf{s}}\in \mathbb{T}^n$. Then \eqref{m0530-3} gives
		\begin{equation*}
			\begin{split}
				&\ \ \ \ \mathbf{E}\bigl[\big|K^{T_{-t_m}(\vec{\mathbf{s}})}\big(t_m,0,\tilde{\xi}^m_{T_{-t_m}(\vec{\mathbf{s}})}\big)\circ \theta_{-t_m}-K^{T_{-t}(\vec{\mathbf{s}})}\big(t_m,0,\tilde{\xi}_{T_{-t}(\vec{\mathbf{s}})}\big)\circ \theta_{-t_m}\big|^2\bigr]\\
				&\leq e^{(8L^2+4L+1)T}\bigl(2\|\hat{\xi}^m-\hat{\xi}\|_{\mathcal{G}_2}^2+2w_{\tilde{\xi}}(n|t_m-t|)+2Tw(n|t_m-t|)\bigr).
			\end{split}
		\end{equation*}
		Thus,
		\begin{equation}\label{m0530-4}
			\lim_{m\to \infty}\sup_{\vec{\mathbf{s}}\in \mathbb{T}^n}\mathbf{E}\bigl[\big|K^{T_{-t_m}(\vec{\mathbf{s}})}\big(t_m,0,\tilde{\xi}^m_{T_{-t_m}(\vec{\mathbf{s}})}\big)\circ \theta_{-t_m}-K^{T_{-t}(\vec{\mathbf{s}})}\big(t_m,0,\tilde{\xi}_{T_{-t}(\vec{\mathbf{s}})}\big)\circ \theta_{-t_m}\big|^2\bigr]=0.
		\end{equation}
		Now for any $m\geq 1$, set
		\begin{equation*}
			F_m(\vec{\mathbf{s}}):=\mathbf{E}\bigl[\big|K^{T_{-t}(\vec{\mathbf{s}})}\big(t_m,0,\tilde{\xi}_{T_{-t}(\vec{\mathbf{s}})}\big)\circ \theta_{-t_m}-K^{T_{-t}(\vec{\mathbf{s}})}\big(t,0,\tilde{\xi}_{T_{-t}(\vec{\mathbf{s}})}\big)\circ \theta_{-t}\big|^2\bigr], \ \text{ for all } \ \vec{\mathbf{s}}\in \mathbb{T}^n.
		\end{equation*}
		By \eqref{n0529-1}, we know that 
		\begin{equation*}
			\lim_{m\to \infty}F_m(\vec{\mathbf{s}})=0, \ \text{ for all } \ \vec{\mathbf{s}}\in \mathbb{T}^n.
		\end{equation*}
		On the other hand, according to \eqref{0528-1}, \eqref{n0530-1} and H{$\ddot{\rm o}$}lder inequality, we know that for any $\vec{\mathbf{t}},\vec{\mathbf{s}}\in \mathbb{T}^n$
		\begin{equation*}
			\begin{split}
				&\ \ \ \ |F_m(\vec{\mathbf{t}})-F_m(\vec{\mathbf{s}})|\\
				&\leq 4\sqrt{2C_T(1+\|\tilde{\xi}\|_{C(\mathbb{T}^n;L^2)}^2)}\Big(\mathbf{E}\bigl[\big|K^{T_{-t}(\vec{\mathbf{t}})}\big(t_m,0,\tilde{\xi}_{T_{-t}(\vec{\mathbf{t}})}\big)-K^{T_{-t}(\vec{\mathbf{s}})}\big(t_m,0,\tilde{\xi}_{T_{-t}(\vec{\mathbf{s}})}\big)\big|^2\bigr]\\
				&\qquad\qquad\qquad\qquad\qquad\qquad \ \ +\mathbf{E}\bigl[\big|K^{T_{-t}(\vec{\mathbf{t}})}\big(t,0,\tilde{\xi}_{T_{-t}(\vec{\mathbf{t}})}\big)-K^{T_{-t}(\vec{\mathbf{s}})}\big(t,0,\tilde{\xi}_{T_{-t}(\vec{\mathbf{s}})}\big)\big|^2\bigr]\Big)^{\frac{1}{2}}\\
				&\leq 8e^{(4L^2+2L+1)T}\sqrt{C_T(1+\|\tilde{\xi}\|_{C(\mathbb{T}^n;L^2)}^2)\bigl(\mathbf{E}\bigl[|\tilde{\xi}_{T_{-t}(\vec{\mathbf{t}})}-\tilde{\xi}_{T_{-t}(\vec{\mathbf{s}})}|^2\bigr]+2Tw\big(d_0(\vec{\mathbf{t}},\vec{\mathbf{s}})\big)\bigr)}\\
				&\leq 8e^{(4L^2+2L+1)T}\sqrt{C_T(1+\|\tilde{\xi}\|_{C(\mathbb{T}^n;L^2)}^2)\bigl(w_{\tilde{\xi}}\big(d_0(\vec{\mathbf{t}},\vec{\mathbf{s}})\big)+2Tw\big(d_0(\vec{\mathbf{t}},\vec{\mathbf{s}})\big)\bigr)},
			\end{split}
		\end{equation*}
		Hence, $\{F_m\}_{m\geq 1}\subset C(\mathbb{T}^n;\mathbb{R}^+)$ is equicontinuous. Since $\mathbb{T}^n$ is compact and $F_m$ converges to $0$ pointwise, it is equivalent that $F_m$ converges uniformly, i.e.,
		\begin{equation}\label{m0530-5}
			\lim_{m\to \infty}\sup_{\vec{\mathbf{s}}\in \mathbb{T}^n}\mathbf{E}\bigl[\big|K^{T_{-t}(\vec{\mathbf{s}})}\big(t_m,0,\tilde{\xi}_{T_{-t}(\vec{\mathbf{s}})}\big)\circ \theta_{-t_m}-K^{T_{-t}(\vec{\mathbf{s}})}\big(t,0,\tilde{\xi}_{T_{-t}(\vec{\mathbf{s}})}\big)\circ \theta_{-t}\big|^2\bigr]=0.
		\end{equation}
		Thus, we conclude \eqref{m0530-1} from \eqref{m0530-4} and \eqref{m0530-5}.
	\end{proof}

    \begin{remark}
        For any fixed $t\in\mathbb{R}$, consider the transformation $\hat{\theta}_t=(T_t, \theta_t): \hat{\Omega}\to \hat{\Omega}$, where $\hat{\Omega}=\mathbb{T}^n\times\Omega$. For any $\tilde{\xi}\in C(\mathbb{T}^n;\mathcal{F}_0)$, define $\hat{\tilde{\xi}}(\vec{\mathbf{s}},\omega)=\tilde{\xi}_{\vec{\mathbf{s}}}$. It is easy to verify that  $\hat{\tilde{\xi}}\in L^{\infty}(\hat{\Omega}; L^2(\mathcal{F}_0))$. Define a mapping $f: \hat{\Omega}\times L^2(\mathcal{F}_0)\to L^2(\mathcal{F}_0)$ by
        \begin{equation*}
            f_{(\vec{\mathbf{s}},\omega)}\xi=K^{\vec{\mathbf{s}}}(t,0,\xi)\circ \theta_{-t}
        \end{equation*}
        Then for any $\hat{\xi}={\rm graph}(\tilde{\xi})$, equation \eqref{eq:graph-of-hat-Phi} yields that
        \begin{equation*}
            \hat{\Phi}_t\hat{\xi}(\vec{\mathbf{s}})=K^{T_{-t}(\vec{\mathbf{s}})}\big(t,0,\tilde{\xi}_{T_{-t}(\vec{\mathbf{s}})}\big)\circ \theta_{-t}=f_{\hat{\theta}_t^{-1}(\vec{\mathbf{s}},\omega)}\hat{\tilde{\xi}}(\hat{\theta}_t^{-1}\big(\vec{\mathbf{s}},\omega)\big),
        \end{equation*}
        which is the same as that in \cite[(1.1) on Page 90]{huang2016ergodic}.
    \end{remark}
	
	Set
	\begin{equation*}
		\mathcal{P}_{2,cont}(\hat{\mathbb{X}})=\Big\{\int_{\mathbb{T}^n}\big(\delta_{\vec{\mathbf{s}}}\times \tilde{\mu}_{\vec{\mathbf{s}}}\big) \nu(\mathrm{d}\vec{\mathbf{s}}): \tilde{\mu}\in C(\mathbb{T}^n;\mathcal{P}_{2}(\mathbb{R}^d)), \ \nu\in \mathcal{P}(\mathbb{T}^n)\Big\}\subset \mathcal{P}_2(\hat{\mathbb{X}}).
	\end{equation*}
	For $\hat{\xi}={\rm graph}(\tilde{\xi})\in \mathcal{G}_{2}$, denote $\tilde{\mu}_{\vec{\mathbf{s}}}=\mathscr{L}(\tilde{\xi}_{\vec{\mathbf{s}}})$. The following result shows that the equation \eqref{eq:new1202-1} (i.e., \eqref{0517} below) holds on $\mathcal{P}_{2,cont}(\hat{\mathbb{X}})$.

    \begin{lemma}\label{lemma 0531}
		Assume Assumptions \ref{A3}, \ref{Quasi-periodic condition} and \ref{One-sided Lip of tilde b} hold. For any $\nu\in \mathcal{P}(\mathbb{T}^n)$ and $\hat{\xi}={\rm graph}(\tilde{\xi})\in \mathcal{G}_{2}$, we have that $\vec{\mathbf{s}}\mapsto \tilde{\mu}_{\vec{\mathbf{s}}}\in C(\mathbb{T}^n; \mathcal{P}_2(\mathbb{R}^d))$. Moreover,
		\begin{equation}\label{0531-2}
			\mathscr{L}_{\nu}(\hat{\xi})=\int_{\mathbb{T}^n}\big(\delta_{\vec{\mathbf{s}}}\times \tilde{\mu}_{\vec{\mathbf{s}}}\big) \nu(\mathrm{d}\vec{\mathbf{s}})\in \mathcal{P}_{2,cont}(\hat{\mathbb{X}}),
		\end{equation}
		and
		\begin{equation}\label{0531-3}
			\mathscr{L}_{\nu}(\hat{\Phi}_t\hat{\xi})=\int_{\mathbb{T}^n}\big(\delta_{T_t(\vec{\mathbf{s}})}\times P^{\vec{\mathbf{s}},*}(t,0)\tilde{\mu}_{\vec{\mathbf{s}}}\big) \nu(\mathrm{d}\vec{\mathbf{s}})\in \mathcal{P}_{2,cont}(\hat{\mathbb{X}}).
		\end{equation}
		In particular,
		\begin{equation}\label{0517}
			\hat{P}_t^* \int_{\mathbb{T}^n}\big(\delta_{\vec{\mathbf{s}}}\times \tilde{\mu}_{\vec{\mathbf{s}}}\big) \nu(\mathrm{d}\vec{\mathbf{s}})=\int_{\mathbb{T}^n}\big(\delta_{T_t(\vec{\mathbf{s}})}\times P^{\vec{\mathbf{s}},*}(t,0)\tilde{\mu}_{\vec{\mathbf{s}}}\big) \nu(\mathrm{d}\vec{\mathbf{s}}).
		\end{equation}	
	\end{lemma}

	\begin{proof}
		Since $\hat{\xi}={\rm graph}(\tilde{\xi})\in \mathcal{G}_{2}$, we have $\tilde{\xi}\in C(\mathbb{T}^n; L^{2}(\mathcal{F}_0))$. Then $\tilde{\mu}_{\cdot}\in C(\mathbb{T}^n;\mathcal{P}_{2}(\mathbb{R}^d))$ follows by the following inequality
		\begin{equation*}
			\mathcal{W}_{2}(\tilde{\mu}_{\vec{\mathbf{s}}},\tilde{\mu}_{\vec{\mathbf{t}}})=\mathcal{W}_{2}\big(\mathscr{L}(\tilde{\xi}_{\vec{\mathbf{s}}}),\mathscr{L}(\tilde{\xi}_{\vec{\mathbf{t}}})\big)\leq \|\tilde{\xi}_{\vec{\mathbf{s}}}-\tilde{\xi}_{\vec{\mathbf{t}}}\|_{L^2(\mathcal{F}_0)}.
		\end{equation*}
		Moreover, for any $f\in C_b(\hat{\mathbb{X}})$, it follows that
		\begin{equation}
			\begin{split}
				\int_{\hat{\mathbb{X}}}f\mathrm{d}\int_{\mathbb{T}^n}\big(\delta_{\vec{\mathbf{s}}}\times \tilde{\mu}_{\vec{\mathbf{s}}}\big) \nu(\mathrm{d}\vec{\mathbf{s}})
				&=\int_{\mathbb{T}^n}\int_{\mathbb{R}^d}f(\vec{\mathbf{s}};x)\mathscr{L}(\tilde{\xi}_{\vec{\mathbf{s}}})(\mathrm{d}x) \nu(\mathrm{d}\vec{\mathbf{s}})\\
				&=\int_{\mathbb{T}^n}\int_{\Omega}f(\vec{\mathbf{s}};\tilde{\xi}_{\vec{\mathbf{s}}}(\omega))\mathbf{P}(\mathrm{d}\omega) \nu(\mathrm{d}\vec{\mathbf{s}})\\
				&=\int_{\hat{\Omega}}f\big(\hat{\xi}(\vec{\mathbf{s}},\omega)\big)(\nu\times\mathbf{P})(\mathrm{d}\vec{\mathbf{s}}\times\mathrm{d}\omega)\\
				&=\int_{\hat{\mathbb{X}}}f\mathrm{d}\mathscr{L}_{\nu}(\hat{\xi}),
			\end{split}
		\end{equation}
		and
		\begin{equation}
			\begin{split}
				\int_{\hat{\mathbb{X}}}f\mathrm{d}\int_{\mathbb{T}^n}\big(\delta_{T_t(\vec{\mathbf{s}})}\times P^{\vec{\mathbf{s}},*}(t,0)\tilde{\mu}_{\vec{\mathbf{s}}}\big) \nu(\mathrm{d}\vec{\mathbf{s}})
				&=\int_{\mathbb{T}^n}\int_{\mathbb{R}^d}f\big(T_t(\vec{\mathbf{s}});x\big)\mathscr{L}\big(K^{\vec{\mathbf{s}}}(t,0,\tilde{\xi}_{\vec{\mathbf{s}}})\big)(\mathrm{d}x) \nu(\mathrm{d}\vec{\mathbf{s}})\\
				&=\int_{\mathbb{T}^n}\int_{\Omega}f\big(T_t(\vec{\mathbf{s}});K^{\vec{\mathbf{s}}}(t,0,\tilde{\xi}_{\vec{\mathbf{s}}})(\omega)\big)\mathbf{P}(\mathrm{d}\omega) \nu(\mathrm{d}\vec{\mathbf{s}})\\
				&=\int_{\mathbb{T}^n}\int_{\Omega}f\big(T_t(\vec{\mathbf{s}});K^{\vec{\mathbf{s}}}(t,0,\tilde{\xi}_{\vec{\mathbf{s}}})(\theta_{-t}\omega)\big)\mathbf{P}(\mathrm{d}\omega) \nu(\mathrm{d}\vec{\mathbf{s}})\\
				&=\int_{\hat{\Omega}}f\big(\hat{\Phi}_t\hat{\xi}(\vec{\mathbf{s}},\omega)\big)(\nu\times\mathbf{P})(\mathrm{d}\vec{\mathbf{s}}\times\mathrm{d}\omega)\\
				&=\int_{\hat{\mathbb{X}}}f\mathrm{d}\mathscr{L}_{\nu}(\hat{\Phi}_t\hat{\xi}).
			\end{split}
		\end{equation}
		Thus, \eqref{0531-2} and \eqref{0531-3} hold true. Finally, \eqref{0517} follows from \eqref{0530-3}, \eqref{0531-2} and \eqref{0531-3} trivially.
	\end{proof}

For any $p\geq 1$, recall the $p$-Wasserstein space
\begin{equation*}
        \mathcal{P}_p(\hat{\mathbb{X}}):=\Big\{\hat{\mu}\in \mathcal{P}(\hat{\mathbb{X}}): \int_{\hat{\mathbb{X}}}\hat{d}^p(\hat{x},\hat{0})\hat{\mu}(\mathrm{d}\hat{x})<\infty \Big\}.
  \end{equation*}
Since $\mathbb{T}^n$ is a compact set, it is easy to check that 
	\begin{equation*}
		\mathcal{P}_p(\hat{\mathbb{X}})=\Big\{\hat{\mu}\in \mathcal{P}(\hat{\mathbb{X}}): \int_{\hat{\mathbb{X}}}|x|^p\hat{\mu}(\mathrm{d}\vec{\mathbf{s}}, \mathrm{d}x)<\infty\Big\}.
	\end{equation*}
For any $\hat{\mu}\in\mathcal{P}_2(\hat{\mathbb{X}})\subset\mathcal{P}(\hat{\mathbb{X}})$, it can be written as the following form:
	\begin{equation*}
		\hat{\mu}=\int_{\mathbb{T}^n}\big(\delta_{\vec{\mathbf{s}}}\times \tilde{\mu}_{\vec{\mathbf{s}}}\big) \nu(\mathrm{d}\vec{\mathbf{s}}),
	\end{equation*}
	where $\nu$ is the marginal measure on $\mathbb{T}^n$ and $\{\tilde{\mu}_{\vec{\mathbf{s}}}\}_{\vec{\mathbf{s}}}$ is the unique disintegration of $\hat{\mu}$. The following lemma shows that $\hat{P}^*_t$ maps $\mathcal{P}_2(\hat{\mathbb{X}})$ into itself.

	\begin{lemma}\label{lem:hat-P*}
		Assume Assumptions \ref{A3}, \ref{Quasi-periodic condition} and \ref{One-sided Lip of tilde b} hold. For any $t\geq 0$, $\hat{P}_t^*\hat{\mu}$ defined by \eqref{eq:new1202-1} for all $\hat{\mu}\in\mathcal{P}_2(\hat{\mathbb{X}})$ satisfies $\hat{P}_t^*\hat{\mu}\in\mathcal{P}_2(\hat{\mathbb{X}})$.
	\end{lemma}
	\begin{proof}
		For any $\hat{\mu}\in\mathcal{P}_2(\hat{\mathbb{X}})$, let $\nu$ and $\{\tilde{\mu}_{\vec{\mathbf{s}}}\}_{\vec{\mathbf{s}}}$ be the marginal measure on $\mathbb{T}^n$ and the disintegration of $\hat{\mu}$. We need to show that for any $t\geq 0$,
        \begin{equation*}
            \hat{P}_t^*\hat{\mu}=\int_{\mathbb{T}^n}\big(\delta_{T_t(\vec{\mathbf{s}})}\times P^{\vec{\mathbf{s}},*}(t,0)\tilde{\mu}_{\vec{\mathbf{s}}}\big) \nu(\mathrm{d}\vec{\mathbf{s}})=\int_{\mathbb{T}^n}\big(\delta_{\vec{\mathbf{s}}}\times P^{T_{-t}(\vec{\mathbf{s}}),*}(t,0)\tilde{\mu}_{T_{-t}(\vec{\mathbf{s}})}\big) \nu\circ T_t^{-1}(\mathrm{d}\vec{\mathbf{s}})\in \mathcal{P}_2(\hat{\mathbb{X}}).
        \end{equation*}
        We first give the following claim, whose proof will be given later.

		\noindent\textbf{Claim}. For any $t\geq 0$ and $\epsilon>0$, there exists an open set $O^{\epsilon}_t\subset \mathbb{T}^n$ with $\nu(T_t^{-1}(O^{\epsilon}_t))<\epsilon$ such that $\tilde{\mu}_{T_{-t}\cdot}$ is $\mathcal{W}_2$-continuous on $\mathbb{T}^n\setminus O^{\epsilon}_t$.

        Now, by Lemma \ref{Lemma 0517-1} and the Claim, $P^{T_{-t}(\vec{\mathbf{s}}),*}(t,0)\tilde{\mu}_{T_{-t}(\vec{\mathbf{s}})}$ is $\mathcal{W}_2$-continuous for $\vec{\mathbf{s}}\in \mathbb{T}^n\setminus O^{\epsilon}_t$. Hence, for any  $f\in C_b(\hat{\mathbb{X}})$, 
        \begin{equation*}
            \int_{\mathbb{R}^d}f(\vec{\mathbf{s}},\cdot){\rm d}P^{T_{-t}(\vec{\mathbf{s}}),*}(t,0)\tilde{\mu}_{T_{-t}(\vec{\mathbf{s}})} 
        \end{equation*}
        is $\nu\circ T_t^{-1}$-integrable. Hence, it is easy to check that $\hat{P}_t^*\hat{\mu}: C_b(\hat{\mathbb{X}})\to \mathbb{R}$, defined by
        \begin{equation*}
            \hat{P}_t^*\hat{\mu} (f):=\int_{\mathbb{T}^n}\bigg(\int_{\mathbb{R}^d}f(\vec{\mathbf{s}},\cdot){\rm d}P^{T_{-t}(\vec{\mathbf{s}}),*}(t,0)\tilde{\mu}_{T_{-t}(\vec{\mathbf{s}})}\bigg) \nu\circ T_t^{-1}({\rm d}\vec{\mathbf{s}}),
        \end{equation*}
        is a positive bounded linear functional defined on $C_b(\hat{\mathbb{X}})$, and maps the 1-function to itself. By the Riesz representation theorem, we know that $\hat{P}_t^*\hat{\mu}\in \mathcal{P}(\hat{\mathbb{X}})$. 

        On the other hand, let $\{\tilde{\xi}_{\vec{\mathbf{s}}}\}_{\vec{\mathbf{s}}\in \mathbb{T}^n}\subset L^2(\mathcal{F}_0)$ with $\mathscr{L}(\tilde{\xi}_{\vec{\mathbf{s}}})=\tilde{\mu}_{\vec{\mathbf{s}}}$ for all $\vec{\mathbf{s}}\in \mathbb{T}^n$. It follows from \eqref{n0530-1} that
		\begin{equation*}
			\begin{split}
				\int_{\mathbb{T}^n}\big(\delta_{T_t(\vec{\mathbf{s}})}\times P^{\vec{\mathbf{s}},*}(t,0)\tilde{\mu}_{\vec{\mathbf{s}}}\big) \nu(\mathrm{d}\vec{\mathbf{s}}) (|\cdot|^2)&=\int_{\mathbb{T}^n} \mathbf{E}\bigl[\big|K^{T_{-t}(\vec{\mathbf{s}})}\big(t,0,\tilde{\xi}_{T_{-t}(\vec{\mathbf{s}})}\big)\big|^2\bigr]\nu\circ T_t^{-1}(\mathrm{d}\vec{\mathbf{s}})\\
				&\leq \int_{\mathbb{T}^n} C_t\big(1+\big\|\tilde{\mu}_{T_{-t}(\vec{\mathbf{s}})}\big\|_2^2\big)\nu\circ T_t^{-1}(\mathrm{d}\vec{\mathbf{s}})\\
				&=C_t\Bigl(1+\int_{\mathbb{T}^n}\big\|\tilde{\mu}_{\vec{\mathbf{s}}}\big\|_2^2\nu(\mathrm{d}\vec{\mathbf{s}})\Bigr)\\
				&=C_t\Bigl(1+\int_{\hat{\mathbb{X}}}|x|^2\hat{\mu}(\mathrm{d}\hat{x})\Bigr)<\infty.
			\end{split}
		\end{equation*}
		Hence, 
		\begin{equation*}
			\hat{P}_t^*\hat{\mu}=\int_{\mathbb{T}^n}\delta_{T_t(\vec{\mathbf{s}})}\times P^{\vec{\mathbf{s}},*}(t,0)\tilde{\mu}_{\vec{\mathbf{s}}} \nu(\mathrm{d}\vec{\mathbf{s}})\in \mathcal{P}_2(\hat{\mathbb{X}}).
		\end{equation*}

        Finally, we present the proof of Claim.

		\noindent\textbf{Proof of Claim}: Note that
		\begin{equation*}
			\int_{\mathbb{T}^n}\int_{\mathbb{R}^d}|x|^2\tilde{\mu}_{T_{-t}(\vec{\mathbf{s}})}(\mathrm{d}x)\nu\circ T_t^{-1}(\mathrm{d}\vec{\mathbf{s}})=\int_{\mathbb{T}^n}\int_{\mathbb{R}^d}|x|^2\tilde{\mu}_{\vec{\mathbf{s}}}(\mathrm{d}x)\nu(\mathrm{d}\vec{\mathbf{s}})=\int_{\hat{\mathbb{X}}}|x|^2\hat{\mu}(\mathrm{d}\hat{x})<\infty.
		\end{equation*}
		Hence, $\tilde{\mu}_{\vec{\mathbf{s}}}\in \mathcal{P}_2(\mathbb{R}^d)$, $\nu$-a.s. Without losing any generality, we assume $\tilde{\mu}_{\vec{\mathbf{s}}}\in \mathcal{P}_2(\mathbb{R}^d)$ for all $\vec{\mathbf{s}}\in \mathbb{T}^n$. For any $N\geq 1$, set
		\begin{equation*}
			A_N:=\{\vec{\mathbf{s}}\in \mathbb{T}^n: \|\tilde{\mu}_{T_{-t}(\vec{\mathbf{s}})}\|_2>N\}.
		\end{equation*}
		Then the Chebyshev inequality gives
		\begin{equation*}
			\nu(T_t^{-1}(A_N))\leq \frac{1}{N^2}\int_{\hat{\mathbb{X}}}|x|^2\hat{\mu}(\mathrm{d}\hat{x})\to 0 \ \text{ as } \ N\to \infty.
		\end{equation*}
		Now we choose $N_0$ large enough such that $\nu(T_t^{-1}(A_{N_0}))<\frac{\epsilon}{3}$. Denote by $C_0(\mathbb{R}^d;\mathbb{R})$ the collection of continuous functions that vanish at infinity, with the supremum norm. It is well-known that $C_0(\mathbb{R}^d;\mathbb{R})$ is separable. Let us choose the sequence $\{f_n\}_{n\geq 1}$ being dense in $C_0(\mathbb{R}^d;\mathbb{R})$ and let $f_0(x)=|x|^2$. Note that 
		\begin{equation*}
			F_n(\vec{\mathbf{s}})=\int_{\mathbb{R}^d}f_n(x)\tilde{\mu}_{T_{-t}(\vec{\mathbf{s}})}(\mathrm{d}x)
		\end{equation*}
		is measurable for all $n\geq 0$. Since $(\mathbb{T}^n,\mathcal{B}(\mathbb{T}^n),\nu\circ T_t^{-1})$ is a Radon space, Lusin's theorem indicates that there exists a measurable set $\Gamma$ of $\mathbb{T}^n$ such that 
		\begin{equation*}
			\nu(T_t^{-1}(\Gamma))<\frac{\epsilon}{3}, \ \text{ and } \ F_n \text{ is continuous on } \mathbb{T}^n\setminus\Gamma \ \text{ for all } \ n\geq 0.
		\end{equation*}
		Since $\nu\circ T_t^{-1}$ is an outer regular probability measure, there is an open set $O^{\epsilon}_t\supset A_{N_0}\cup \Gamma$ such that $\nu(T_t^{-1}(O^{\epsilon}_t))<\epsilon$. Then we will show that $\tilde{\mu}_{T_{-t}\cdot}$ is $\mathcal{W}_2$-continuous on $\mathbb{T}^n\setminus O^{\epsilon}_t$. Since $\{f_n\}_{n\geq 1}$ is dense in $C_0(\mathbb{R}^d;\mathbb{R})$, we know that the function $F$ given by
		\begin{equation*}
			F(\vec{\mathbf{s}})=\int_{\mathbb{R}^d}f(x)\tilde{\mu}_{T_{-t}(\vec{\mathbf{s}})}(\mathrm{d}x)
		\end{equation*}
		is continuous on $\mathbb{T}^n\setminus O^{\epsilon}_t$ for all $f\in C_0(\mathbb{R}^d;\mathbb{R})$. Hence, $\tilde{\mu}_{T_{-t}\cdot}$ is continuous on $\mathbb{T}^n\setminus O^{\epsilon}_t$ with respect to the vague topology. Note also that $O^{\epsilon}_t\supset A_{N_0}$, by the construction of $A_{N_0}$, we conclude that $\{\tilde{\mu}_{T_{-t}(\vec{\mathbf{s}})}\}_{\vec{\mathbf{s}}\in \mathbb{T}^n\setminus O^{\epsilon}_t}$ is tight. Thus, $\tilde{\mu}_{T_{-t}\cdot}$ is continuous on $\mathbb{T}^n\setminus O^{\epsilon}_t$ with respect to the weak topology. Moreover, since $F_0$ is continuous, Theorem 6.9 in \cite{Villani2009} implies that $\tilde{\mu}_{T_{-t}\cdot}$ is $\mathcal{W}_2$-continuous on $\mathbb{T}^n\setminus O^{\epsilon}_t$. The claim is proved.
	\end{proof}
    
The following result indicates that $\{\hat{P}^*_t\}_{t\geq 0}$ is not only a family of maps from $\mathcal{P}_2(\hat{\mathbb{X}})$ to itself, but also constitutes a continuous dynamical system on $\mathcal{P}_2(\hat{\mathbb{X}})$.
    
    \begin{proposition}\label{prop:hat-P*-CDS}
        Assume Assumptions \ref{A3}, \ref{Quasi-periodic condition} and \ref{One-sided Lip of tilde b} hold. Then $\hat{P}^*:\mathbb{R}^+\times \mathcal{P}_2(\hat{\mathbb{X}})\to \mathcal{P}_2(\hat{\mathbb{X}})$ is a continuous dynamical system.
    \end{proposition}
		
    \begin{proof}
        For any $t,s\geq 0$ and $\hat{\mu}\in \mathcal{P}_2(\hat{\mathbb{X}})$ with the following disintegration form
        \begin{equation*}
            \hat{\mu}=\int_{\mathbb{T}^n}\big(\delta_{\vec{\mathbf{s}}}\times \tilde{\mu}_{\vec{\mathbf{s}}}\big) \nu(\mathrm{d}\vec{\mathbf{s}}),
        \end{equation*}
        it is easy to verify that
        \begin{equation*}
            \begin{split}
                \hat{P}^*_t\circ\hat{P}^*_s\hat{\mu}&=\hat{P}^*_t\int_{\mathbb{T}^n}\big(\delta_{T_s(\vec{\mathbf{s}})}\times P^{\vec{\mathbf{s}},*}(s,0)\tilde{\mu}_{\vec{\mathbf{s}}}\big) \nu(\mathrm{d}\vec{\mathbf{s}})\\
                &=\int_{\mathbb{T}^n}\big(\delta_{T_{t+s}(\vec{\mathbf{s}})}\times P^{T_s(\vec{\mathbf{s}}),*}(t,0)P^{\vec{\mathbf{s}},*}(s,0)\tilde{\mu}_{\vec{\mathbf{s}}}\big) \nu(\mathrm{d}\vec{\mathbf{s}})\\
                &=\int_{\mathbb{T}^n}\big(\delta_{T_{t+s}(\vec{\mathbf{s}})}\times P^{\vec{\mathbf{s}},*}(t+s,s)P^{\vec{\mathbf{s}},*}(s,0)\tilde{\mu}_{\vec{\mathbf{s}}}\big) \nu(\mathrm{d}\vec{\mathbf{s}})\\
                &=\int_{\mathbb{T}^n}\big(\delta_{T_{t+s}(\vec{\mathbf{s}})}\times P^{\vec{\mathbf{s}},*}(t+s,0)\tilde{\mu}_{\vec{\mathbf{s}}} \big) \nu(\mathrm{d}\vec{\mathbf{s}})\\
                &=\hat{P}^*_{t+s}\hat{\mu}.
            \end{split}
        \end{equation*}
        Then it remains to show that $\hat{P}^*:\mathbb{R}^+\times \mathcal{P}_2(\hat{\mathbb{X}})$ is continuous. We first state the following two claims, whose proofs will be given later.

        \noindent \textbf{Claim 1:} For any $\epsilon>0$ and $T>0$, there exists $C_{T,\epsilon}>0$ such that 
        \begin{equation}\label{eq:1127}
            \mathcal{D}_2(\hat{P}^*_t\hat{\mu}_1,\hat{P}^*_t\hat{\mu}_2)\leq \epsilon+C_{T,\epsilon}\mathcal{D}_2(\hat{\mu}_1,\hat{\mu}_2), \ \text{ for all } \ \hat{\mu}_1,\hat{\mu}_2\in \mathcal{P}_2(\hat{\mathbb{X}}), \ 0\leq t\leq T.
        \end{equation}

        \noindent \textbf{Claim 2:} For any $\hat{\mu}\in \mathcal{P}_2(\hat{\mathbb{X}})$, $\lim_{t\to 0}\mathcal{D}_2(\hat{P}^*_t\hat{\mu},\hat{\mu})=0$.

        \noindent To show the continuity of $\hat{P}^*$, we need to prove that 
        \begin{equation}\label{eq:n1128}
            \lim_{m\to\infty}\mathcal{D}_2(\hat{P}^*_{t_m}\hat{\mu}_m,\hat{P}^*_t\hat{\mu})=0,
        \end{equation}
        for any sequences $\{t_m,t\}_{m\geq 1}\subset \mathbb{R}^+$ and $\{\hat{\mu}_m,\hat{\mu}\}_{m\geq 1}\subset \mathcal{P}_2(\hat{\mathbb{X}})$  with
        \begin{equation*}
            \lim_{m\to \infty}\big(|t_m-t|+\mathcal{D}_2(\hat{\mu}_m,\hat{\mu})\big)=0.
        \end{equation*}
        Let $T=\sup_{m\geq 1}t_m$. By Claim 1, for any $\epsilon>0$, there exists $C_{T,\epsilon}>0$ such that for all $m\geq 1$,
        \begin{equation}\label{eq:n1128-1}
            \mathcal{D}_2(\hat{P}^*_{t_m}\hat{\mu}_m,\hat{P}^*_{t_m}\hat{\mu})\leq \epsilon+C_{T,\epsilon}\mathcal{D}_2(\hat{\mu}_m,\hat{\mu}),
        \end{equation}
        and 
        \begin{equation}\label{eq:n1128-2}
            \mathcal{D}_2(\hat{P}^*_{t_m}\hat{\mu},\hat{P}^*_{t}\hat{\mu})=\mathcal{D}_2(\hat{P}^*_{t_m\wedge t}\hat{P}^*_{(t_m-t)\vee 0}\hat{\mu},\hat{P}^*_{t_m\wedge t}\hat{P}^*_{(t_m-t)\wedge 0}\hat{\mu})\leq \epsilon+C_{T,\epsilon}\mathcal{D}_2(\hat{P}^*_{|t_m-t|}\hat{\mu},\hat{\mu}).
        \end{equation}
        Then we conclude from \eqref{eq:n1128-1} and \eqref{eq:n1128-2} that
        \begin{equation*}
            \limsup_{m\to\infty}\mathcal{D}_2(\hat{P}^*_{t_m}\hat{\mu}_m,\hat{P}^*_t\hat{\mu})\leq 2\epsilon+C_{T,\epsilon}\lim_{m\to\infty}\big(\mathcal{D}_2(\hat{\mu}_m,\hat{\mu})+\mathcal{D}_2(\hat{P}^*_{|t_m-t|}\hat{\mu},\hat{\mu})\big)=2\epsilon.
        \end{equation*}
        Since $\epsilon>0$ can be arbitrary, we obtain \eqref{eq:n1128}.

        Finally, we give the proofs of Claim 1 and Claim 2.

        \noindent \textbf{Proof of Claim 1:} By \cite[Theorem 4.1]{Villani2009}, for any $\hat{\mu}_1,\hat{\mu}_2\in \mathcal{P}_2(\hat{\mathbb{X}})$, there exists an optimal coupling $\hat{\pi}^*$ of $\hat{\mu}_1$ and $\hat{\mu}_2$, i.e.,
        \begin{equation*}
            \mathcal{D}_2^2(\hat{\mu}_1,\hat{\mu}_2)=\int_{\hat{\mathbb{X}}\times \hat{\mathbb{X}}}\hat{d}^2(\hat{x},\hat{y})\hat{\pi}^*(\mathrm{d}\hat{x},\mathrm{d}\hat{y}).
        \end{equation*}
        Let $\nu_i$ and $\{\tilde{\mu}^i_{\vec{\mathbf{s}}}\}_{\vec{\mathbf{s}}}$ be the marginal measure on $\mathbb{T}^n$ and the disintegration of $\hat{\mu}_i$, $i=1,2$, respectively. Then $\{\hat{\mu}_i\}_{i=1,2}$ have the following forms
        \begin{equation*}
            \hat{\mu}_i=\int_{\mathbb{T}^n}\big(\delta_{\vec{\mathbf{s}}}\times \tilde{\mu}^i_{\vec{\mathbf{s}}} \big)\nu_i(\mathrm{d}\vec{\mathbf{s}}), \ i=1,2.
        \end{equation*}
        Let $\gamma$ be the marginal of $\hat{\pi}^*$ on $\mathbb{T}^n\times \mathbb{T}^n$ and $\{\tilde{\hat{\pi}}^*_{\vec{\mathbf{s}}_1,\vec{\mathbf{s}}_2}\}_{(\vec{\mathbf{s}}_1,\vec{\mathbf{s}}_2)\in\mathbb{T}^n\times\mathbb{T}^n}$ be the corresponding disintegration of $\hat{\pi}^*$, i.e., 
        \begin{equation*}
            \gamma(A\times B):=\hat{\pi}^*\big((A\times \mathbb{R}^d)\times (B\times \mathbb{R}^d)\big), \ \text{ for any } \ A,B\in \mathcal{B}(\mathbb{T}^n),
        \end{equation*}
        and
        \begin{equation*}
            \hat{\pi}^*=\int_{\mathbb{T}^n\times\mathbb{T}^n}\big(\delta_{(\vec{\mathbf{s}}_1,\vec{\mathbf{s}}_2)}\times \tilde{\hat{\pi}}^*_{\vec{\mathbf{s}}_1,\vec{\mathbf{s}}_2}\big) \gamma({\rm d}\vec{\mathbf{s}}_1\times {\rm d}\vec{\mathbf{s}}_2).
        \end{equation*}
        Then $\gamma$ is a coupling of $\nu_1, \nu_2$ and $\tilde{\hat{\pi}}^*_{\vec{\mathbf{s}}_1,\vec{\mathbf{s}}_2}$ is a coupling of $\tilde{\mu}^1_{\vec{\mathbf{s}}_1}, \tilde{\mu}^2_{\vec{\mathbf{s}}_2}$, $\gamma$-a.s. 
        
        Let $(\eta_1,\eta_2):\Omega\to \mathbb{T}^n\times \mathbb{T}^n$ be a random vector with the law $\gamma$ which is independent of $\mathcal{F}_0$ and $(\tilde{\xi}^1_{\vec{\mathbf{s}}_1,\vec{\mathbf{s}}_2}, \tilde{\xi}^2_{\vec{\mathbf{s}}_1,\vec{\mathbf{s}}_2})\in L^2(\Omega, \mathcal{F}_0; \mathbb{R}^{d}\times\mathbb{R}^d)$ has the law $\tilde{\hat{\pi}}^*_{\vec{\mathbf{s}}_1,\vec{\mathbf{s}}_2}$ for all $(\vec{\mathbf{s}}_1,\vec{\mathbf{s}}_2)\in \mathbb{T}^n\times \mathbb{T}^n$. Then it is easy to verify that
        \begin{equation*}
            \mathscr{L}\big(\big(\eta_1,\tilde{\xi}^1_{\eta_1,\eta_2}, \eta_2,\tilde{\xi}^2_{\eta_1,\eta_2}\big)\big)=\hat{\pi}^*, \ \ \mathscr{L}\big(\big(\eta_1,\tilde{\xi}^1_{\eta_1,\eta_2}\big)\big)=\hat{\mu}_1, \ \ \mathscr{L}\big(\big(\eta_2,\tilde{\xi}^2_{\eta_1,\eta_2}\big)\big)=\hat{\mu}_2.
        \end{equation*}
        Then
        \begin{equation}\label{new1201}
            \mathcal{D}_2^2(\hat{\mu}_1,\hat{\mu}_2)=\mathbf{E}\big[\hat{d}^2\big((\eta_1,\tilde{\xi}^1_{\eta_1,\eta_2}), (\eta_2,\tilde{\xi}^2_{\eta_1,\eta_2})\big)\big]=\mathbf{E}[d_0^2(\eta_1,\eta_2)]+\mathbf{E}\big[\big|\tilde{\xi}^1_{\eta_1,\eta_2}-\tilde{\xi}^2_{\eta_1,\eta_2}\big|^2\big].
        \end{equation}
        On the other hand, by the definition of $\hat{P}^*_t$, it can be seen that for any $t\geq 0$,
        \begin{equation*}
            \mathscr{L}\Big(\big(T_t(\eta_1),K^{\eta_1}\big(t,0,\tilde{\xi}^1_{\eta_1,\eta_2}\big)\circ\theta_{-t}\big)\Big)=\hat{P}^*_t\hat{\mu}_1, \ \ \mathscr{L}\Big(\big(T_t(\eta_2),K^{\eta_2}\big(t,0,\tilde{\xi}^2_{\eta_1,\eta_2}\big)\circ\theta_{-t}\big)\Big)=\hat{P}^*_t\hat{\mu}_2.
        \end{equation*}
        Hence, for any $0\leq t\leq T$,
        \begin{equation}\label{eq:1127-1}
            \begin{split}
                \mathcal{D}_2^2(\hat{P}^*_t\hat{\mu}_1,\hat{P}^*_t\hat{\mu}_2)&\leq \mathbf{E}\Big[\hat{d}^2\Big(\big(T_t(\eta_1),K^{\eta_1}\big(t,0,\tilde{\xi}^1_{\eta_1,\eta_2}\big)\circ\theta_{-t}\big), \big(T_t(\eta_2),K^{\eta_2}\big(t,0,\tilde{\xi}^2_{\eta_1,\eta_2}\big)\circ\theta_{-t}\big)\Big)\Big]\\
                &=\mathbf{E}[d_0^2(\eta_1,\eta_2)]+\mathbf{E}\Big[\Big|K^{\eta_1}\big(t,0,\tilde{\xi}^1_{\eta_1,\eta_2}\big)\circ\theta_{-t}-K^{\eta_2}\big(t,0,\tilde{\xi}^2_{\eta_1,\eta_2}\big)\circ\theta_{-t}\Big|^2\Big].
            \end{split}
        \end{equation}
        Since $(\eta_1,\eta_2)$ is independent of $\mathcal{F}_0$, according to \eqref{0528-1}, we know that there exist a constant $C_T>0$ and an increasing function $w: [0,\infty)\to [0,\infty)$ satisfying $w(0)=0$ and continuous at $0$ such that for all $t\leq T$,
        \begin{equation}\label{eq:1127-2}
            \begin{split}
                &\ \ \ \ \mathbf{E}\Big[\Big|K^{\eta_1}\big(t,0,\tilde{\xi}^1_{\eta_1,\eta_2}\big)\circ\theta_{-t}-K^{\eta_2}\big(t,0,\tilde{\xi}^2_{\eta_1,\eta_2}\big)\circ\theta_{-t}\Big|^2\Big]\\
                &=\mathbf{E}\bigg[\mathbf{E}\Big[\Big|K^{\vec{\mathbf{s}}_1}\big(t,0,\tilde{\xi}^1_{\vec{\mathbf{s}}_1,\vec{\mathbf{s}}_2}\big)\circ\theta_{-t}-K^{\vec{\mathbf{s}}_2}\big(t,0,\tilde{\xi}^2_{\vec{\mathbf{s}}_1,\vec{\mathbf{s}}_2}\big)\circ\theta_{-t}\Big|^2\Big]\bigg|_{(\vec{\mathbf{s}}_1,\vec{\mathbf{s}}_2)=(\eta_1,\eta_2)}\bigg]\\
                &=\mathbf{E}\bigg[\mathbf{E}\Big[\Big|K^{\vec{\mathbf{s}}_1}\big(t,0,\tilde{\xi}^1_{\vec{\mathbf{s}}_1,\vec{\mathbf{s}}_2}\big)-K^{\vec{\mathbf{s}}_2}\big(t,0,\tilde{\xi}^2_{\vec{\mathbf{s}}_1,\vec{\mathbf{s}}_2}\big)\Big|^2\Big]\bigg|_{(\vec{\mathbf{s}}_1,\vec{\mathbf{s}}_2)=(\eta_1,\eta_2)}\bigg]\\
                &\leq C_T\big(\mathbf{E}[w(d_0(\eta_1,\eta_2))]+\mathbf{E}\big[\big|\tilde{\xi}^1_{\eta_1,\eta_2}-\tilde{\xi}^2_{\eta_1,\eta_2}\big|^2\big]\big)
            \end{split}
        \end{equation}
        Since $w$ is continuous at $0$ and $w(0)=0$, for any $\epsilon>0$, there exists $r_{\epsilon}>0$ such that $w(r)<\epsilon/C_T$ for all $r\leq r_{\epsilon}$. Hence,
        \begin{equation}\label{eq:1127-3}
\begin{split}
	 \mathbf{E}[w(d_0(\eta_1,\eta_2))]&\leq \frac{\epsilon}{C_T}+w({\rm diam}(\mathbb{T}^n))\mathbf{P}\{d_0(\eta_1,\eta_2)\geq r_{\epsilon}\}\\
&\leq \frac{\epsilon}{C_T}+\frac{w(\tau_1+\cdots+\tau_n)}{r_{\epsilon}^2}\mathbf{E}[d_0^2(\eta_1,\eta_2)].
\end{split}
        \end{equation}
        Then \eqref{eq:1127} follows from \eqref{new1201}-\eqref{eq:1127-3}.

        \noindent \textbf{Proof of Claim 2:} For any $\hat{\mu}\in \mathcal{P}_2(\hat{\mathbb{X}})$, similar to the proof of Claim 1, let $(\eta,\tilde{\xi}_{\eta}): \Omega\to \hat{\mathbb{X}}$ be a random vector such that $\eta$ is independent of $\mathcal{F}_0$,  $\tilde{\xi}_{\vec{\mathbf{s}}}$ is $\mathcal{F}_0$-measurable for all $\vec{\mathbf{s}}\in\mathbb{T}^n$ and
        \begin{equation*}
            \mathscr{L}\big((\eta,\tilde{\xi}_{\eta})\big)=\hat{\mu}, \ \ \mathscr{L}\Big(\big(T_t(\eta),K^{\eta}\big(t,0,\tilde{\xi}_{\eta}\big)\circ\theta_{-t}\big)\Big)=\hat{P}^*_t\hat{\mu}, \ \text{ for all } \ t\geq 0.
        \end{equation*}
        Then, it follows that
        \begin{equation*}
            \begin{split}
                \mathcal{D}_2^2(\hat{P}^*_t\hat{\mu},\hat{\mu})&\leq\mathbf{E}[d_0^2(T_t(\eta),\eta)]+\mathbf{E}\big[\big|K^{\eta}\big(t,0,\tilde{\xi}_{\eta}\big)\circ \theta_{-t}-\tilde{\xi}_{\eta}\big|^2\big]\\
                &\leq t+\mathbf{E}\Big[\mathbf{E}\big[\big|K^{\vec{\mathbf{s}}}\big(t,0,\tilde{\xi}_{\vec{\mathbf{s}}}\big)\circ \theta_{-t}-\tilde{\xi}_{\vec{\mathbf{s}}}\big|^2\big]\Big|_{\vec{\mathbf{s}}=\eta}\Big].
            \end{split}
        \end{equation*}
        By \eqref{n0529-1}, we know that
        \begin{equation*}
            \lim_{t\to 0}\mathbf{E}\big[\big|K^{\vec{\mathbf{s}}}\big(t,0,\tilde{\xi}_{\vec{\mathbf{s}}}\big)\circ \theta_{-t}-\tilde{\xi}_{\vec{\mathbf{s}}}\big|^2\big]=0, \ \text{ for all } \ \vec{\mathbf{s}}\in\mathbb{T}^n.
        \end{equation*}
        According to \eqref{n0530-1} and the dominated convergence theorem, we conclude that
        \begin{equation*}
            \lim_{t\to 0}\mathbf{E}\Big[\mathbf{E}\big[\big|K^{\vec{\mathbf{s}}}\big(t,0,\tilde{\xi}_{\vec{\mathbf{s}}}\big)\circ \theta_{-t}-\tilde{\xi}_{\vec{\mathbf{s}}}\big|^2\big]\Big|_{\vec{\mathbf{s}}=\eta}\Big]=0.
        \end{equation*}
        Thus, claim 2 is proved.
    \end{proof}

Now, we give the proof of the second main theorem. 
	
    \begin{proof}[Proof of Theorem \ref{thm:existence-of-invariant-measure}]
        We only need to show that $\hat{P}^*$ has an invariant measure $\hat{\mu}$ in $\mathcal{P}_{2}(\hat{\mathbb{X}})$.
        
        According to Theorem \ref{Theorem existence of entrance measure}, equation \eqref{McKean-Vlasov SDE 2} has an entrance measure $\mu\in \mathcal{M}_{2,\vartheta}$. 

Next, we show that $\mu\in \mathcal{M}_p(\mathbb{R})$ for all $p\geq 2$. In fact, since $T_t: \mathbb{T}^n\to\mathbb{T}^n$ is a minimal rotation, and the unique ergodic measure on $\mathbb{T}^n$ with respect to $T_t$ is $\frac{1}{\tau_1\cdots\tau_n}{\rm d}\vec{\mathbf{s}}$, where ${\rm d}\vec{\mathbf{s}}$ is the Lebesgue measure on $\mathbb{T}^n$. Hence, 
\begin{equation*}
	\frac{1}{\tau_1\cdots\tau_n}\int_{\mathbb{T}^n}\big(\tilde{\alpha}_{\vec{\mathbf{s}}}+\tilde{\beta}_{\vec{\mathbf{s}}}\big){\rm d}\vec{\mathbf{s}}=\lim_{T\to\infty}\int_{-T}^0(\alpha_r+\beta_r){\rm d}r<0.
\end{equation*}
According to the same method as in \cite[Lemma 6.3]{Feng-Qu-Zhao2023}, we conclude that there exists $l>0$ such that 
				\begin{equation}\label{0523-1}
					\sup_{t,t_1,\cdots,t_n\in \mathbb{R}}\int_{t}^{t+l}\bigl(\tilde{\alpha}_{r+t_1,\cdots,r+t_n}+\tilde{\beta}_{r+t_1,\cdots,r+t_n}\bigr)\mathrm{d}r<0. 
				\end{equation} 
Then it follows from Remark \ref{rem:bounded-of-entrance-measure} that the entrance measure $\mu\in \mathcal{M}_p(\mathbb{R})$ for all $p\geq 2$. Thus, for all $p\geq 2$, 
        \begin{equation*}
            \sup_{t\in\mathbb{R}}\int_{\mathbb{R}^d}|x|^p\mu_t({\rm d}x)\leq C_p<\infty, \ \text{ for some $C_p>0$ and } \ \ P^*(t,s)\mu_s=\mu_t, \ \text{ for all } \ t\geq s.
        \end{equation*}
        For any $T>0$, set
        \begin{equation*}
            \hat{\mu}_T:=\frac{1}{2T}\int_{-T}^T \big(\delta_{r | \vec{\mathbf{\tau}}}\times \mu_r\big) {\rm d}r,
        \end{equation*}
        where $r|\vec{\mathbf{\tau}}:=(r \mod \tau_1, \cdots, r\mod \tau_n)$. It is easy to check that $\hat{\mu}_T\in \mathcal{P}_{p}(\hat{\mathbb{X}})$ for all $p\geq 2$.
        
        Now we can verify that $\hat{\mu}_T, \ T>0$ is tight. To see this, for any $\epsilon>0$, choose $B_{\epsilon}:=\{x\in\mathbb{R}^d: |x|\leq \sqrt{C_2/\epsilon}\}$. By Chebyshev's inequality, we conclude that $\mu_t(B_{\epsilon}^c)\leq \epsilon$ for all $t\in\mathbb{R}$. Consider the compact set $K_{\epsilon}:=\mathbb{T}^n\times B_{\epsilon}$ of $\hat{\mathbb{X}}$, we have
        \begin{equation*}
            \hat{\mu}_T(K_{\epsilon}^c)=\frac{1}{2T}\int_{-T}^T\mu_r(B_{\epsilon}^c) {\rm d}r\leq \epsilon.
        \end{equation*}
        Hence, $\{\hat{\mu}_T\}_{T>0}\subset \mathcal{P}_{2}(\hat{\mathbb{X}})$ is weakly compact. Thus, there exists a sequence $T_k\uparrow\infty$ and a measure $\hat{\mu}\in \mathcal{P}(\hat{\mathbb{X}})$ such that the sequence of measures $\hat{\mu}_{T_k}$ weakly converges to $\hat{\mu}$ as $k\to \infty$. So, for any $N>0$,
        \begin{equation*}
            \int_{\hat{\mathbb{X}}}(|x|^p\wedge N) \hat{\mu} ({\rm d}\hat{x})=\lim_{k\to\infty}\int_{\hat{\mathbb{X}}}(|x|^p\wedge N) \hat{\mu}_{T_k} ({\rm d}\hat{x})\leq \sup_{k\geq 1}\frac{1}{2T_k}\int_{-T_k}^{T_k}\int_{\mathbb{R}^d}|x|^p\mu_t({\rm d}x){\rm d}t\leq C_p.
        \end{equation*}
        Letting $N\to \infty$, we conclude that $\hat{\mu}\in \mathcal{P}_p(\hat{\mathbb{X}})$ and there exists $C_p>0$ such that
        \begin{equation*}
            \sup_{k\geq 1}\int_{\hat{\mathbb{X}}}\hat{d}^p(\hat{x},\hat{0})\hat{\mu}_{T_{k}}({\rm d}\hat{x}) \vee \int_{\hat{\mathbb{X}}}\hat{d}^p(\hat{x},\hat{0})\hat{\mu}({\rm d}\hat{x}) \leq C_p, \ \text{ for all } \ p\geq 2.
        \end{equation*}
        So, for all $N>0$,
        \begin{equation*}
					\int_{\{\hat{d}^2(\hat{x},\hat{0})\geq N\}}\hat{d}^2(\hat{x},\hat{0})\hat{\mu}_{T_{k}}({\rm d}\hat{x})\leq \frac{1}{N}\int_{\{\hat{d}^2(\hat{x},\hat{0})\geq N\}}\hat{d}^4(\hat{x},\hat{0})\hat{\mu}_{T_{k}}({\rm d}\hat{x})\leq \frac{C_4}{N}, \ \text{ for all } \ k\geq 1,
				\end{equation*}
and
         \begin{equation*}
					\int_{\{\hat{d}^2(\hat{x},\hat{0})\geq N\}}\hat{d}^2(\hat{x},\hat{0})\hat{\mu}({\rm d}\hat{x})\leq \frac{1}{N}\int_{\{\hat{d}^2(\hat{x},\hat{0})\geq N\}}\hat{d}^4(\hat{x},\hat{0})\hat{\mu}({\rm d}\hat{x})\leq \frac{C_4}{N}.
				\end{equation*}   
Hence,
        \begin{equation*}
            \begin{split}
                &\ \ \ \ \bigg|\int_{\hat{\mathbb{X}}}\hat{d}^2(\hat{x},\hat{0})\hat{\mu}_{T_{k}}({\rm d}\hat{x})-\int_{\hat{\mathbb{X}}}\hat{d}^2(\hat{x},\hat{0})\hat{\mu}({\rm d}\hat{x})\bigg|\\
                &\leq \bigg|\int_{\hat{\mathbb{X}}}(\hat{d}^2(\hat{x},\hat{0})\wedge N)\hat{\mu}_{T_{k}}({\rm d}\hat{x})-\int_{\hat{\mathbb{X}}}(\hat{d}^2(\hat{x},\hat{0})\wedge N)\hat{\mu}({\rm d}\hat{x})\bigg|\\
                &\ \ \ \ +\bigg|\int_{\{\hat{d}^2(\hat{x},\hat{0})\geq N\}}\hat{d}^2(\hat{x},\hat{0})\hat{\mu}_{T_{k}}({\rm d}\hat{x})\bigg|+\bigg|\int_{\{\hat{d}^2(\hat{x},\hat{0})\geq N\}}\hat{d}^2(\hat{x},\hat{0})\hat{\mu}({\rm d}\hat{x})\bigg|\\
                &\leq \bigg|\int_{\hat{\mathbb{X}}}(\hat{d}^2(\hat{x},\hat{0})\wedge N)\hat{\mu}_{T_{k}}({\rm d}\hat{x})-\int_{\hat{\mathbb{X}}}(\hat{d}^2(\hat{x},\hat{0})\wedge N)\hat{\mu}({\rm d}\hat{x})\bigg|+\frac{2C_4}{N}.
            \end{split}
        \end{equation*}

        Since $\hat{\mu}_{T_k}$ weakly converges to $\hat{\mu}$, we conclude that
        \begin{equation*}
            \limsup_{k\to\infty}\bigg|\int_{\hat{\mathbb{X}}}\hat{d}^2(\hat{x},\hat{0})\hat{\mu}_{T_{k}}({\rm d}\hat{x})-\int_{\hat{\mathbb{X}}}\hat{d}^2(\hat{x},\hat{0})\hat{\mu}({\rm d}\hat{x})\bigg|\leq \frac{2C_4}{N}.
        \end{equation*}
        Letting $N\to\infty$, we have
        \begin{equation*}
            \lim_{k\to\infty}\int_{\hat{\mathbb{X}}}\hat{d}^2(\hat{x},\hat{0})\hat{\mu}_{T_{k}}({\rm d}\hat{x})=\int_{\hat{\mathbb{X}}}\hat{d}^2(\hat{x},\hat{0})\hat{\mu}({\rm d}\hat{x}).
        \end{equation*}
        By \cite[Theorem 6.9]{Villani2009}, we know that $\lim_{k\to\infty}\mathcal{D}_2(\hat{\mu}_{T_{k}},\hat{\mu})=0$. Then Proposition \ref{prop:hat-P*-CDS} yields 
        \begin{equation}\label{eq:n1127}
            \hat{P}^*_t\hat{\mu}_{T_{k}}\to \hat{P}^*_t\hat{\mu}, \ \text{ as } \ k\to\infty, \ \text{ in } \ \mathcal{P}_2(\hat{\mathbb{X}}), \ \text{ for all } \ t\geq 0.
        \end{equation}
        On the other hand, for any $t\geq 0$,
        \begin{equation*}
            \begin{split}
                \hat{P}^*_t\hat{\mu}_T&=\frac{1}{2T}\int_{-T}^T \big(\delta_{(t+r) | \vec{\mathbf{\tau}}}\times P^{r | \vec{\mathbf{\tau}}, *}(t,0)\mu_r\big) {\rm d}r\\
                &=\frac{1}{2T}\int_{-T}^T \big(\delta_{(t+r) | \vec{\mathbf{\tau}}}\times P^{*}(t+r,r)\mu_r \big){\rm d}r\\
                &=\frac{1}{2T}\int_{-T}^T \big(\delta_{(t+r)|\vec{\mathbf{\tau}}}\times \mu_{t+r}\big) {\rm d}r\\
                &=\hat{\mu}_T+\frac{1}{2T}\bigg(\int_{T}^{T+t}\big(\delta_{r | \vec{\mathbf{\tau}}}\times \mu_r\big) {\rm d}r-\int_{-T}^{-T+t}\big(\delta_{r|\vec{\mathbf{\tau}}}\times \mu_r\big) {\rm d}r\bigg).
            \end{split}
        \end{equation*}
        Hence, for any $f\in C_b(\hat{\mathbb{X}})$,
        \begin{equation*}
            \bigg|\int_{\hat{\mathbb{X}}}f{\rm d}\hat{P}^*_t\hat{\mu}_{T_k}-\int_{\hat{\mathbb{X}}}f{\rm d}\hat{\mu}\bigg|\leq \bigg|\int_{\hat{\mathbb{X}}}f{\rm d}\hat{\mu}_{T_k}-\int_{\hat{\mathbb{X}}}f{\rm d}\hat{\mu}\bigg|+\frac{t|f|_{C_b(\hat{\mathbb{X}})}}{T_k}\to 0, \ \text{ as } \ k\to\infty.
        \end{equation*}
        Thus, for all $t\geq 0$, $\hat{P}^*_t\hat{\mu}_{T_k}$ weakly converge to $\hat{\mu}$ as $k\to \infty$. This together with \eqref{eq:n1127} gives $\hat{P}^*_t\hat{\mu}=\hat{\mu}$ for all $t\geq 0$, which means $\hat{\mu}$ is an invariant measure. 
    \end{proof}

In the rest of this section, we show that under the quasi-periodic condition, equation \eqref{McKean-Vlasov SDE 2} has an asymptotic quasi-periodic  measure (Theorem \ref{Thm existence quasi-periodic measure}).
Before the proof of Theorem \ref{Thm existence quasi-periodic measure}, we first introduce the following notations. Denote by $\mathcal{T}_{\tau_1,\cdots,\tau_n}$ the collection of $(\mu,\tilde{\mu})$ where $\mu$ is a measure-valued quasi-periodic function with periods $\tau_1,\cdots,\tau_n$ and $\tilde{\mu}$ is its quasi-periodic representation. Set
	\begin{equation}\label{eq:def-Q_2}
		\mathcal{Q}_2:=\{\mu\in \mathcal{M}_2(\mathbb{R}): \text{ there exists } \tilde{\mu} \text{ such that } (\mu,\tilde{\mu})\in \mathcal{T}_{\tau_1,\cdots,\tau_n} \text{ and } \tilde{\mu}\in C(\mathbb{R}^n; \mathcal{P}_p(\mathbb{R}^d))\}.
	\end{equation}

	\begin{remark}\label{remark 0517}
		Note that $\tilde{\mu}\in C(\mathbb{R}^n; \mathcal{P}_p(\mathbb{R}^d))$ is the quasi-periodic representation of $\mu$, so $\tilde{\mu}$ restricts on $\mathbb{T}^n$, denoted by $\tilde{\mu}|_{\mathbb{T}^n}$, is in $C(\mathbb{T}^n; \mathcal{P}_p(\mathbb{R}^d))$. Since the torus $\mathbb{T}^n$ is compact, we know that $\tilde{\mu}$ is uniformly continuous on $\mathbb{T}^n$. Thus, there exists an increasing function $w: [0,\infty)\to [0,\infty)$ satisfying $w(0)=0$ and continuous at $0$ such that
		\begin{equation*}
			\mathcal{W}_2(\tilde{\mu}_{\vec{\mathbf{s}}_1}, \tilde{\mu}_{\vec{\mathbf{s}}_2})\leq w(d_0(\vec{\mathbf{s}}_1,\vec{\mathbf{s}}_2)), \ \text{ for all } \ \vec{\mathbf{s}}_1,\vec{\mathbf{s}}_2\in \mathbb{T}^n.
		\end{equation*}
	\end{remark}

	For any fixed $\mu\in \mathcal{Q}_2$, we rewrite \eqref{Classical SDE} as the following SDE:
	\begin{equation}\label{New SDE}
		\begin{cases}
			\mathrm{d}X_t=b^{\mu}(t,X_t)\mathrm{d}t+\sigma^{\mu}(t,X_t)\mathrm{d}W_t, \ t\geq s,\\
			X_s=\xi,
		\end{cases}
	\end{equation}
	where $f^{\mu}(t,x)=f(t,x,\mu_t), f=b,\sigma$. Under Assumption \ref{Quasi-periodic condition}, we consider the following reparameterized SDE to SDE \eqref{New SDE} which was first introduced in \cite{Feng-Qu-Zhao2021}:
	\begin{equation}\label{Equation K_r_1,r_2}
			K^{\tilde{\mu},\vec{\mathbf{s}}}(t,s,\xi)=\xi+\int_{s}^{t}\tilde{b}^{\tilde{\mu},\vec{\mathbf{s}}}(u,K^{\tilde{\mu},\vec{\mathbf{s}}}(u,s,\xi))\mathrm{d}u+\int_{s}^{t}\tilde{\sigma}^{\tilde{\mu},\vec{\mathbf{s}}}(u,K^{\tilde{\mu},\vec{\mathbf{s}}}(u,s,\xi))\mathrm{d}W_u,
	\end{equation}
	where $\vec{\mathbf{s}}=(s_1,\cdots,s_n)\in\mathbb{T}^n$, and $\tilde{f}^{\tilde{\mu},\vec{\mathbf{s}}}(t,x)=\tilde{f}(T_t(\vec{\mathbf{s}}),x,\tilde{\mu}_{T_t(\vec{\mathbf{s}})}), \  f=b,\sigma$. By Remark \ref{lemma 0427}, we know that if Assumptions \ref{A3} and \ref{Quasi-periodic condition} hold, then SDE \eqref{Equation K_r_1,r_2} has a unique solution $K^{\tilde{\mu},\vec{\mathbf{s}}}(t,s,\xi)$ with any starting time $s\in \mathbb{R}$ and initial condition $\xi\in L^2(\mathcal{F}_s)$. 
	
	Denote by $P^{\tilde{\mu},\vec{\mathbf{s}}}(t,s,x,\cdot)$ the transition function of $K^{\tilde{\mu},\vec{\mathbf{s}}}(t,s,x)$, $P^{\tilde{\mu},\vec{\mathbf{s}}}(t,s)$ the kernel acting as a linear operator on functions $f: \mathbb{R}^d\to \mathbb{R}$ and $P^{\tilde{\mu},\vec{\mathbf{s}},*}(t,s)$ the dual operator of $P^{\tilde{\mu},\vec{\mathbf{s}}}(t,s)$ acting on measures in $\mathcal{P}(\mathbb{R}^d)$.

	\begin{theorem}\label{Thm 0501}
		Assume Assumptions \ref{A3}, \ref{A4}, \ref{Quasi-periodic condition} and \ref{One-sided Lip of tilde b} hold. For any $\mu\in \mathcal{Q}_2$, SDE \eqref{New SDE} has a unique quasi-periodic measure $\rho\in \mathcal{Q}_{2}$, whose quasi-periodic representation is denoted by $\tilde{\rho}: \mathbb{R}^n\to \mathcal{P}(\mathbb{R}^d)$ satisfying
		\begin{equation}\label{0502}
			P^{\tilde{\mu},\vec{\mathbf{s}},*}(t,s)\tilde{\rho}_{T_s(\vec{\mathbf{s}})}=\tilde{\rho}_{T_t(\vec{\mathbf{s}})}, \ \text{ for all $t\geq s$ and $\vec{\mathbf{s}}\in \mathbb{T}^n$}. 
		\end{equation}
	\end{theorem}
	\begin{proof}
		For any fixed $\mu\in \mathcal{Q}_2$, it follows from Remark \ref{lemma 0427} that $(b^{\mu}, \sigma^{\mu})$ and $(\tilde{b}^{\tilde{\mu},\vec{\mathbf{s}}}, \tilde{\sigma}^{\tilde{\mu},\vec{\mathbf{s}}})$ satisfy all assumptions in Theorem 6.4 in \cite{Feng-Qu-Zhao2023} and hence SDE \eqref{New SDE} has a unique quasi-periodic measure $\rho\in \mathcal{M}_2(\mathbb{R})$ with periods $\tau_1,\cdots,\tau_n$ and $\{\tilde{\rho}^{\vec{\mathbf{s}}}_t:=\tilde{\rho}_{T_t(\vec{\mathbf{s}})}; t\in \mathbb{R}\}$ is the unique entrance measure of $P^{\mu,\vec{\mathbf{s}},*}$ (i.e., \eqref{0502} holds). It remains to show that for any $\tilde{\rho}|_{\mathbb{T}^n}\in C(\mathbb{T}^n; \mathcal{P}_p(\mathbb{R}^d))$.

		By Remark \ref{lemma 0427} and \cite[Theorem 4.8]{Feng-Qu-Zhao2023}, we know that there exist $C,\lambda>0$ such that
		\begin{equation*}
			{\rm dist}_2\big(\tilde{\rho}^{\vec{\mathbf{s}}}_t, \mathscr{L}(K^{\tilde{\mu},\vec{\mathbf{s}}}(t,s,0))\big)\leq Ce^{-\lambda(t-s)}, \ \text{ for all } \ \vec{\mathbf{s}}\in \mathbb{T}^n, \ t\geq s,
		\end{equation*}
    where
    \begin{equation*}
			{\rm dist}_2(\nu_1,\nu_2):=\sup_{|f(x)|\leq |x|^2}\bigg|\int_{\mathbb{R}^d}f{\rm d}\nu_1-\int_{\mathbb{R}^d}f{\rm d}\nu_2\bigg|=\int_{\mathbb{R}^d}|x|^2{\rm d}|\nu_1-\nu_2|(x).
		\end{equation*}
    By \cite[Theorem 6.15]{Villani2009}, we conclude that for all $\vec{\mathbf{s}}\in \mathbb{T}^n, \ t\geq s$,
    \begin{equation*}
			\mathcal{W}_2^2\big(\tilde{\rho}^{\vec{\mathbf{s}}}_t, \mathscr{L}(K^{\tilde{\mu},\vec{\mathbf{s}}}(t,s,0))\big)\leq 2{\rm dist}_2\big(\tilde{\rho}^{\vec{\mathbf{s}}}_t, \mathscr{L}(K^{\tilde{\mu},\vec{\mathbf{s}}}(t,s,0))\big)\leq 2Ce^{-\lambda(t-s)}.
		\end{equation*}
Hence, for any $\vec{\mathbf{t}},\vec{\mathbf{s}}\in\mathbb{T}^n$, we know that for any $T>0$,
\begin{equation}\label{eq:1204-1}
	\begin{split}
		&\ \ \ \ \mathcal{W}_2^2\big(\tilde{\rho}_{\vec{\mathbf{t}}}, \tilde{\rho}_{\vec{\mathbf{s}}}\big)=\mathcal{W}_2^2\big(\tilde{\rho}^{\vec{\mathbf{t}}}_0, \tilde{\rho}^{\vec{\mathbf{s}}}_0\big)\\
&\leq 3\mathcal{W}_2^2\big(\tilde{\rho}^{\vec{\mathbf{t}}}_0, \mathscr{L}(K^{\tilde{\mu},\vec{\mathbf{t}}}(0,-T,0))\big)+3\mathcal{W}_2^2\big(\tilde{\rho}^{\vec{\mathbf{s}}}_0, \mathscr{L}(K^{\tilde{\mu},\vec{\mathbf{s}}}(0,-T,0))\big)\\
&\ \ \ \ +3\mathcal{W}_2^2\big(\mathscr{L}(K^{\tilde{\mu},\vec{\mathbf{t}}}(0,-T,0)), \mathscr{L}(K^{\tilde{\mu},\vec{\mathbf{s}}}(0,-T,0))\big)\\
&\leq 12Ce^{-\lambda T}+3\mathbf{E}\big[\big|K^{\tilde{\mu},\vec{\mathbf{t}}}(0,-T,0)-K^{\tilde{\mu},\vec{\mathbf{s}}}(0,-T,0)\big|^2\big].
	\end{split}
\end{equation}
According to Remark \ref{remark 0517}, there exists an increasing function $w_1: \mathbb{R}^+\to\mathbb{R}^+$ satisfying $w_1(0)=0$ and continuous at $0$ such that $\mathcal{W}_2(\tilde{\mu}_{\vec{\mathbf{t}}}, \tilde{\mu}_{\vec{\mathbf{s}}})\leq w_1(d_0(\vec{\mathbf{t}},\vec{\mathbf{s}}))$ for all $\vec{\mathbf{t}},\vec{\mathbf{s}}\in \mathbb{T}^n$. Then by Assumptions \ref{A3}, \ref{Quasi-periodic condition} and \ref{One-sided Lip of tilde b}, we conclude that for all $x,y\in\mathbb{R}^d, \ \vec{\mathbf{t}},\vec{\mathbf{s}}\in\mathbb{T}^n, \ t\in\mathbb{R}$,
\begin{equation}\label{eq:1204-2}
	\begin{split}
		&\ \ \ \ 2\big\langle x-y,  b^{\tilde{\mu},\vec{\mathbf{t}}}(t,x)-b^{\tilde{\mu},\vec{\mathbf{s}}}(t,y)\big\rangle+\|\sigma^{\tilde{\mu},\vec{\mathbf{t}}}(t,x)-\sigma^{\tilde{\mu},\vec{\mathbf{s}}}(t,y)\|_2^2\\
    &\leq 2L|x-y|^2+2L|x-y|\mathcal{W}_2\big(\tilde{\mu}_{T_t(\vec{\mathbf{t}})}, \tilde{\mu}_{T_t(\vec{\mathbf{s}})}\big)+ 2|x-y|\big|b\big(T_t(\vec{\mathbf{t}}),y,\tilde{\mu}_{T_t(\vec{\mathbf{s}})}\big)-b\big(T_t(\vec{\mathbf{s}}),y,\tilde{\mu}_{T_t(\vec{\mathbf{s}})}\big)\big|\\
    &\ \ \ \ +2L^2\big(|x-y|+\mathcal{W}_2\big(\tilde{\mu}_{T_t(\vec{\mathbf{t}})}, \tilde{\mu}_{T_t(\vec{\mathbf{s}})}\big)\big)^2+2\big\|\sigma\big(T_t(\vec{\mathbf{t}}),y,\tilde{\mu}_{T_t(\vec{\mathbf{s}})}\big)-\sigma\big(T_t(\vec{\mathbf{s}}),y,\tilde{\mu}_{T_t(\vec{\mathbf{s}})}\big)\big\|_2^2\\
   &\leq (4L^2+3L+1)|x-y|^2+(4L^2+L)w_1^2\big(d_0(T_t(\vec{\mathbf{t}}),T_t(\vec{\mathbf{s}}))\big)+2h\big(T_t(\vec{\mathbf{t}}),T_t(\vec{\mathbf{s}})\big).
	\end{split}
\end{equation}
Note that by letting $(\vec{\mathbf{t}}_1,\vec{\mathbf{t}}_2)=(T_t(\vec{\mathbf{t}}),T_t(\vec{\mathbf{s}}))$ and $(\vec{\mathbf{s}}_1,\vec{\mathbf{s}}_2)=(T_t(\vec{\mathbf{s}}),T_t(\vec{\mathbf{s}}))$ in \eqref{eq:new1204}, we have
\begin{equation}\label{eq:1204-3}
	h\big(T_t(\vec{\mathbf{t}}),T_t(\vec{\mathbf{s}})\big)\leq w\big(d_0(T_t(\vec{\mathbf{t}}),T_t(\vec{\mathbf{s}}))\big),
\end{equation}
for another increasing function $w: \mathbb{R}^+\to\mathbb{R}^+$ satisfying $w(0)=0$ and continuous at $0$. 
Then applying It{\^ o}'s formula to $\big|K^{\tilde{\mu},\vec{\mathbf{t}}}(t,-T,0)-K^{\tilde{\mu},\vec{\mathbf{s}}}(t,-T,0)\big|^2$ on $[0,T]$ together with \eqref{eq:1204-1}-\eqref{eq:1204-3} and the fact that $d_0(T_t(\vec{\mathbf{t}}),T_t(\vec{\mathbf{s}}))=d_0(\vec{\mathbf{t}},\vec{\mathbf{s}})$ for all $t\in\mathbb{R}$, we conclude that for all $t\in [-T,0]$,
\begin{equation*}
	\begin{split}
		&\ \ \ \ \mathbf{E}\big[\big|K^{\tilde{\mu},\vec{\mathbf{t}}}(t,-T,0)-K^{\tilde{\mu},\vec{\mathbf{s}}}(t,-T,0)\big|^2\big]\\
    &=2\mathbf{E}\int_{-T}^t\big\langle K^{\tilde{\mu},\vec{\mathbf{t}}}(r,-T,0)-K^{\tilde{\mu},\vec{\mathbf{s}}}(r,-T,0), b^{\tilde{\mu},\vec{\mathbf{t}}}\big(r,K^{\tilde{\mu},\vec{\mathbf{t}}}(r,-T,0)-b^{\tilde{\mu},\vec{\mathbf{s}}}\big(r,K^{\tilde{\mu},\vec{\mathbf{s}}}(r,-T,0)\big)\big\rangle{\rm d}r\\
    &\ \ \ \ +\mathbf{E}\int_{-T}^t\big\|\sigma^{\tilde{\mu},\vec{\mathbf{t}}}\big(r,K^{\tilde{\mu},\vec{\mathbf{t}}}(r,-T,0)-\sigma^{\tilde{\mu},\vec{\mathbf{s}}}\big(r,K^{\tilde{\mu},\vec{\mathbf{s}}}(r,-T,0)\big)\big\|_2^2{\rm d}r\\
    &\leq Tw^*\big(d_0(\vec{\mathbf{t}},\vec{\mathbf{s}})\big)+(4L^2+3L+1)\int_{-T}^t\mathbf{E}\big[\big|K^{\tilde{\mu},\vec{\mathbf{t}}}(r,-T,0)-K^{\tilde{\mu},\vec{\mathbf{s}}}(r,-T,0)\big|^2\big]{\rm d}r,
	\end{split}
\end{equation*}
where $w^*=(4L^2+L)w_1^2+2w$ is also an increasing function that vanishes and is continuous at $0$. By Gronwall inequality, we have
\begin{equation}\label{eq:1204-4}
	\mathbf{E}\big[\big|K^{\tilde{\mu},\vec{\mathbf{t}}}(0,-T,0)-K^{\tilde{\mu},\vec{\mathbf{s}}}(0,-T,0)\big|^2\big]\leq Te^{(4L^2+3L+1)T}w^*\big(d_0(\vec{\mathbf{t}},\vec{\mathbf{s}})\big).
\end{equation}
Then it follows from \eqref{eq:1204-1} and \eqref{eq:1204-4} that for all $T>0$,
\begin{equation*}
	\mathcal{W}_2^2\big(\tilde{\rho}_{\vec{\mathbf{t}}}, \tilde{\rho}_{\vec{\mathbf{s}}}\big)\leq 12Ce^{-\lambda T}+3Te^{(4L^2+3L+1)T}w^*\big(d_0(\vec{\mathbf{t}},\vec{\mathbf{s}})\big)
\end{equation*}
Hence, 
\begin{equation*}
	\limsup_{\vec{\mathbf{t}}\to \vec{\mathbf{s}}}\mathcal{W}_2^2\big(\tilde{\rho}_{\vec{\mathbf{t}}}, \tilde{\rho}_{\vec{\mathbf{s}}}\big)\leq 12Ce^{-\lambda T}, \ \text{ for all } \ T>0.
\end{equation*}
By letting $T\to \infty$, we conclude that $\tilde{\rho}|_{\mathbb{T}^n}\in C(\mathbb{T}^n;\mathcal{P}_2(\mathbb{R}^d))$. The proof is complete.
	\end{proof}

	According to Theorem \ref{Thm 0501}, for any $\mu\in \mathcal{Q}_2$, there is a unique quasi-periodic measure $\rho^{\mu}\in \mathcal{Q}_{2}$ to SDE \eqref{New SDE}. Similarly, we define a map
	$\Theta: \mathcal{Q}_2\to \mathcal{Q}_2$ by $\Theta (\mu):=\rho^{\mu}$. Then we give a proof of Theorem \ref{Thm existence quasi-periodic measure} by applying the fixed point theorem.

	\begin{proof}[Proof of Theorem \ref{Thm existence quasi-periodic measure}]
		Note that $\mu$ is an asymptotic quasi-periodic measure of \eqref{McKean-Vlasov SDE 2} if and only if $\mu$ is asymptotic quasi-periodic and it is a fixed point of $\Theta$, i.e., $\Theta(\mu)=\mu$. 

		Set 
		\begin{equation}\label{eq:my1130}
			\widetilde{\mathcal{Q}}:=\{\Theta(\mu): \mu\in \mathcal{Q}_2\cap \mathcal{M}_{2,\vartheta}\}.
		\end{equation}
		According to Theorem \ref{Thm 0501}, following a similar procedure in the proofs of Lemma \ref{Lemma of invariant set} and Lemma \ref{Lemma of compact set}, we conclude that $\Theta(\widetilde{\mathcal{Q}})\subset \widetilde{\mathcal{Q}}\subset \mathcal{Q}_{\infty}$ and $\widetilde{\mathcal{Q}}$ is pre-compact under the metric $d_2$ given as in \eqref{eq:distance-d_2}. Let $\overline{{\rm co}}(\widetilde{\mathcal{Q}})$ be the closed convex hull of $\widetilde{\mathcal{Q}}$ under $d_2$. Then all elements in $\overline{{\rm co}}(\widetilde{\mathcal{Q}})$ are asymptotic quasi-periodic. Moreover, repeating the proof of Theorem \ref{Theorem existence of entrance measure}, we can prove that there exists a $\mu \in \overline{{\rm co}}(\widetilde{\mathcal{Q}})$ such that $\Theta(\mu)=\mu$. 
    In the proof of Theorem \ref{thm:existence-of-invariant-measure}, we have already proved that $\mu\in\mathcal{M}_p(\mathbb{R})$ for all $p\geq 2$. Thus, the proof is complete.
	\end{proof}

    \begin{remark}
        Assume that Assumptions \ref{A3}, \ref{A4}, and \ref{Quasi-periodic condition} hold. 
        \begin{itemize}
            \item [(i)] In the case that $n=1$, i.e., when the coefficients $b$ and $\sigma$ are time-periodic, it follows that $\overline{{\rm co}}(\widetilde{\mathcal{Q}}) \subset \mathcal{Q}_2$, where $\widetilde{\mathcal{Q}}$ is defined as in \eqref{eq:my1130}. Consequently, equation \eqref{McKean-Vlasov SDE 2} admits a periodic measure.

            \item [(ii)] In the case that $n\geq 2$, we have not been able to prove that $\overline{{\rm co}}(\widetilde{\mathcal{Q}}) \subset \mathcal{Q}_2$ as it is not clear that the limit of a sequence of continuous multi-parameter measure-valued functions is still a continuous multi-parameter measure-valued function. Thus, we can only conclude that the fixed point $\mu$ of $\Theta$ is an asymptotic quasi-periodic measure rather than a quasi-periodic measure. But this is adequate to obtain an invariant measure for the lifted system.
        \end{itemize}
    \end{remark}

	\begin{example}\label{Example 1}
		We consider the 1-dimensional generalized Curie-Weiss model with a time-inhomogeneous forcing term arising in the study of ferromagnetism as a propagation of chaos model for ferromagnetic interacting particles
		\begin{equation}\label{0628}
			\mathrm{d}X_t=\beta\biggl(X_t-X_t^3+f(t)+\int_{\mathbb{R}}\nabla_xW(X_t,y)\mathscr{L}(X_t)(\mathrm{d}y)\biggr)\mathrm{d}t+\sigma \mathrm{d}B_t,
		\end{equation}
		where $\beta>0$ is a constant, known as inverse temperature, $\sigma\neq 0$ is a constant and $\{B_t\}_{t\in \mathbb{R}}$ is a standard 1-dimensional two-sided Brownian motion. Assume $W: \mathbb{R}^2\to \mathbb{R}$ satisfying $|\nabla_xW(x,y)|\leq C(1+|x|^{\kappa}+|y|)$ for some $C>0$ and $0\leq \kappa<3$, $f: \mathbb{R}\to \mathbb{R}$ is continuous and bounded. The growth condition on $\nabla_x W$ controls particle interactions to guarantee well-posedness of the stochastic system (see e.g. Liu-Wu-Zhang \cite{Liu-Wu-Zhang2021}), and the boundedness of the forcing term $f(t)$ reflect realistic, time-varying external influences. Here, we consider the continuous forcing case only. By Theorem \ref{Theorem existence of entrance measure}, \eqref{0628} has an entrance measure. Moreover, according to the results obtained in this section, \eqref{0628} has
		\begin{itemize}
			\item an invariant measure if $f$ is a constant;
			\item a periodic measure if $f$ is a periodic continuous function;
			\item an asymptotic quasi-periodic measure if $f$ is a quasi-periodic function.
		\end{itemize}
		In the last two cases, \eqref{0628} also has an invariant measure upon lifting.
	\end{example}

	\begin{remark}
		Liu-Wu-Zhang \cite{Liu-Wu-Zhang2021}, Guillin-Liu-Wu-Zhang \cite{Guillin-Liu-Wu-Zhang2022} considered the case when $f\equiv 0, W(x,y)=-kxy$. Under the condition that $|k|\sqrt{\pi\beta}e^{\beta/4}\leq 1$ ($\sigma$ is normalized to be $\sqrt{2}$), they proved that \eqref{0628} has a unique invariant measure. Our results show that \eqref{0628} in this case has an invariant measure without requiring any condition on $\beta$ and $k$. They also considered the case where $f\equiv 0, W(x,y)=k(x-y)^2$. Under the following condition ($\sigma=\sqrt{2}$)
		\begin{equation}\label{eq:parameter-uniqueness}
			\begin{cases}
				2|k|\sqrt{\pi\beta}e^{(1-2k)^2\beta/4}\leq 1, & \text{ if } k\leq \frac{1}{2},\\
				2|k|\sqrt{\pi\beta}\leq 1, & \text{ if } k>\frac{1}{2},
			\end{cases}
		\end{equation}
		They showed that \eqref{0628} has a unique invariant measure. Our result also includes this case where no conditions on $\beta, k$ are needed and the existence of invariant measures has been obtained. Needless to say, uniqueness cannot be validated without imposing conditions on the amplitudes of $\beta$ and $\kappa$. It is also worth mentioning that the condition posed by Carrillo-McCann-Villani \cite{Carrillo-McCann-Villani2003} requires potential to be convex, such a condition has already been improved in literature e.g. Guillin-Liu-Wu-Zhang \cite{Guillin-Liu-Wu-Zhang2022}, and our result does not need such a kind of condition either.  Our contribution is mainly on the study of nonstationary Curie-Weiss SDE \eqref{0628}, we obtain the existence of entrance measure and its dynamics without imposing any condition on $\beta, \sigma$ and the growth order $\kappa$ of $\nabla_xW$ in $x$ (except $\beta>0, \sigma\neq 0, \kappa<3$).
	\end{remark}

	\section{Non-uniqueness of entrance measures: Proof of Theorem \ref{Theorem nonuniqueness}}

	In this section, we will mainly study non-uniqueness of entrance (periodic, asymptotic quasi-periodic, invariant) measures for some McKean-Vlasov equations satisfying all Assumptions in Theorem \ref{Theorem existence of entrance measure}. The nonuniqueness of invariant measures 
	for a class of stationary McKean-Vlasov equations with small noise was observed by Dawson \cite{Dawson1983}. Inspired by Zhang \cite{Zhang2023}, for any $a\in \mathbb{R}^d$ and $\nu\in \mathcal{P}(\mathbb{R}^d)$, denote by $\nu^a$ the shift probability of $\nu$ by $a$:
	\begin{equation}\label{0913-1}
		\int_{\mathbb{R}^d}f(x)\nu^a(\mathrm{d}x):=\int_{\mathbb{R}^d}f(x-a)\nu(\mathrm{d}x), \ \text{ for any bounded measurable function } f.  
	\end{equation}
	Then for any $a,b\in \mathbb{R}^d$, we can obtain that $(\nu^a)^b=\nu^{a+b}$. Fix $\mu,\nu\in \mathcal{P}(\mathbb{R}^d)$,
	for any $\gamma\in \mathscr{C}(\mu,\nu)\subset \mathcal{P}(\mathbb{R}^{2d})$, we define $S^{a,b}\gamma=\gamma^{a,b}$ as in \eqref{0913-1}. It is easy to show that $S^{a,b}\mathscr{C}(\mu,\nu)=\mathscr{C}(\mu^a,\nu^b)$ and $S^{-a,-b}\mathscr{C}(\mu^a,\nu^b)=\mathscr{C}(\mu,\nu)$. It is then easy to prove that $\mathcal{W}_p(\mu^a,\nu^b)=\mathcal{W}_p(\mu,\nu)$ for any $\mu,\nu\in \mathcal{P}_p(\mathbb{R}^d)$.

	Fix $a\in \mathbb{R}^d$ and $\theta_a>0$, set
	\begin{equation*}
		\mathcal{M}^a_{1,\theta_a}:=\bigg\{\mu:\mathbb{R}\to \mathcal{P}_1(\mathbb{R}^d)\big| \sup_{t\in \mathbb{R}}\|\mu_t^a\|_1\leq \theta_a\bigg\},
	\end{equation*}
	where $\|\mu_t^a\|_1=\int_{\mathbb{R}^d}|x|\mu_t^a(\mathrm{d}x)=\int_{\mathbb{R}^d}|x-a|\mu_t(\mathrm{d}x)=\mathcal{W}_1(\mu_t,\delta_a)$. In fact, $\mathcal{M}^a_{1,\theta_a}$ is a closed ball centered at the constant measure function $\nu_t\equiv\delta_a$ with radius $\theta_a$ in $C(\mathbb{R};\mathcal{P}_1(\mathbb{R}^d))$.
	
	Now set
	\begin{equation}\label{0815-1}
		f^a(t,x,\nu):=f(t,x+a,\nu^{-a}), \ f=b,\sigma, \ \text{ for all } t\in \mathbb{R}, a,x\in \mathbb{R}^d, \nu\in \mathcal{P}(\mathbb{R}^d).
	\end{equation}
	It is easy to see that for any random variable $X_t$,
	\begin{equation*}
		f^a\big(t,X_t-a,\mathscr{L}(X_t-a)\big)=f\big(t,X_t,(\mathscr{L}(X_t-a))^{-a}\big)=f\big(t,X_t,\mathscr{L}(X_t)\big).
	\end{equation*}
	Thus it follows that if $X_{\cdot}$ solves equation \eqref{McKean-Vlasov SDE 2}, then $Y_t:=X_t-a$ solves the following McKean-Vlasov equation:
	\begin{equation*}
		\mathrm{d}Y_t=b^a(t,Y_t,\mathscr{L}(Y_t))\mathrm{d}t+\sigma^a(t,Y_t,\mathscr{L}(Y_t))\mathrm{d}W_t, \ t\geq s.
	\end{equation*}
	
Now, we give a proof of Theorem \ref{Theorem nonuniqueness}.
	\begin{proof}[Proof of Theorem \ref{Theorem nonuniqueness}]
		For any $\mu\in \mathcal{M}^a_{2,\vartheta^a}, s\in \mathbb{R}$ and $y\in \mathbb{R}^d$, denote by $Y_t^{\mu^a,s,y}$ be the unique solution of the following SDE
		\begin{equation}\label{SDE 0913}
			\begin{cases}
				&\mathrm{d}Y_t^{\mu^a}=b^a(t,Y_t,\mu^a_t)\mathrm{d}t+\sigma^a(t,Y_t,\mu^a_t)\mathrm{d}W_t, \ t\geq s,\\
				& Y_s^{\mu^a}=y.
			\end{cases}
		\end{equation}
		Then by Theorem \ref{Theorem geometric converge to periodic measure}, SDE \eqref{SDE 0913} has a unique entrance measure $\rho$ in $\mathcal{M}_{2,\vartheta^a}$ and there exists a decreasing sequence $\{t_n\}_{n\geq 0}$ with $\lim_{n\to \infty}t_n=-\infty$ such that for all $y\in \mathbb{R}^d$, $t\in \mathbb{R}$,
		\begin{equation*}
			\lim_{n\to \infty}\|\mathscr{L}(Y_t^{\mu^a,t_n,x})-\rho_t\|_{TV}=0.
		\end{equation*}
		Then for any $\mu\in \mathcal{M}^a_{2,\vartheta^a}, s\in \mathbb{R}$ and $x\in \mathbb{R}^d$, $X_t^{\mu,s,x}:=Y_t^{\mu^a,s,x-a}$ solves SDE \eqref{Classical SDE} and hence equation \eqref{Classical SDE} has a unique entrance measure, denoted by $\Psi(\mu):=\rho^{-a}$, in $\mathcal{M}^a_{2,\vartheta^a}$ and for all $t\in \mathbb{R}$,
		\begin{equation}\label{0814-1}
			\lim_{n\to \infty}\mathbf{E}\big[|X_t^{\mu,t_n,a}-a|\big]=\|\Psi(\mu)_t^a\|_1.
		\end{equation}
		Now fix $\mu\in \mathcal{M}^a_{2,\vartheta^a}\cap\mathcal{M}^a_{1,\theta_a}$, we will show that $\Psi(\mu)\in \mathcal{M}^a_{2,\vartheta^a}\cap\mathcal{M}^a_{1,\theta_a}$. According to the discussions above, we only need to prove that $\Psi(\mu)\in \mathcal{M}^a_{1,\theta_a}$. To see this, for any $t\in \mathbb{R}$ and $t_n\leq t$, applying It{\^o}'s formula to $\big|Y_t^{\mu^{a},t_n,0}\big|^2=\big|X_t^{\mu,t_n,a}-a\big|^2$, we have 
		\begin{equation*}
			\begin{split}
				\big|Y_t^{\mu^{a},t_n,0}\big|^2&=\int_{t_n}^t\Bigl(2\big\langle Y_s^{\mu^{a},t_n,0},b^a(s,Y_s^{\mu^{a},t_n,0},\mu^a_s)\big\rangle+\|\sigma^a(s,Y_s^{\mu^{a},t_n,0},\mu^a_s)\|_2^2\Bigr)\mathrm{d}s\\
				&\ \ \ \ +2\int_{t_n}^t\big\langle Y_s^{\mu^{a},t_n,0},\sigma^a(s,Y_s^{\mu^{a},t_n,0},\mu^a_s)\mathrm{d}W_s\big\rangle\\
				&\leq -\int_{t_n}^tg_a\bigl(\big|Y_s^{\mu^{a},t_n,0}\big|, \theta_a\bigr)\mathrm{d}s+2\int_{t_n}^t\big\langle Y_s^{\mu^{a},t_n,0},\sigma^a(s,Y_s^{\mu^{a},t_n,0},\mu^a_s)\mathrm{d}W_s\big\rangle.
			\end{split}
		\end{equation*}
		Taking expectation on both sides and dividing both sides by $t-t_n$, the convexity of $g_a(\cdot,\theta_a)$ concludes that
		\begin{equation*}
			g_a\Bigl(\frac{1}{t-t_n}\int_{t_n}^t\mathbf{E}\bigl[\big|Y_s^{\mu^{a},t_n,0}\big|\bigr]\mathrm{d}s, \theta_a\Bigr)\leq \frac{1}{t-t_n}\int_{t_n}^t\mathbf{E}\bigl[g_a\bigl(\big|Y_s^{\mu^{a},t_n,0}\big|, \theta_a\bigr)\bigr]\mathrm{d}s\leq -\frac{\mathbf{E}\bigl[\big|Y_t^{\mu^{a},t_n,0}\big|^2\bigr]}{t-t_n}\leq 0.
		\end{equation*}
		It follows from \eqref{0814-1} and the continuity of $g_a(\cdot,\theta_a)$ that
		\begin{equation*}
			g_a(\|\Psi(\mu)^a_t\|_1, \theta_a)=\lim_{n\to \infty}g_a\Bigl(\frac{1}{t-t_n}\int_{t_n}^t\mathbf{E}\bigl[\big|Y_s^{\mu^{a},t_n,0}\big|\bigr]\mathrm{d}s, \theta_a\Bigr)\leq 0.
		\end{equation*}
		Hence, by \eqref{0814-2}, we know that $\|\Psi(\mu)_t^a\|_1\leq \theta_a$ for all $t\in \mathbb{R}$, i.e., $\Psi(\mu)\in \mathcal{M}^a_{1,\theta_a}$.

		Now consider the single-valued map $\Psi: \mathcal{M}^a_{2,\vartheta^a}\cap\mathcal{M}^a_{1,\theta_a}\to \mathcal{M}^a_{2,\vartheta^a}\cap\mathcal{M}^a_{1,\theta_a}$ and set
		\begin{equation*}
			\widetilde{\mathcal{M}}:=\{\Psi(\mu): \mu\in \mathcal{M}^a_{2,\vartheta^a}\cap\mathcal{M}^a_{1,\theta_a}\}.
		\end{equation*}
		Similar to the proof of Lemmas \ref{Lemma of invariant set}, \ref{Lemma of compact set} and Theorem \ref{Theorem existence of entrance measure}, we know that $\overline{{\rm co}}(\widetilde{\mathcal{M}})$ the closed convex hull of $\widetilde{\mathcal{M}}$ in $C(\mathbb{R};\mathcal{P}_2(\mathbb{R}^d))$ is nonempty and compact in $C(\mathbb{R}; \mathcal{MS}_2)$ and
		\begin{itemize}
			\item $\overline{{\rm co}}(\widetilde{\mathcal{M}})\subset \mathcal{M}^a_{2,\vartheta^a}\cap\mathcal{M}^a_{1,\theta_a}$;
			\item $\Psi\bigl(\overline{{\rm co}}(\widetilde{\mathcal{M}})\bigr)\subset \overline{{\rm co}}(\widetilde{\mathcal{M}})$;
			\item $\Big\{\big(\mu,\Psi(\mu)\big): \mu\in \overline{{\rm co}}(\widetilde{\mathcal{M}})\Big\}$ is closed in $\overline{{\rm co}}(\widetilde{\mathcal{M}})\times \overline{{\rm co}}(\widetilde{\mathcal{M}})$.
		\end{itemize}
		Then Kakutani's fixed point theorem (Theorem \ref{Lemma Kakutani's fixed point theorem}) yields $\Psi$ has a fixed point in $\mathcal{M}^a_{2,\vartheta^a}\cap\mathcal{M}^a_{1,\theta_a}$, which is an entrance measure of equation \eqref{McKean-Vlasov SDE 2}.

		If there exists two different $a_1,a_2\in \mathbb{R}^d$ and Assumptions \ref{A3}, \ref{A4} hold for both $(b^{a_1},\sigma^{a_1})$ and $(b^{a_2},\sigma^{a_2})$ respectively, then by what we have proved above, \eqref{McKean-Vlasov SDE 2} has an entrance measure in $\mathcal{M}^{a_1}_{2,\vartheta^{a_1}}\cap\mathcal{M}^{a_1}_{1,\theta_{a_1}}$ and an entrance measure in $\mathcal{M}^{a_2}_{2,\vartheta^{a_2}}\cap\mathcal{M}^{a_2}_{1,\theta_{a_2}}$. To prove \eqref{McKean-Vlasov SDE 2} has at least two distinct entrance measures, we only need to show that $\mathcal{M}^{a_1}_{1,\theta_{a_1}}\cap\mathcal{M}^{a_2}_{1,\theta_{a_2}}=\emptyset$ when $\theta_{a_1}+\theta_{a_2}<|a_1-a_2|$. Otherwise, there exists $\nu\in \mathcal{M}^{a_1}_{1,\theta_{a_1}}\cap\mathcal{M}^{a_2}_{1,\theta_{a_2}}$, then it is easy to see that
		\begin{equation*}
			\theta_{a_1}\geq \|\nu^{a_1}\|_1=\int_{\mathbb{R}^d}|x-a_1|\nu(\mathrm{d}x)\geq |a_1-a_2|-\int_{\mathbb{R}^d}|x-a_2|\nu(\mathrm{d}x)=|a_1-a_2|-\|\nu^{a_2}\|_1\geq |a_1-a_2|-\theta_{a_2},
		\end{equation*}
		which contradicts the assumption that $\theta_{a_1}+\theta_{a_2}<|a_1-a_2|$. The proof is then completed as there exists a periodic measure/asymptotic quasi-periodic in both $\overline{{\mathcal Q}_2}\cap  \mathcal{M}^{a_1}_{2,\theta_{a_1}}$ and $\overline{{\mathcal Q}_2}\cap  \mathcal{M}^{a_2}_{2,\theta_{a_2}}$ respectively using the same argument as that in the proof of Theorem \ref{Thm existence quasi-periodic measure}.
	\end{proof}

	We utilize Theorem \ref{Theorem nonuniqueness} to show that the following Curie-Weiss model possess at least two distinct entrance (invariant, periodic, asymptotic quasi-periodic) measures.

	\begin{example}\label{Example:2}
		We still consider the 1-dimensional Curie-Weiss model as in Example \ref{Example 1} with 
		\begin{equation}\label{Ex 2}
			W(x,y)=-k(x-y)^2 \text{ for some } k>0.
		\end{equation}
		In this case we know that
		\begin{equation*}
			b(t,x,\nu)=\beta\Big(x-x^3+f(t)-2k\int_{\mathbb{R}}(x-y)\nu(\mathrm{d}y)\Big).
		\end{equation*}
		Note when noise and time dependent term disappear in equation \eqref{0628}, then the dynamical system has three invariant point $1,0,-1$. Therefore if we would like to construct multiple entrance measures, we ought to think that these measures would be concentrated around $1,0,-1$. Note the definition of $\mathcal{M}^a_{1,\theta_a}$, thus we consider $\mathcal{M}^a_{1,\theta_a}$ for $a=1,0,-1$, respectively.

		By a simple calculation, for any $\beta>0,k>0,|\sigma|>0$ and $|f|_{\infty}:=\sup_{t\in\mathbb{R}}|f(t)|<\infty$,  we know that 
		\begin{equation*}
			\alpha^{a}_t\equiv -\beta(k+1), \  \beta^{a}_t\equiv \beta k, \ \gamma^{a}_t\equiv \beta(|f|_{\infty}^2+10), \ c^a_{\sigma}\equiv |\sigma|^2, \ \vartheta^{a}_t\equiv |f|_{\infty}^2+10+\frac{|\sigma|^2}{2\beta}
		\end{equation*}
		for all $a=1,0,-1$ satisfying assumptions in Theorem \ref{Theorem nonuniqueness}.

		(i) Case $a=1$. According to \eqref{0815-1} and by calculation, for any $t\in \mathbb{R}, y\in \mathbb{R}^d, \nu\in \mathcal{P}_1(\mathbb{R}^d)$, we have
		\begin{equation*}
			\begin{split}
				2yb^1(t,y,\nu)+|\sigma|^2&=2\beta\Bigl(y(y+1)-y(y+1)^3+yf(t)-2ky\int_{\mathbb{R}}(y-z)\nu(\mathrm{d}z)\Bigr)+|\sigma|^2\\
				&\leq -\Big[2\beta \bigl(|y|^4-3|y|^3+2(k+1)|y|^2-(|f|_{\infty}+2k\|\nu\|_1)|y|\bigr)-|\sigma|^2\Big].
			\end{split}
		\end{equation*}
		Thus we can choose
		\begin{equation}\label{0815-3}
			g_1(z,w)=2\beta\bigl(z^4-3z^3+2(k+1)z^2-(|f|_{\infty}+2kw)z\bigr)-|\sigma|^2
		\end{equation}
		such that \eqref{0815-2} holds true. Obviously, $g_1$ is continuous. Since $\beta>0,k>0$, then for any $z\geq 0$, $g_1(z,\cdot)$ is non-increasing. On the other hand,
		\begin{equation*}
			\frac{\partial^2 g_1(z,w)}{\partial z^2}=4\beta \big(6z^2-9z+2(k+1)\big).
		\end{equation*}
		Then for any $w\geq 0$, $g_1(\cdot,w)$ is convex if and only if $k\geq \frac{11}{16}$.

		Note that for $1\leq \theta_{1}\leq 2$, we have
		\begin{equation*}
			g_{1}(\theta_{1},\theta_{1})=2\beta\theta_{1}\bigl(\theta_{1}^3-3\theta_{1}^2+2\theta_{1}-|f|_{\infty}\bigr)-|\sigma|^2\leq -|\sigma|^2<0.
		\end{equation*}
		Then \eqref{0814-2} fails to hold. Now for any $0<\theta_1<1$ or $\theta_1>2$, since $g_1(z,\theta_1)$ is a convex nontrival polynomial, then \eqref{0814-2} holds if and only if
		\begin{equation*}
			g_1(\theta_1,\theta_1)\geq 0 \ \text{ and } \ \frac{\partial g_1}{\partial z}(\theta_1,\theta_1)\geq 0,
		\end{equation*}
		which gives
		\begin{equation*}
			|f|_{\infty}< \theta_1^3-3\theta_1^2+2\theta_1 \ \text{ and } \ \beta\geq \frac{|\sigma|^2}{2\theta_1(\theta_1^3-3\theta_1^2+2\theta_1-|f|_{\infty})}.
		\end{equation*}
		Hence, for any fixed $\theta_1\in (0,1)\cup (2,\infty)$ and $g_1$ defined as in \eqref{0815-3}, all assumptions in Theorem \ref{Theorem nonuniqueness} hold for this example in the case $a=1$ if and only if
		\begin{equation}\label{case 1 in Ex}
			k\geq \frac{11}{16}, \ \ |f|_{\infty}< \theta_1^3-3\theta_1^2+2\theta_1, \ \text{ and } \ \beta\geq \frac{|\sigma|^2}{2\theta_1(\theta_1^3-3\theta_1^2+2\theta_1-|f|_{\infty})}.
		\end{equation}
		Therefore, under \eqref{case 1 in Ex}, equation \eqref{0628} with \eqref{Ex 2} has an entrance measure in $\mathcal{M}^1_{1,\theta_1}$.

		(ii) Case $a=-1$. We can also choose $g_{-1}$ the same as $g_1$:
		\begin{equation}\label{0815-4}
			g_{-1}(z,w)=2\beta\bigl(z^4-3z^3+2(k+1)z^2-(|f|_{\infty}+2kw)z\bigr)-|\sigma|^2.
		\end{equation}
		Then arguing as the case (i), for any fixed $\theta_{-1}\in (0,1)\cup (2,\infty)$ and $g_{-1}$ defined as in \eqref{0815-4}, all assumptions in Theorem \ref{Theorem nonuniqueness} hold in the case $a=-1$ if and only if
		\begin{equation}\label{case 2 in Ex}
			k\geq \frac{11}{16}, \ \ |f|_{\infty}< \theta_{-1}^3-3\theta_{-1}^2+2\theta_{-1}, \ \text{ and } \ \beta\geq \frac{|\sigma|^2}{2\theta_{-1}(\theta_{-1}^3-3\theta_{-1}^2+2\theta_{-1}-|f|_{\infty})}.
		\end{equation}
		Therefore, under \eqref{case 2 in Ex}, equation \eqref{0628} with \eqref{Ex 2} has an entrance measure in $\mathcal{M}^{-1}_{1,\theta_{-1}}$.

		(iii) Case $a=0$. We have
		\begin{equation*}
			g_{0}(z,w)=2\beta\bigl(z^4+(2k-1)z^2-(|f|_{\infty}+2kw)z\bigr)-|\sigma|^2.
		\end{equation*}
		Similarly, \eqref{0814-2} fails to hold when $0<\theta_0\leq 1$. Moreover, for any $\theta_0>1$ and $g_0$ defined as above, all assumptions in Theorem \ref{Theorem nonuniqueness} hold in the case $a=0$ if and only if
		\begin{equation}\label{case 3 in Ex}
			k\geq \frac{1}{2}, \ \ |f|_{\infty}< \theta_{0}^3-\theta_{0}, \ \text{ and } \ \beta\geq \frac{|\sigma|^2}{2\theta_{0}(\theta_{0}^3-\theta_{0}-|f|_{\infty})}.
		\end{equation}
		Therefore, under \eqref{case 3 in Ex}, equation \eqref{0628} with \eqref{Ex 2} has an entrance measure in $\mathcal{M}^{0}_{1,\theta_{0}}$. But in this case, both $\theta_0+\theta_1<1$ and $\theta_0+\theta_{-1}<1$ are not satisfied, so we cannot claim  $\mathcal{M}^{0}_{1,\theta_{0}}\cap\mathcal{M}^{1}_{1,\theta_{1}}=\emptyset$ and $\mathcal{M}^{0}_{1,\theta_{0}}\cap\mathcal{M}^{-1}_{1,\theta_{-1}}=\emptyset$.

		However, note that
		\begin{equation*}
			\max_{0\leq \theta\leq 1}(\theta^3-3\theta^2+2\theta)=(\theta^3-3\theta^2+2\theta)|_{\theta=\frac{3-\sqrt{3}}{3}}=\frac{2\sqrt{3}}{9}.
		\end{equation*}
		Hence \eqref{case 1 in Ex} and \eqref{case 2 in Ex} hold true if we take 
		\begin{equation*}
			k\geq \frac{11}{16}, \ \theta_1=\theta_{-1}=\frac{3-\sqrt{3}}{3}, \ |f|_{\infty}\leq \frac{\sqrt{3}}{9}, \ \text{ and } \ \beta\geq \frac{27(\sqrt{3}+1)|\sigma|^2}{12}.
		\end{equation*}
		Since $\theta_1+\theta_{-1}<2$, equation \eqref{0628} with \eqref{Ex 2} has two distinct entrance measures in $\mathcal{M}^{1}_{1,\frac{3-\sqrt{3}}{3}}$ and $\mathcal{M}^{-1}_{1,\frac{3-\sqrt{3}}{3}}$, respectively. Moreover, equation \eqref{0628} with \eqref{Ex 2} has at least
		\begin{itemize}
			\item two distinct invariant measures if $f$ is a constant;
            \item two distinct periodic measures if $f$ is a periodic continuous function;
			\item two distinct asymptotic quasi-periodic measures if $f$ is a quasi-periodic function.
		\end{itemize}
	\end{example}

    \begin{remark}
    In Example \ref{Example:2} for the time-homogeneous case (e.g. $f$ is a constant), our result indicates that the Curie-Weiss model has at least two different invariant measures when $|f|<\frac{\sqrt{3}}{9}$ and the parameters satisfy
       \begin{equation*}
			k\geq \frac{11}{16} \ \text{ and } \ \beta\geq \frac{27(\sqrt{3}+1)|\sigma|^2}{12}.
		\end{equation*}
        Liu-Wu-Zhang \cite{Liu-Wu-Zhang2021} established the uniqueness of invariant measures when the parameters satisfy \eqref{eq:parameter-uniqueness}. Thus, there is still a gap between these two criteria.
        However, for the standard time-homogeneous Curie-Weiss model (i.e., $f \equiv 0$), for any $\beta,k > 0$, Dowson \cite{Dawson1983}, Tugaut \cite{Tugaut2014} showed that there exists a critical value $\sigma_c > 0$ such that
        for any fixed $\sigma^2 \geq \sigma_c$, there is a unique invariant measure, while for any $\sigma^2 < \sigma_c$, there exist three invariant measures. It is worth noting that their result relies on the explicit formula of the invariant measure. But such a formula has not been found in the time in-homogeneous case, thus, it has not been possible for us to close the gap in such a setting. 
    \end{remark}

    \begin{remark}
        Aside from the time-inhomogeneous setting, the conclusions of our work (Theorem \ref{Theorem nonuniqueness} and Example \ref{Example:2}) and that in \cite{Zhang2023} are similar. However, the methods used differ substantially. Specifically, our approach relies on the Kakutani fixed point theorem, whereas \cite{Zhang2023} employed the Schauder fixed point theorem. In \cite{Zhang2023}, the authors are able to show that the map $\mathcal{T}:\mathcal{P}_2(\mathbb{R}^d)\to \mathcal{P}_2(\mathbb{R}^d)$ defined in \cite[proof of Theorem 2.2, Page 6]{Zhang2023} is continuous, which is sufficient for applying Schauder's theorem. In contrast, in our setting, we are unable to establish the continuity of the map $\Phi: C(\mathbb{R};\mathcal{P}_2(\mathbb{R}^d)) \to C(\mathbb{R};\mathcal{P}_2(\mathbb{R}^d))$ with respect to the metric $d_2$ defined as in \eqref{eq:distance-d_2}. 

Moreover, while the compactness of the embedding $\mathcal{P}_2(\mathbb{R}^d) \to \mathcal{P}_1(\mathbb{R}^d)$ is relatively straightforward, proving the pre-compactness of the set $\widetilde{\mathcal{M}} \subset C(\mathbb{R};\mathcal{P}_2(\mathbb{R}^d))$ is substantially more challenging in our framework. We address these technical difficulties in detail in Lemma \ref{Lemma of compact set}.
    \end{remark}

\begin{appendices}
    \section{}
    Let $X_n=\{X_n(t), \ t\in [0,T]\}_{n\geq 1}$ be a sequence of $d$-dimensional continuous processes satisfying the following two conditions:
	\begin{equation}\label{Condition 1}
		\lim_{N\to \infty}\sup_{n\geq 1}\mathbf{P}\{|X_n(0)|>N\}=0;
	\end{equation}
	\begin{equation}\label{Condition 2}
	  \lim_{h\to 0}\sup_{n\geq 1}\mathbf{P}\Big\{\sup_{s,t\in [0,T], |t-s|\leq h}|X_n(t)-X_n(s)|>\epsilon\Big\}=0, \ \text{ for any } \epsilon>0.
	\end{equation}
	The following two theorems can be found as Theorem I-4.2 and Theorem I-4.3 in \cite{ikeda1989}.
	\begin{theorem}\label{Thm I-4.2}
	  Let $X_n=\{X_n(t), \ t\in [0,T]\}, n\geq 1$ be a sequence of $d$-dimensional continuous processes satisfying \eqref{Condition 1} and \eqref{Condition 2}. Then there exist a subsequence $\{n_k\}_{k\geq 1}$, a probability space $(\tilde{\Omega},\tilde{\mathcal{F}},\tilde{\mathbf{P}})$ and $d$-dimensional continuous processes $\{\tilde{X}_{n_k}\}_{k\geq 1}$ and $\tilde{X}$ on it such that
		\begin{itemize}
			\item [\emph{(i)}] $\mathscr{L}(\tilde{X}_{n_k})=\mathscr{L}(X_{n_k}), \ k\geq 1$;
			\item [\emph{(ii)}] the sequence $\{\tilde{X}_{n_k}, k\geq 1\}$ uniformly converges to $\tilde{X}$, $\tilde{\mathbf{P}}$-almost surely, i.e.,
			\begin{equation*}
			  \tilde{\mathbf{P}}\Big\{\lim_{k\to \infty}\sup_{t\in [0,T]}|\tilde{X}_{n_k}(t)-\tilde{X}(t)|=0\Big\}=1.
			\end{equation*}
		\end{itemize}
	\end{theorem}
	\begin{theorem}\label{Thm I-4.3}
	  Let $X_n=\{X_n(t), \ t\in [0,T]\}_{n\geq 1}$ be a sequence of $d$-dimensional continuous processes satisfying the following two conditions:
	  \begin{itemize}
		  \item [(1)] there exists a positive constant $\gamma$ such that $\sup_{n\geq 1}\mathbf{E}[|X_n(0)|^{\gamma}]<\infty$;
		  \item [(2)] there exist positive constants $\alpha,\beta,M$ such that
		  \begin{equation*}
			  \mathbf{E}[|X_n(t)-X_n(s)|^{\alpha}]\leq M|t-s|^{1+\beta}, \ \text{ for all } n\geq 1, \ t,s\in [0,T].
		  \end{equation*}
	   \end{itemize}
	   Then $\{X_n\}_{n\geq 1}$ satisfies the conditions \eqref{Condition 1} and \eqref{Condition 2}.
	\end{theorem}

     The following theorem is basic and can be found in Skorohod \cite{skorohod1965}.

	\begin{theorem}[\cite{skorohod1965}]\label{Theorem of Skorohod}
		Let $\{W_n(t), t\in [0,T]\}_{n\geq 1}$ be a sequence of $\mathcal{F}_t^n$-Brownian motions and $W(t), t\in [0,T]$ be a $\mathcal{F}_t$-Brownian motion such that $W_n$ converges to $W$ uniformly on $[0,T]$, $\mathbf{P}$-a.s.. Assume that $\mathcal{F}_t^n$-adapted measurable processes $F_n(t), t\in [0,T]$ and $\mathcal{F}_t$-adapted measurable process $F(t), t\in [0,T]$ are such that
		\begin{equation*}
			\sup_{n\geq 1}\int_{0}^{T}\mathbf{E}[|F_n(t)|^2]\mathrm{d}t<\infty, \ \text{ and } \ \lim_{n\to \infty}\int_{0}^{T}\mathbf{E}[|F_n(t)-F(t)|^2]\mathrm{d}t=0.
		\end{equation*}
		Then
		\begin{equation*}
			\lim_{n\to \infty}\mathbf{E}\biggl[\bigg|\int_{0}^{T}F_n(t)\mathrm{d}W_n(t)-\int_{0}^{T}F(t)\mathrm{d}W(t)\bigg|^2\biggr]=0.
		\end{equation*}
	\end{theorem}

	The following Kakutani's fixed point theorem (see e.g. \cite[Theorem 8.6]{Granas-Dugundji2003}) provides a foundational framework for proving the existence of fixed points in complex systems.

	\begin{theorem}\label{Lemma Kakutani's fixed point theorem} 
		Let $S$ be a nonempty, compact and convex subset of a locally convex Hausdorff space. Let $F: S\to 2^S$ be a set-valued function on $S$ satisfying the following conditions:
		\begin{itemize}
			\item [(i)] $F(x)$ is nonempty and convex for all $x\in S$;
			\item [(ii)] The graph $\{(x,y)\in S\times S: y\in F(x)\}$ is a closed subset of $S\times S$.
		\end{itemize}
		Then $F$ has a fixed point $x^*\in S$, i.e., $x^*\in F(x^*)$.
	\end{theorem}
\end{appendices}

\appendix
\renewcommand\thesection{\normalsize Acknowledgements}
\section{}
We are grateful to anonymous referees for their useful and constructive comments which have helped significantly improve the quality of the article.
We acknowledge the financial supports from an EPSRC grant ref EP/S005293/2, and a Royal Society Newton Fund grant (ref. NIF\textbackslash R1\textbackslash 221003) and the National Natural Science Foundation of China under Grant 12501184.

\addtolength{\itemsep}{-1.5 em} 
\setlength{\itemsep}{-3pt}
\footnotesize

\phantomsection
	\addcontentsline{toc}{section}{References}
	\tolerance=500
\bibliographystyle{siam}
\bibliography{DDSDE}

\begin{thebibliography}{10}

\bibitem{Ahmed-Ding1993}
{\sc N.~U. Ahmed and X.~Ding}, {\em On invariant measures of nonlinear {M}arkov processes}, J. Appl. Math. Stochastic Anal., 6 (1993), pp.~385--406.

\bibitem{Aliprantis-Border2006}
{\sc C.~D. Aliprantis and K.~C. Border}, {\em Infinite dimensional analysis}, Springer, Berlin, third~ed., 2006.
\newblock A hitchhiker's guide.

\bibitem{Bao-Scheutzow-Yuan2022}
{\sc J.~Bao, M.~Scheutzow, and C.~Yuan}, {\em Existence of invariant probability measures for functional {M}c{K}ean-{V}lasov {SDE}s}, Electron. J. Probab., 27 (2022), pp.~1--14.

\bibitem{Butkovsky2014}
{\sc O.~A. Butkovsky}, {\em On ergodic properties of nonlinear {M}arkov chains and stochastic {M}c{K}ean-{V}lasov equations}, Theory Probab. Appl., 58 (2014), pp.~661--674.

\bibitem{Carrillo-Gvalani-Pavliotis-Schlichting2020}
{\sc J.~A. Carrillo, R.~S. Gvalani, G.~A. Pavliotis, and A.~Schlichting}, {\em Long-time behaviour and phase transitions for the {M}c{K}ean-{V}lasov equation on the torus}, Arch. Ration. Mech. Anal., 235 (2020), pp.~635--690.

\bibitem{Carrillo-McCann-Villani2003}
{\sc J.~A. Carrillo, R.~J. McCann, and C.~Villani}, {\em Kinetic equilibration rates for granular media and related equations: entropy dissipation and mass transportation estimates}, Rev. Mat. Iberoamericana, 19 (2003), pp.~971--1018.

\bibitem{Carrillo-McCann-Villani2006}
\leavevmode\vrule height 2pt depth -1.6pt width 23pt, {\em Contractions in the 2-{W}asserstein length space and thermalization of granular media}, Arch. Ration. Mech. Anal., 179 (2006), pp.~217--263.

\bibitem{Da-Prato-Zabczyk1996}
{\sc G.~Da~Prato and J.~Zabczyk}, {\em Ergodicity for infinite-dimensional systems}, vol.~229 of London Mathematical Society Lecture Note Series, Cambridge University Press, Cambridge, 1996.

\bibitem{Dawson1983}
{\sc D.~A. Dawson}, {\em Critical dynamics and fluctuations for a mean-field model of cooperative behavior}, J. Statist. Phys., 31 (1983), pp.~29--85.

\bibitem{Delgadino-Gvalani-Pavliotis-Smith2023}
{\sc M.~G. Delgadino, R.~S. Gvalani, G.~A. Pavliotis, and S.~A. Smith}, {\em Phase transitions, logarithmic {S}obolev inequalities, and uniform-in-time propagation of chaos for weakly interacting diffusions}, Comm. Math. Phys., 401 (2023), pp.~275--323.

\bibitem{Dudley2002}
{\sc R.~M. Dudley}, {\em Real analysis and probability}, vol.~74 of Cambridge Studies in Advanced Mathematics, Cambridge University Press, Cambridge, 2002.
\newblock Revised reprint of the 1989 original.

\bibitem{Feng-Qu-Zhao2021}
{\sc C.~Feng, B.~Qu, and H.~Zhao}, {\em Random quasi-periodic paths and quasi-periodic measures of stochastic differential equations}, J. Differential Equations, 286 (2021), pp.~119--163.

\bibitem{Feng-Qu-Zhao2023}
\leavevmode\vrule height 2pt depth -1.6pt width 23pt, {\em Entrance measures for semigroups of time-inhomogeneous sdes: possibly degenerate and expanding}, arXiv: 2307.07891,  (2023).

\bibitem{Feng-Zhao2020}
{\sc C.~Feng and H.~Zhao}, {\em Random periodic processes, periodic measures and ergodicity}, J. Differential Equations, 269 (2020), pp.~7382--7428.

\bibitem{Feng-Zhao-Zhong2023}
{\sc C.~Feng, H.~Zhao, and J.~Zhong}, {\em Existence of geometric ergodic periodic measures of stochastic differential equations}, J. Differential Equations, 359 (2023), pp.~67--106.

\bibitem{Granas-Dugundji2003}
{\sc A.~Granas and J.~Dugundji}, {\em Fixed point theory}, Springer Monographs in Mathematics, Springer-Verlag, New York, 2003.

\bibitem{Guillin-Liu-Wu-Zhang2021}
{\sc A.~Guillin, W.~Liu, L.~Wu, and C.~Zhang}, {\em The kinetic {F}okker-{P}lanck equation with mean field interaction}, J. Math. Pures Appl. (9), 150 (2021), pp.~1--23.

\bibitem{Guillin-Liu-Wu-Zhang2022}
\leavevmode\vrule height 2pt depth -1.6pt width 23pt, {\em Uniform {P}oincar\'{e} and logarithmic {S}obolev inequalities for mean field particle systems}, Ann. Appl. Probab., 32 (2022), pp.~1590--1614.

\bibitem{huang2016ergodic}
{\sc W.~Huang, Z.~Lian, and K.~Lu}, {\em Dynamical complexity of {A}nosov systems driven by a quasi-periodic force}, Sci. China Math., 68 (2025), pp.~89--136.

\bibitem{ikeda1989}
{\sc N.~Ikeda and S.~Watanabe}, {\em Stochastic Differential Equations and Diffusion Processes}, vol.~24 of North-{{Holland Mathematical Library}}, {North-Holland Publishing Co., Amsterdam; Kodansha, Ltd., Tokyo}, second~ed., 1989.

\bibitem{Kac1956}
{\sc M.~Kac}, {\em Foundations of kinetic theory}, in Proceedings of the {T}hird {B}erkeley {S}ymposium on {M}athematical {S}tatistics and {P}robability, 1954--1955, vol. {III}, University of California Press, Berkeley-Los Angeles, Calif., 1956, pp.~171--197.

\bibitem{Lanford1975}
{\sc O.~E. Lanford, III}, {\em Time evolution of large classical systems}, in Dynamical systems, theory and applications ({R}encontres, {B}attelle {R}es. {I}nst., {S}eattle, {W}ash., 1974), Lecture Notes in Phys., Vol. 38, Springer, Berlin, 1975, pp.~1--111.

\bibitem{Liu-Lu2025}
{\sc R.~Liu and K.~Lu}, {\em Exponential mixing and limit theorems of quasi-periodically forced 2{D} stochastic {N}avier-{S}tokes equations in the hypoelliptic setting}, Comm. Math. Phys., 406 (2025), pp.~Paper No. 55, 46.

\bibitem{Liu-Wu-Zhang2021}
{\sc W.~Liu, L.~Wu, and C.~Zhang}, {\em Long-time behaviors of mean-field interacting particle systems related to {M}c{K}ean-{V}lasov equations}, Comm. Math. Phys., 387 (2021), pp.~179--214.

\bibitem{McKean1967}
{\sc H.~P. McKean, Jr.}, {\em Propagation of chaos for a class of non-linear parabolic equations}, in Stochastic {D}ifferential {E}quations ({L}ecture {S}eries in {D}ifferential {E}quations, {S}ession 7, {C}atholic {U}niv., 1967), Air Force Office Sci. Res., Arlington, Va., 1967, pp.~41--57.

\bibitem{Muzychka-Vaninsky2011}
{\sc S.~A. Muzychka and K.~L. Vaninsky}, {\em A class of nonlinear random walks related to the {O}rnstein-{U}hlenbeck process}, Markov Process. Related Fields, 17 (2011), pp.~277--304.

\bibitem{rachev1984}
{\sc S.~T. Rachev}, {\em The {{Monge-Kantorovich}} problem on mass transfer and its applications in stochastics}, Akademiya Nauk SSSR. Teoriya Veroyatnoste\textbackslash u \i{} i ee Primeneniya, 29 (1984), pp.~625--653.

\bibitem{Ren-Wang2021}
{\sc P.~Ren and F.-Y. Wang}, {\em Exponential convergence in entropy and {W}asserstein for {M}c{K}ean-{V}lasov {SDE}s}, Nonlinear Anal., 206 (2021), pp.~Paper No. 112259, 20.

\bibitem{skorohod1965}
{\sc A.~Skorohod}, {\em Studies in the Theory of Random Processes}, {Translated from the Russian by Scripta Technica, Inc. Addison-Wesley Publishing Co., Inc., Reading, Mass.}, 1965.

\bibitem{Sznitman1991}
{\sc A.-S. Sznitman}, {\em Topics in propagation of chaos}, in \'{E}cole d'\'{E}t\'{e} de {P}robabilit\'{e}s de {S}aint-{F}lour {XIX}---1989, vol.~1464 of Lecture Notes in Math., Springer, Berlin, 1991, pp.~165--251.

\bibitem{Tugaut2014}
{\sc J.~Tugaut}, {\em Phase transitions of {M}c{K}ean-{V}lasov processes in double-wells landscape}, Stochastics, 86 (2014), pp.~257--284.

\bibitem{Uchiyama1988}
{\sc K.~Uchiyama}, {\em Derivation of the {B}oltzmann equation from particle dynamics}, Hiroshima Math. J., 18 (1988), pp.~245--297.

\bibitem{Villani2009}
{\sc C.~Villani}, {\em Optimal transport}, vol.~338 of Grundlehren der mathematischen Wissenschaften [Fundamental Principles of Mathematical Sciences], Springer-Verlag, Berlin, 2009.
\newblock Old and new.

\bibitem{wang2018}
{\sc F.-Y. Wang}, {\em Distribution dependent {{SDEs}} for {{Landau}} type equations}, Stochastic Processes and their Applications, 128 (2018), pp.~595--621.

\bibitem{Zhang2023}
{\sc S.-Q. Zhang}, {\em Existence and non-uniqueness of stationary distributions for distribution dependent {SDE}s}, Electron. J. Probab., 28 (2023), pp.~1--34.

\end{thebibliography}
\end{document}